\documentclass[english,11pt]{smfart}

\usepackage{etex}

\usepackage{amsbsy}
\usepackage{amsmath,amsfonts,amssymb,amsthm,mathrsfs,mathtools}

\usepackage{bm}

\usepackage[a4paper,vmargin={3.5cm,3.5cm},hmargin={2.5cm,2.5cm}]{geometry}
\usepackage[font=sf, labelfont={sf,bf}, margin=1cm]{caption}
\usepackage{graphicx}
\usepackage{epsfig}
\usepackage{latexsym}
\usepackage{xcolor}
\usepackage{ae,aecompl}
\usepackage{soul,framed}
\usepackage{comment}

\usepackage{xcolor}
\usepackage[pdfpagemode=UseNone,bookmarksopen=false,colorlinks=true,urlcolor=blue,citecolor=blue,citebordercolor=blue,linkcolor=blue]{hyperref}
\usepackage{smfhyperref}
\usepackage[capitalize]{cleveref}

\usepackage{pstricks}
\usepackage{enumerate}
\usepackage{tikz,animate,media9}						
\usepackage{todonotes}
\usepackage{pifont}
\usepackage{bm,marvosym}
\usepackage{algorithm}
\usepackage{algorithmic}

\usepackage{bbm}

\usepackage{tcolorbox}

\usepackage{natbib}  

\definecolor{aleacolor}{rgb}{0.16,0.59,0.78}

\renewcommand{\cite}{\citet}

\hypersetup{
	breaklinks,
	colorlinks=true,
	linkcolor=aleacolor,
	urlcolor=aleacolor,
	citecolor=aleacolor}

\newcommand{\ndN}{\mathbb{N}}
\newcommand{\ndZ}{\mathbb{Z}}

\newcommand{\ndR}{\mathbb{R}}

\renewcommand{\Pr}[1]{\mathbb{P}(#1)}

\newcommand{\Prb}[1]{\mathbb{P}\left(#1\right)}

\newcommand{\Ex}[1]{\mathbb{E}[#1]}

\newcommand{\Exb}[1]{\mathbb{E}\left[#1\right]}

\newcommand{\one}{{\mathbbm{1}}}

\newcommand{\convdis}{\,{\buildrel d \over \longrightarrow}\,}
\newcommand{\convd}{\,{\buildrel d \over \longrightarrow}\,}

\newcommand{\convp}{\,{\buildrel p \over \longrightarrow}\,}

\newcommand{\eqdist}{\,{\buildrel d \over =}\,}

\newcommand{\atv}{\,{\buildrel d \over \approx}\,}

\newcommand{\w}{\mathbf{w}}

\newcommand{\cA}{\mathcal{A}}
\newcommand{\cB}{\mathcal{B}}
\newcommand{\cC}{\mathcal{C}}
\newcommand{\cD}{\mathcal{D}}
\newcommand{\cE}{\mathcal{E}}
\newcommand{\cF}{\mathcal{F}}
\newcommand{\cG}{\mathcal{G}}
\newcommand{\cH}{\mathcal{H}}
\newcommand{\cI}{\mathcal{I}}
\newcommand{\cJ}{\mathcal{J}}
\newcommand{\cK}{\mathcal{K}}
\newcommand{\cL}{\mathcal{L}}
\newcommand{\cM}{\mathcal{M}}
\newcommand{\cN}{\mathcal{N}}
\newcommand{\cO}{\mathcal{O}}
\newcommand{\cP}{\mathcal{P}}

\newcommand{\cR}{\mathcal{R}}
\newcommand{\cS}{\mathcal{S}}
\newcommand{\cT}{\mathcal{T}}

\newcommand{\cV}{\mathcal{V}}
\newcommand{\cW}{\mathcal{W}}
\newcommand{\cX}{\mathcal{X}}

\newcommand{\cZ}{\mathcal{Z}}

\newcommand{\Set}{\textsc{SET}}
\newcommand{\Seq}{\textsc{SEQ}}

\newcommand{\mB}{\mathsf{B}}

\newcommand{\mD}{\mathsf{D}}

\newcommand{\mF}{\mathsf{F}}
\newcommand{\mG}{\mathsf{G}}

\newcommand{\mK}{\mathsf{K}}

\newcommand{\mM}{\mathsf{M}}
\newcommand{\mN}{\mathsf{N}}
\newcommand{\mO}{\mathsf{O}}
\newcommand{\mP}{\mathsf{P}}

\newcommand{\mR}{\mathsf{R}}
\newcommand{\mS}{\mathsf{S}}
\newcommand{\mT}{\mathsf{T}}

\newcommand{\mV}{\mathsf{V}}

\newcommand{\mX}{\mathsf{X}}

\newcommand{\mfB}{\mathfrak{B}}

\newcommand{\mfL}{\mathfrak{L}}

\newcommand{\emb}{\mathrm{emb}}
\newcommand{\co}{\mathrm{c}}
\newcommand{\ve}{\mathrm{v}}
\newcommand{\ed}{\mathrm{e}}

\newcommand{\Inv}[1]{\mathrm{Inv}(#1)}

\newtheorem{theorem}{Theorem}[section]

\newtheorem{corollary}[theorem]{Corollary}
\newtheorem{proposition}[theorem]{Proposition}
\newtheorem{lemma}[theorem]{Lemma}

\newtheorem{remark}[theorem]{Remark}

\numberwithin{equation}{section}


\title{\textbf{Local convergence of random planar graphs}}
\date{}

\author{Benedikt Stufler}
\address[Benedikt Stufler]{Institute of Mathematics, University of Zurich}
\email{benedikt.stufler@math.uzh.ch}

\begin{document}

	\maketitle

\vspace {-0.5cm}

\begin{abstract}
	The present work describes the asymptotic local  shape of a graph drawn uniformly at random from all connected simple planar graphs with $n$ labelled vertices. We establish a novel uniform infinite planar graph (UIPG) as  quenched limit in the local topology as $n$ tends to infinity. We also establish such limits for random $2$-connected planar graphs and maps as their number of edges tends to infinity. Our approach encompasses a new probabilistic view on the Tutte decomposition. This allows us to follow the path along the decomposition of connectivity from planar maps to planar graphs in a uniformed way, basing each step on condensation phenomena for random walks under subexponentiality and Gibbs partitions. Using large deviation results, we recover the asymptotic formula by Gim\'enez and Noy (2009) for the number of planar graphs.
\end{abstract}


\section{Introduction}

\subsection{Main results}


A graph is planar if it may be drawn in the plane such that edges intersect only at endpoints. The reader may consult the book by \cite{MR1844449} for details of graph embeddings on surfaces. We are interested in properties of the graph $\mP_n$ selected uniformly at random among all simple connected planar graphs with vertices labelled from $1$ to $n$. Here the term \emph{simple} refers to the absence of loops and multiple edges.

Properties of the random graph $\mP_n$ have received considerable attention in recent literature \cite{MR3273487,MR3318042,MR2476775,MR2802191,MR2858393,MR1946145}. We refer the reader to the comprehensive survey by \cite{noysurvey} for a detailed account. Our main theorem shows that $\mP_n$ admits a local limit. 




\begin{theorem}
	\label{te:mainfinal}
	The uniform $n$-vertex connected planar graph $\mP_n$ rooted at a uniformly selected vertex~$v_n$ admits a distributional limit~$\hat{\mP}$ in the local topology. We call $\hat{\mP}$ the uniform infinite planar graph (UIPG). The regular conditional law $\Pr{ (\mP_n, v_n) \mid \mP_n}$ satisfies
	\begin{align}
	\label{eq:quenched}
	\Pr{ (\mP_n, v_n) \mid \mP_n} \convp \mfL(\hat{\mP}).
	\end{align}
\end{theorem}
The \emph{quenched} convergence in \eqref{eq:quenched} implies the \emph{annealed} convergence $\mP_n \convdis \hat{\mP}$.  See Section~\ref{sec:localconv} for details on these forms of convergence. The root degree of $\hat{\mP}$ follows the asymptotic degree distribution of $\mP_n$  established by \cite{MR2802191} and  \cite{MR2858393}. We also prove a version of this theorem (with a different limit object) where $v_n$ is chosen according to the stationary distribution instead. By a celebrated result of \cite[Thm. 1.1]{MR3010812}, this implies that the limit $\hat{\mP}$ is almost surely recurrent.  Milestones in the proof of our main result include local limits for  $2$-connected planar structures: 

\begin{theorem}
Let $v_n^\cB$ denote a uniformly selected vertex of the uniform $2$-connected planar graph $\mB_n$ with $n$ edges. There is a uniform infinite planar graph $\hat{\mB}$ with
\begin{align}
	\Pr{ (\mB_n, v_n^\cB) \mid \mB_n} \convp \mfL(\hat{\mB}).
\end{align}
We call $\hat{\mB}$ the uniform infinite $2$-connected planar graph (UI2PG).
\end{theorem}
In fact, we prove a more general vertex-weighted version, see Theorem~\ref{te:2conlim}. There is a natural coupling where $\hat{\mP}$ is obtained from $\hat{\mB}$ by attaching i.i.d. Boltzmann distributed connected vertex-marked planar graphs at the non-root vertices of $\hat{\mB}$, and a Boltzmann distributed doubly vertex-marked connected planar graph at the root of $\hat{\mB}$. (See Section~\ref{sec:subseq} for the definition of \emph{the} Boltzmann distribution of a class of structures.)
\begin{theorem}
Let $v_n^\cV$ denote a uniformly selected corner of the random non-separable planar map $\mV_n$ with $n$ edges. There is uniform infinite planar map $\hat{\mV}$ with
\begin{align}
\Pr{ (\mV_n, v_n^\cV) \mid \mV_n} \convp \mfL(\hat{\mV}).
\end{align}
We call $\hat{\mV}$ the uniform infinite $2$-connected planar map (UI2PM).
\end{theorem}
Again, we actually prove a more general version with vertex-weights, see  Theorem~\ref{te:map2}. The degree distribution of the non-separable case has been studied by~\cite{MR3071845}. The well-known uniform infinite planar map has received considerable attention in the literature, see \cite{MR3769811, MR3183575, MR3083919, MR3256879} (and also \cite{MR2013797,2005math.....12304K}). It may be obtained from $\hat{\mV}$ by attaching i.i.d. Boltzmann distributed planar maps at each non-root corner, and a Boltzmann distributed doubly corner rooted planar map at its root-corner.

The methods we develop in this paper yield a novel probabilistic view on the Tutte decomposition of these objects, see Sections~\ref{sec:graphs} and \ref{sec:mapdecomp}. We do not prove or build upon local convergence of uniform $3$-connected planar maps and graphs with $n$ edges. This highly relevant result was established by \cite{MR3254733} using a different approach.  As a further mayor application we recover a celebrated result in enumerative combinatorics by Gim\'enez and Noy:
\begin{theorem}[{\cite[Thm. 1]{MR2476775}}]
	\label{te:asymp}
The number $p_n$ of labelled simple planar graphs with $n$ vertices satisfies the asymptotic
\begin{align}
	\label{eq:asymp}
 	p_n \sim c_\cG \rho_{\cC}^{-n} n^{-7/2},
\end{align}
with the constants $c_{\cG}$ and $\rho_\cC$ admitting analytic expressions given in Equation~\eqref{eq:thefirst} and \eqref{eq:thesecond}.
\end{theorem}

See Section~\ref{sec:asympenum}, in particular Subsection~\ref{sec:subpla}, for a detailed proof. Gim\'enez and Noy obtained this result (resolving a history of rougher estimates by \cite{MR1946145,MR1973264,MR2047234,MR1393702}) by performing analytic integration and	man m la using analytic methods and results by \cite{MR1946145} on the number of $2$-connected graphs.  An approach employing ``combinatorial integration'' was  given by  \cite{MR2465772}. We reprove Equation~\eqref{eq:asymp} by different methods, without any integration step at all, deducing the asymptotic number of connected graphs from the number of $2$-connected graphs using results for the big-jump domain by \cite[Cor. 2.1]{MR2440928} and properties of subexponential probability distributions, see \cite{MR3097424}.  We emphasize that the approach by Gim\'enez and Noy additionally yields singular expansions for the involved generating series, and our proof does not. Hence the methods of~\cite{MR2476775} yield stronger results, and the methods employed here work under weaker assumptions.
  
Theorem~\ref{te:mainfinal} has applications concerning subgraph count asymptotics. By a  general result of \cite[Lem. 4.3]{2015arXiv150408103K} and using the asymptotic degree distribution of $\mP_n$ established by \cite{MR2802191}, it follows that:
\begin{corollary}
	For any finite connected graph $H$ the number $\emb(H, \mP_n)$ of occurrences of $H$ in $\mP_n$ as a subgraph satisfies
	\begin{align}
	\label{eq:subgraph}
	\frac{\emb(H, \mP_n)}{n} \convp \Ex{ \emb^\bullet(H^\bullet, \hat{\mP})}.
	\end{align}
	Here $H^\bullet$ denotes any fixed vertex rooted version of $H$, and $\emb^\bullet(H^\bullet, \hat{\mP})$ counts the number of root-preserving embeddings of $H^\bullet$ into $\hat{\mP}$.
\end{corollary} 
The study of the number of \emph{pendant} copies (or \emph{appearances}) of a fixed graph in $\mP_n$ was initiated by \cite{MR2117936}, and a normal central limit theorem was established by \cite[Sec. 4.3]{MR3068033}. The difficulty of studying $\emb(H, \mP_n)$ stems from the fact that it requires us to look inside the giant $2$-connected component of $\mP_n$, whereas pendant copies lie with high probability in the components attached to it. It is natural to conjecture convergence to a normal limit law for the fluctuations of $\emb(H, \mP_n)$ around $n\Ex{ \emb^\bullet(H^\bullet, \hat{\mP})}$ at the scale $\sqrt{n}$. Such a result has recently been established for the number of triangles in random cubic planar graphs by \cite{2018arXiv180206679N}, and the number of double triangles in random planar maps by \cite{inprocdoubletri}. In light of \cite{2015arXiv150408103K}, it would  be interesting to know whether such a central limit theorem may be established in a way that applies to general sequences of random graphs that are locally convergent in some strengthened sense.

\subsection{Summary of the main theorem's proof}

A quenched local limit for the random planar  map $\mM_n^t$ with $n$ edges and weight $t>0$ at vertices was established in \cite{2019arXivsubmit}.  We pass this convergence  down to a quenched limit for the non-separable core $\cV(\mM_n^t)$. For this, we employ a quenched version of an inductive argument discovered by \cite[Thm. 6.59]{2016arXiv161202580S}. The idea is that we have full information about the components attached to the core. The neighbourhood of a uniformly selected corner of $\mM_n^t$ gets patched together from a connected component containing it, a neighbourhood in the core, and neighbourhoods in components attached to the core neighbourhood. Expressing this yields a recursive equation, which by an inductive arguments allows us to prove convergence of $\cV(\mM_n^t)$. It is important to note that $\cV(\mM_n^t)$ has a random size, hence a priori properties of $\cV(\mM_n^t)$ do not carry over automatically to properties of the  uniform non-separable planar map $\mV_n^t$ with $n$ edges and weight $t$ at vertices.

We reduce the study of the non-separable core to the study of non-serial networks using a Gibbs partition result of \cite{doi:10.1002/rsa.20771}. We proceed to establish a novel fully recursive tree-like combinatorial encoding for non-serial networks in terms of a complex construct that we call $\bar{\cR}$-networks. This allows us to generate a non-serial network by starting with a random network $\bar{\mR}$ where one edge is marked as ``terminal''. The process proceeds recursively by substituting non-terminal edges by independent copies of $\bar{\mR}$ until only terminal edges are left. This allows us to apply recent results by \cite{2019arXiv190104603S} on subcritical branching processes. This yields a local limit theorem for the number of edges in a giant $\bar{\cR}$-core $\bar{\cR}(\mM_n^t)$ of the $\cV$-core $\cV(\mM_n^t)$, and implies that the network $\cV(\mM_n^t)$ has vanishing total variational distance from a network obtained from the $\bar{\cR}$-core $\bar{\cR}(\mM_n^t)$ by substituting all but a negligible number of edges by independent copies of the Boltzmann distributed $\bar{\cR}$-network $\bar{\mR}$. If we choose any fixed number of corners independently and uniformly at random, the corresponding $\bar{\cR}$-components containing them will follow size-biased distributions by the famous waiting time paradox. This gives us full information on the $\bar{\cR}$-components in the vicinity of these components. Since we substitute at edges, the resulting recursive equation for the probability of neighbourhoods in $\cV(\mM_n^t)$ to have a fixed shape do \emph{not} allow for the same inductive argument as before. The reason for this problem is that the event for a radius $r$ neighbourhood in $\cV(\mM_n^t)$ to have a fixed shape with $k$ edges may correspond to configurations with  more than $k$ edges in an $r$-neighbourhood in $\bar{\cR}(\mM_n^t)$, since components of edges between vertices of distance $r$ from the center do not always contribute to the $r$-neighbourhood in $\cV(\mM_n^t)$. We solve bis problem by abstraction, working with a more general convergence determining family of events (instead of shapes of neighbourhoods we look at shapes of what we call \emph{communities}, see the proof of Lemma~\ref{le:convRbar}) that allows the induction step to work.

Having arrived at a quenched local limit for the $\bar{\cR}$-core, we again apply the Gibbs partition result of \cite{doi:10.1002/rsa.20771} to deduce convergence of what we call the $\bar{\cO}$-core $\bar{\cO}(\mM_n^t)$ and is a randomly sized  map obtained from a $3$-connected planar map by blowing up edges into paths.  As we have a local limit theorem at hand for the number of edges of  $\bar{\cO}(\mM_n^t)$, we may transfer properties of $\bar{\cO}(\mM_n^t)$ to other randomly sized $\bar{\cO}$-networks satisfying a similar local limit theorem (but with possibly different constants). For example, we may define similarly the $\bar{\cO}$-core $\bar{\cO}(\mV_n^t)$ of $\mV_n^t$. The quenched convergence of $\bar{\cO}(\mM_n^t)$ transfers to quenched local convergence of $\bar{\cO}(\mV_n^t)$. The arguments we used to pass convergence from $\cV(\mM_n^t)$ to $\bar{\cO}(\mM_n^t)$ may be reversed to pass convergence from $\bar{\cO}(\mV_n^t)$ to $\mV_n^t$, yielding a quenched local limit for $\mV_n^t$.

Whitney's theorem ensures that we may group $\bar{\cO}$-maps into pairs such that each pair corresponds to a unique graph. We call such graphs $\cO$-graphs. $\bar{\cO}$-networks and $\cO$-graphs form the link between planar maps and planar graphs in our proof. We could have proceeded further to work with $3$-connected planar maps / graphs instead, but this detour is not necessary for our arguments. 

Networks that encode $2$-connected planar graphs differ from the networks that encode $2$-connected planar maps, since we do not allow multiple edges and do not care about the order in parallel compositions. But guided by the fully recursive decomposition for non-serial networks encoding maps, we  establish a somewhat more technical novel fully recursive decomposition for non-serial networks encoding graphs. The price we have to pay is that this decomposition is no longer isomorphism preserving. We stress this point in Subsubsection~\ref{eq:enrichedmaps}. This constitutes no issue  or downside  for the present work, which concerns itself exclusively with \emph{labelled} planar graphs. However, future applications to random \emph{unlabelled} planar graphs will require careful consideration and further study of how this step affects the symmetries. The decomposition allows us to argue analogously as for planar maps (again using repeated application of the authors' results on  Gibbs partitions \cite{doi:10.1002/rsa.20771} and subcritical branching processes \cite{2019arXiv190104603S}). Hence quenched local convergence of the random $2$-connected planar graph $\mB_n^t$ with $n$ edges and weight $t$ at vertices follows from the corresponding convergence of a giant $\cO$-core $\cO(\mB_n^t)$, obtained via a transfer from the core $\bar{\cO}(\mM_n^t)$. 

If we condition the $2$-connected core of the random planar graph $\mP_n$ to have a fixed number $m$ of edges, we do not obtain the uniform distribution on the $2$-connected planar graphs with $m$ edges. This effect does \emph{not} go away as $n$ tends to infinity. Instead, a result by \cite{MR3068033} shows that the $2$-connected core has a vanishing total variational distance from a mixture of $(\mB_n^t)_{n \ge 1}$ for the special case of vertex weight $t= \rho_\cB$, the radius of convergence of the generating series for $2$-connected planar graphs. This allows us to deduce quenched local convergence of the $2$-connected core $\cB(\mP_n)$ of the random connected planar graph $\mP_n$.  A quenched extension of a result by \cite[Thm. 6.39]{2016arXiv161202580S} then yields quenched convergence of $\mP_n$, completing the proof of Theorem~\ref{te:mainfinal}.

\subsection{Notation}

All unspecified limits are as $n \to \infty$. For any Polish space $S$ we let $\mathbb{M}_1(S)$ denote the collection of probability measures on the Borel $\sigma$-algebra $\mfB(S)$. We equip $\mathbb{M}_1(S)$ with the weak convergence topology, making it a Polish space itself. Given an $S$-valued random variable~$X$ that is defined on some probability space $(\Omega, \mathfrak{F}, \mathbb{P})$, we let $\mfL(X) \in \mathbb{M}_1(S)$ denote its law. If $Y: \Omega \to T$ is a random variable with values in some Polish space $T$, we let $\Pr{X \mid Y}$ denote the conditional law of $X$ given $Y$. 

For two sequences $(X_n)_n$ and $(Y_n)_n$ of random variables with values in $S$ we write $X_n \atv Y_n$ if their total variation distance $d_{\textsc{TV}}(X_n,Y_n) = \sup_{A \in \mfB(S)} | \Pr{X_n \in A} - \Pr{Y_n \in A}|$ tends to zero. We say an event holds with high probability if its probability tends to $1$ as $n$ becomes large. Convergence in probability and distribution are denoted by $\convp$ and $\convd$. 
For any sequence $a_n>0$ we let $o_p(a_n)$ denote an unspecified random variable $Z_n$ such that $Z_n / a_n \convp 0$. Likewise $O_p(a_n)$ is a random variable $Z_n$ such that $Z_n / a_n$ is stochastically bounded. 

\section{Local convergence}
\label{sec:localconv}

\subsection{The local topology}

The local topology quantifies how similar two rooted graphs are in the vicinity of the root vertices.  We briefly recall relevant notions and refer the reader to the elegant presentations by \cite{curienlecture} for details.

Let $\mathfrak{G}$ denote the collection of (representatives of) vertex-rooted locally finite connected simple graphs viewed up to root-preserving graph isomorphism. For any integer $k \ge 0$ we may consider the subset $\mathfrak{G}_k \subset \mathfrak{G}$ of graphs with radius at most $k$, equipped with the discrete topology. The projection $U_k: \mathfrak{G} \to \mathfrak{G}_k$ maps a rooted graph $G^\bullet$ to the $k$-neighbourhood  $U_k(G^\bullet)$  of its root vertex. The local topology on $\mathfrak{G}$ is the coarsest topology that makes these projections continuous. This projective limit topology is metrizable by
\[
	d_{\mathrm{loc}}(G^\bullet, H^\bullet) = \frac{1}{1 + \sup\{k \ge 0 \mid U_k(G^\bullet) = U_k(H^\bullet)\}}, \qquad G^\bullet, H^\bullet \in \mathfrak{G},
\]
making $(\mathfrak{G}, d_{\mathrm{loc}})$ a Polish space. Analogously, the collection $\mathfrak{M}$ of corner-rooted locally finite planar maps may be endowed with a local metric. We will always implicitly refer to convergence in $\mathfrak{M}$ when talking about local convergence of random maps, so that the limit preserves the embedding. Distributional convergence of a sequence $(\mG_n^\bullet)_{n \ge 0}$ of random elements of $\mathfrak{G}$ or $\mathfrak{M}$ is equivalent to distributional convergence of each neighbourhood $U_k(\mG_n^\bullet)$, $k \ge 0$, as $n$ tends to infinity.

 Given a finite graph or finite planar map $G$ there are two natural ways to select a random root vertex, either uniformly or with probability proportional to the vertex degree. The latter is called the stationary distribution, and also corresponds to selecting a uniform corner of a planar maps.


\subsection{Convergence of random probability measures}

The set $\mathbb{M}_1(\mathfrak{G})$  of probability measures on the Borel sigma algebra of $\mathfrak{G}$ is a Polish space with respect to the weak convergence topology.  Given a random finite connected simple graph $\mG$, we may view the law  of the rooted graph obtained via the uniform distribution or stationary distribution on the vertex set of $\mG$ as a random probability measure. That is, it is a random element of the space $\mathbb{M}_1(\mathfrak{G})$.  Convergence in probability of such random probability measures may be characterized as follows.

\begin{proposition}
	\label{prop:conv}
	Let $\mu, \mu_1,\mu_2,\ldots$ be random Borel probability measures on $\mathfrak{G}$, defined on a common probability space $(\Omega, \mathfrak{F}, \mathbb{P})$. The following statements are equivalent.
	\begin{enumerate}
		\item $\mu_n \convp \mu$, that is $d_\mathrm{P}(\mu_n, \mu) \convp 0$ for the Prokhorov distance $d_\mathrm{P}$.
		\item $\mathbb{E}_{\mu_n}[f] \convp \mathbb{E}_{\mu}[f]$ for each bounded continuous function $f: \mathfrak{G} \to \ndR$.
		\item Each subsequence $(n')$ has a subsequence $(n'')$ with $\mu_{n''}(\omega) \to \mu(\omega)$  for almost all $\omega \in \Omega$. 
		\item $\mu_n(U_k(\cdot) = H^\bullet) \convp \mu(U_k(\cdot) = H^\bullet)$ for each $k \ge 0$ and finite simple rooted graph $H^\bullet$.
	\end{enumerate}
An analogous statement holds for random planar maps and random measures on  $\mathfrak{M}$.
\end{proposition}
\begin{proof}
	The equivalence of the first three conditions is a classical property of general random measures, and the fourth condition is a special case of the second. It remains to show that the fourth condition already implies one of the others.
	
	Any open set in $\mathfrak{G}$ is the countable union of pre-image sets $U_k^{-1}(H^\bullet)$ with $k \ge 0$ and $H^\bullet\in \mathfrak{G}$ a graph with radius at most $k$.	
	Hence the indicator random variables $(\one_{U_k(\cdot) = H^\bullet})_{k, H^\bullet}$ form a countable convergence determining family by \cite[Thm. 2.2]{MR1700749}. Letting $(f_i)_{i \ge 1}$ denote a fixed ordering of this family, it follows that the metric
	\[d_\mathrm{l}(\nu, \nu') = \sum_{i \ge 1} \frac{1}{2^i} \left| \mathbb{E}_\nu[f_i] -  \mathbb{E}_{\nu'}[f_i] \right|, \qquad \nu, \nu' \in \mathbb{M}_1(\mathfrak{G}) \]
	induces the weak convergence topology on $\mathbb{M}_1(\mathfrak{G})$, see \cite[Ex. B.9]{MR3309619}. 
	
	Suppose that the fourth condition holds. It follows by a diagonalizing argument that any subsequence $(n')$ has a subsequence $(n'')$ such that almost all $\omega \in \Omega$  have the property that $\mathbb{E}_{\mu_{n''}(\omega)}[f_i] \to  \mathbb{E}_{\mu(\omega)}[f_i]$ for all $i \ge 1$. Thus $d_\mathrm{l}(\mu_{n''}(\omega), \mu(\omega)) \to 0$ and consequently $\mu_{n''}(\omega) \to  \mu(\omega)$ for almost all $\omega \in \Omega$.
\end{proof}


For example, if $\mu_n$ is the uniform rooting of a random simple finite  connected graph, then $\mu_n(U_k(\cdot) = H^\bullet)$ equals the percentage of vertices whose $k$-neighbourhood equals $H^\bullet$. If $\mu_n$ is the stationary distribution, it equals the  percentage of oriented edges with the $k$-neighbourhood of the origin being equal to $H^\bullet$. 
  
\begin{proposition}
	\label{prop:char}
	Let $\mG_n$ be sequence of random finite simple connected graphs. Let $v_n, v_n^{(1)}, v_n^{(2)}$ denote i.i.d. random vertices  of $\mG_n$, and let $\hat{\mG}, \hat{\mG}^{(1)}, \hat{\mG}^{(2)}$ be i.i.d. random elements of $\mathfrak{G}$. The following statements are equivalent.
	\begin{enumerate}
		\item $\mfL( (\mG_n, v_n) \mid \mG_n) \convp \mfL( \hat{\mG})$.
		\item  $((\mG_n, v_n^{(1)}), (\mG_n, v_n^{(2)})) \convd (\hat{\mG}^{(1)},\hat{\mG}^{(2)})$. 
		\item For any integer $k \ge 0$ and any two graphs $H_1^\bullet, H_2^\bullet \in \mathfrak{G}$ with radius at most $k$
		$\Pr{U_k(\mG_n, v_n^{(1)}) = H_1^\bullet,  U_k(\mG_n, v_n^{(2)}) = H_2^\bullet} \to  \Pr{U_k(\hat{\mG}) = H_1^\bullet}\Pr{U_k(\hat{\mG}) = H_2^\bullet}$. 
	\end{enumerate}
An analogous statement holds for random planar maps and random measures on  $\mathfrak{M}$.
\end{proposition}
\begin{proof}
	The equivalence of the first and second condition follows from general results by \cite[Lem. 2.3]{Buecher2018}. In detail:
	Consider the bounded Lipschitz metric
	\[
		d_{\mathrm{BL}}(\nu, \nu') = \sup_f \left|\mathbb{E}_{\nu}[f] - \mathbb{E}_{\nu'}[f] \right|, \qquad \nu, \nu' \in \mathbb{M}_1(\mathfrak{G}),
	\]
	with $f$ ranging over all functions $f: \mathfrak{G} \to [-1,1]$ that are Lipschitz continuous with Lipschitz constant $1$. By \cite[Lem. 2.3]{Buecher2018},
	\begin{align}
		\mfL( (\mG_n, v_n^{(1)}), (\mG_n, v_n^{(2)}) ) \to \mfL(\hat{\mG}) \otimes \mfL(\hat{\mG})
	\end{align}
	in the weak convergence topology is equivalent to
	\begin{align}
		\label{eq:intermediate}
		d_{\mathrm{BL}}(\mfL( (\mG_n, v_n) \mid \mG_n), \mfL(\hat{\mG})) \convp 0.
	\end{align}
	As $d_{\mathrm{BL}}$ induces the weak convergence topology on $\mathbb{M}_1(\mathfrak{G})$ by \cite[Lem. 2.4]{Buecher2018}, it follows by Proposition~\ref{prop:char} that  \eqref{eq:intermediate} is equivalent to
	\begin{align}
		\mfL( (\mG_n, v_n) \mid \mG_n) \convp \mfL( \hat{\mG}).
	\end{align}
	Hence the first and second condition stated in Proposition~\ref{prop:char} are equivalent. The third condition is a special case of the second. It remains to show that the third already implies the second. Since $U_i \circ U_j = U_i$ for all integers $j \ge i \ge 0$, it follows from the third condition that
	\begin{align}
		\left(U_k(\mG_n, v_n^{(1)}), U_\ell(\mG_n, v_n^{(2)})\right) \convd \left(U_k(\hat{\mG}^{(1)}), U_\ell(\hat{\mG}^{(2)})\right)
	\end{align}
	for all $k,\ell \ge 0$. Hence
	\begin{align}
	((\mG_n, v_n^{(1)}), (\mG_n, v_n^{(2)})) \convd (\hat{\mG}^{(1)},\hat{\mG}^{(2)}).
	\end{align}
\end{proof}

Using language from statistical physics, we may call  $(\mG_n, v_n) \convd \hat{\mG}$ the annealed version, and $\mfL( (\mG_n, v_n) \mid \mG_n) \convp \mfL( \hat{\mG})$ the quenched version. See also \cite[Sec. 7]{MR2908619} for similar terminology regarding fringe subtree count asymptotics. \cite{2015arXiv150408103K} used the quenched version of convergence in the context of the local topology for applications concerning subgraph counts.

The following observation is an easy exercise, that will be useful later on.
\begin{proposition}
	\label{pro:sparse}
	Let $(\mG_n,u_n)$ be a sequence of random finite vertex-pointed connected graphs. Let $v_n$ be chosen according to the stationary or uniform distribution on the vertex set. Suppose that $(\mG_n, v_n)$ admits a distributional limit in the local topology. If the number $|\mG_n|$ of vertices satisfies $|\mG_n| \convd \infty$, then $d_{\mG_n}(v_n, u_n) \convd \infty$.
\end{proposition}
\begin{proof}
	Let $v_n'$ be drawn independently from $v_n$ and according to the same law. The convergence of $(\mG_n, v_n)$ implies that for any integer $r \ge 1$ the numbers of edges and vertices in the neighbourhood $U_{r}(\mG_n,v_n)$ are stochastically bounded. Since $|\mG_n|\convd \infty$, this implies that  $v_n'$ lies outside of $U_r(\mG_n, v_n)$ with high probability. As $r\ge 1$ was arbitrary, this implies $d_{\mG_n}(v_n,v_n') \convd \infty$. Now suppose that $d_{\mG_n}(v_n, u_n)$ does not converge in distribution to infinity. Then there is an $\epsilon>0$, a constant $r \ge 1$, and a subsequence $(n')$ such that along that subsequence $\Pr{d_{\mG_n}(v_{n'}, u_{n'}) \le r} \ge \epsilon$.  By the triangle inequality it follows that $\Pr{d_{\mG_n}(v_{n'}, v'_{n'}) \le 2r} \ge \epsilon^2$, contradicting $d_{\mG_n}(v_n,v_n') \convd \infty$. Hence $d_{\mG_n}(v_n, u_n) \convd \infty$.
\end{proof}

All statements for random graphs in the present section hold analogously for random planar maps and random measures on the Borel sigma algebra of $\mathfrak{M}$.


\section{Condensation in simply generated trees}

\subsection{Trees and random walk}

\label{sec:sgt}

A detailed exposition of simply generated trees may be found in the comprehensive survey by \cite{MR2908619}.  A planted plane tree is a rooted unlabelled tree $T$, where the offspring of any vertex $v$ is endowed with a linear order. We let $d_T^+(v)$ denote the outdegree of $v$. The number of vertices of $T$ is denoted by $|T|$. Given a weight-sequence $\w = (\omega_k)_{k \ge 0}$ with $\omega_0>0$ and $\omega_k>0$ for at least one $k \ge 2$, we may form the weight $\omega(T) = \prod_{v \in T} \omega_{d^+_T(v)}$. The corresponding generating series for the class $\cZ$ of finite planted plane trees satisfies
\begin{align}
\cZ(z) = z \phi(\cZ(z))
\end{align}
with $\phi(z) := \sum_{k \ge 0} \omega_k z^k$. 
Letting $\rho_\phi$ denote the radius of $\phi(z)$, we may form the parameter 
\begin{align}
\label{eq:nu}
\nu= \lim_{t \nearrow \rho_\phi} \frac{\phi'(t) t}{\phi(t)}
\end{align}
For $0 < \nu \le 1$ we set $\tau= \rho_\phi$. For $\nu >1$ we let $\tau$ denote the unique positive real number with $\tau \phi'(\tau) = \phi(\tau)$. We let $\xi$ denote a random non-negative integer with probability generating function $\Ex{z^\xi} = \phi(\tau z) / \phi(\tau)$. The following result is given in \cite[Sec. 3, 7, and 15; Cor. 18.17]{MR2908619}.

\begin{lemma}
	\label{le:simplygen}
	Suppose that $\nu>0$. The simply generated tree $\mT_n$ is distributed like  a $\xi$-Galton--Watson tree $\mT$ conditioned on having $n$ vertices. It holds that $\Ex{\xi} = \min(1, \nu)$ and $\cZ(z)$ evaluated at its radius of convergence $\rho_\cZ = \tau/ \phi(\tau) <\infty$ equals $\cZ(\rho_\cZ) = \tau<\infty$. Moreover
	\begin{align}
	\label{eq:lem}
	[z^n] \cZ(z) =  \left( \frac{\tau}{\phi(\tau)} \right)^{-n} \frac{\tau}{n} \mathbb{P}(\xi_1 + \ldots + \xi_n= n-1),
	\end{align}
	with $(\xi_i)_{i \ge 1}$ denoting independent copies of $\xi$.
\end{lemma}

It follows from results for the big-jump domain in random walk, established by \cite[Cor. 2.1]{MR2440928} in a more general context, that  if 
\begin{align}
\label{eq:prerequisite}
\Pr{\xi=n} = f(n) n^{-1-\alpha}
\end{align} for some constant $\alpha>1$ and a slowly varying function $f$, then for each $\epsilon>0$
\begin{align}
	\label{eq:doit}
	\lim_{n \to \infty} \sup_{x \ge \epsilon + \Ex{\xi}} \left|\frac{\Pr{\xi_1 + \ldots + \xi_n= nx}}{ \Pr{\xi = \lfloor n(x-\Ex{\xi}) \rfloor}} -1\right| = 0.
\end{align}
Hence if $\Ex{\xi}<1$, then
\begin{align}
\label{eq:gwt}
\frac{1}{n}\Prb{\sum_{i=1}^{n} \xi_i = n-1} \sim  \Pr{\xi = \lfloor n(1- \Ex{\xi})\rfloor} \sim \frac{f(n)}{ (n(1-\Ex{\xi}))^{1+\alpha}}.
\end{align}
Thus in this setting, Equation~\eqref{eq:lem} simplifies to
\begin{align}
\label{eq:awesome}
[z^n] \cZ(z) \sim \left( \frac{\rho_\phi}{\phi(\rho_\phi)} \right)^{-n} n^{-1-\alpha} \frac{\rho_\phi f(n)}{ \left(1-\frac{\phi'(\rho_\phi) \rho_\phi}{\phi(\rho_\phi)}\right)^{1+\alpha}}.
\end{align}

\subsection{Condensation}
\label{sec:condens}

Setting $\theta = \min(\alpha,2)$, we let $X$ denote the $\theta$-stable random variable with Laplace exponent $\Ex{\exp(-\lambda X)} = \exp(\lambda^\theta )$. Let $h$ be the density of $X$. \cite[Thm. 19.34]{MR2908619} and \cite[Thm. 1]{MR3335012} established limits concerning the maximum degree  $\Delta(\mT_n)$ of the simply generated tree $\mT_n$ in the condensation regime. The following extension was recently given in \cite[Thm. 1.1]{2019arXiv190104603S}.

\begin{lemma}
	\label{le:maxllt}
	Suppose that $\Ex{\xi}<1$ and that \eqref{eq:prerequisite} holds. Then there is a slowly varying function~$g$ such that
	\begin{align}
	\Pr{\Delta(\mT_n) = \ell} = \frac{1}{g(n)n^{1/\theta}}\left(h\left(\frac{ (1-\Ex{\xi})n - \ell}{g(n)n^{1/\theta}}   \right) + o(1)\right)
	\end{align}
	uniformly for all $\ell \in \ndZ$.  If $f$ converges to a constant, then $g$ may be chosen to be constant.
\end{lemma}

A marked plane tree is a (planted) plane tree with a distinguished vertex. For $\Ex{\xi}<1$ it holds that $\Ex{|\mT|}< \infty$.  This allows us to define the size-biased tree $\mT^{\bullet}$ with distribution
\begin{align}
	\Pr{\mT^{\bullet} = (T,v)} = \frac{\Pr{\mT= T}}{\Exb{|\mT|}}.
\end{align}
We let $\mT^{\circ}$ denote a random marked tree that is distributed like $\mT^\bullet$ conditioned on having a marked leaf.

The fringe subtree of a plane tree at a vertex  is the subtree consisting of the vertex and all its descendants. The simply generated tree $\mT_n$ may be completely described by the ordered list $(F_i(\mT_n))_{1 \le i \le \Delta(\mT_n)}$ of fringe subtrees dangling from the lexicographically first vertex $v_\Delta$ with maximum outdegree $\Delta(\mT_n)$, and the marked tree $F_0(\mT_n)$ obtained from $\mT_n$ by marking the vertex $v_\Delta$ and  cutting away all its descendants.

\cite{MR2764126}, \cite[Thm. 20.2]{MR2908619}, \cite[Thm. 1.3]{MR3164755}, and \cite[Cor. 2.7, Thm. 2]{MR3335012} proved limits concerning $F_0(\mT_n)$ and the list $(F_i(\mT_n))_{1 \le i \le \Delta(\mT_n)}$. The following extension is given in \cite[Thm. 1.2]{2019arXiv190104603S}.

\begin{lemma}
	\label{le:fringe}
	Suppose that $\Ex{\xi}<1$ and that \eqref{eq:prerequisite} holds. There is a constant $C>0$ such that for any sequence of integers $(t_n)_{n \ge 1}$  with $t_n \to \infty$ it holds that
	\[
	\left(F_0(\mT_n), (F_i(\mT_n))_{1 \le i \le \Delta(\mT_n) - t_n}, \one_{\sum_{i = \Delta(\mT_n) - t_n }^{\Delta(\mT_n)} |F_i(\mT_n)| \ge C t_n } \right) \atv \left(\mT^\circ, (\mT^i)_{1 \le i \le \Delta(\mT_n)- t_n},0\right).
	\]
\end{lemma}

We are also going to require knowledge on the vicinity of some fixed number $k \ge 1$ of uniformly and independently selected vertices $u_1, \ldots, u_k$ of $\mT_n$. For $k=1$ this question was answered in \cite{rerootedstufler2019} in a very general setting. In the specific setting of Lemma~\ref{le:fringe} it is clear that with high probability each of the randomly selected vertices falls into a distinct fringe subtree dangling from the vertex~$v_\Delta$. We set $\mT_n^{\bullet(k)} = (\mT_n, u_1, \ldots, u_k)$ and let $\mT^{\bullet,1},\mT^{\bullet,2}, \ldots$ denote independent copies of $\mT^\bullet$.

\begin{corollary}
	\label{co:sizebias}
	Suppose that $\Ex{\xi}<1$ and that Condition~\eqref{eq:prerequisite} is satisfied. Then for each sequence of integers $t_n \to \infty$ with $t_n = o(n)$ it holds that
	\begin{align}
	\left(F_0(\mT_n^{\bullet(k)}), (F_i(\mT_n^{\bullet(k)}))_{1 \le i \le \Delta(\mT_n) - t_n}  \right) \atv \left(\mT^\circ, (\tilde{\mT}^i)_{1 \le i \le \Delta(\mT_n)- t_n}\right).
	\end{align}
	Here $\tilde{\mT}_i$ is defined by letting $j_1, \ldots, j_k$ be a uniformly at random selected sequence of distinct integers between $1$ and $\Delta(\mT_n)-t_n$, and setting
	\[
		\tilde{\mT}^i := \begin{cases} \mT^i, &i \notin \{j_1, \ldots, j_k\} \\ \mT^{\bullet, r}, &i=j_r\text{ with } 1 \le r \le k. \end{cases}
	\]
\end{corollary}
\begin{proof}
	By Lemmas~\ref{le:maxllt} and~\ref{le:fringe} it follows that with high probability each of the $u_1, \ldots, u_k$ falls into a distinct fringe subtree from $(F_i(\mT_n^{\bullet(k)}))_{1 \le i \le \Delta(\mT_n)-t_n}$. Hence the problem is reduced to understanding what happens when we select $k$ random vertices $v_1, \ldots, v_k$ from distinct trees in  a forest $(\mT^i)_{1 \le i \le \ell}$ where $\ell := \Delta(\mT_n) - t_n$ becomes large.
	
	


We set $\bar{\mT}^i = (\mT^i, v_j)$ when $v_j \in \mT^i$ for some  $1 \le j \le k$, and $\bar{\mT}^i = \mT^i$ otherwise.
For any sequence $(T_1, \ldots, T_\ell)$ of planted plane trees among which precisely $k$ selected trees $T_{\ell_1}, \ldots, T_{\ell_k}$ have a marked vertex, it holds that
	\begin{align*}
		\frac{\Prb{(\bar{\mT}^i)_{1 \le i \le \ell} = (T_1, \ldots, T_\ell) \mid \ell } }{ \Prb{(\tilde{\mT}^i)_{1 \le i \le \ell} = (T_1, \ldots, T_\ell) \mid \ell } } = \prod_{j=1}^k \frac{(\ell -j+1)\Ex{|\mT|}}{\sum_{i=1}^\ell |T_i| - \sum_{r=1}^{j-1} |T_{\ell_r}|}.
	\end{align*}
By the law of large numbers, it follows that $(\bar{\mT}^i)_{1 \le i \le \ell} \atv (\tilde{\mT}^i)_{1 \le i \le \ell}$. This completes the proof.
\end{proof}


\section{Enriched tree encodings and sampling procedures}

\label{sec:species}

Combinatorial bijections may serve to reduce the study of complex random graphs to the study of random structures that are more tractable by probabilistic methods. The notion of enriched trees in the framework of combinatorial species additionally allows for a unified view on a large class of (random) structures. 



\subsection{Combinatorial species of structures}
\label{sec:subseq}

A detailed account on combinatorial species may be found in the pioneering work by \cite{MR633783} and the comprehensive book by \cite{MR1629341}. The present sections  only aims to give a brief informal recap.

Informally speaking, a weighted combinatorial species $\cF$ is a collection  of labelled combinatorial  structures (such as trees or graphs) which is closed under isomorphism.  Any structure is assigned a weight, that for our purposes will always be a non-negative real number. The default case is where each structure receives weight $1$, and we will explicitly state whenever we deviate from this.

Each structure also has an underlying finite set (such as the corners of a map or the vertices of a graph), and weight-preserving isomorphisms are obtained by relabelling along bijections between underlying sets. We refer to the elements of the set as labels or atoms.

  The collection of structures over some fixed set $U$ is denoted by $\cF[U]$, and we use notation $\cF_n$ for the case $U = \{1, \ldots, n\}$. The size $|F|$ of an $\cF$-structure $F$ is given by the number of elements of its underlying set. We may form the exponential generating series $\cF(z)$ where the coefficient $[z^n]\cF(z)$ of $z^n$ is given by $1/n!$ times the sum of all weights of structures from~$\cF_n$. We require this number to be finite for all $n$. 

There are special examples of species, such as the species $\Set$ with $\Set[U] = \{U\}$ for all finite sets $U$, yielding $\Set(z) = \exp(z)$. The species $\Seq$ of sequences sends a set $U$ to the collection of all linear orderings of $U$, yielding $\Seq(z) = 1/(1-z)$. 

The radius of convergence of $\cF(z)$ is denoted by~$\rho_\cF$. Let's say the weight of an $\cF$-object $F$ is denoted by $\omega(F)$. Given a parameter $0 < t \le \rho_\cF$ with $\cF(t)<\infty$, the Boltzmann distribution with parameter $t$ selects an  object $F$ from $\bigcup_{n \ge 0} \cF_n$ with probability  $\omega(F)t^{|F|} / \cF(t)$. Unless we explicitly state a parameter, we will always refer to the case $t= \rho_\cF$ as \emph{the} Boltzmann distribution of $\cF$.


An unlabelled $\cF$-object is given by an isomorphism class. We may form the collection $\tilde{\cF}_n$ of all $n$-sized unlabelled $\cF$-objects. For asymmetric classes of structures (such as corner-rooted planar maps or planted plane trees), where each structure over an $n$-sized set has precisely $n!$ different labellings, the notions of labelled and unlabelled structures are equivalent. 

Given weighted species $\cF$ and $\cG$ there are numerous ways to create new species such as the sum $\cF + \cG$, the product $\cF \cdot \cG$, the substitution $\cF(\cG)$. There are also unary operations such as the derivative $\cF'$, and the pointing $\cF^\bullet$. It is also possible to restrict a species, for example for any integer $k\ge0$ the species $\cF_{\ge k}$ is the restriction of $\cF$ to all object of size at least $k$. We refer the reader to the cited sources for details on these operations.

The species $\cF$ and $\cG$ are termed isomorphic, denoted by $\cF \simeq \cG$ or $\cF=\cG$, if for any finite set $U$ there is a weight preserving bijection $\cF[U] \to \cG[U]$. This family of bijections is required to be compatible with the relabelling bijections.

The concept of species may be generalized to $k$-type species (with $k \ge 1$ fixed) in a straight-forward way. Rather than sets one uses $k$-type sets, which are  tuples $(U_1, \ldots, U_k)$ of finite sets. Multivariate generating series $\cF(z_1, \ldots, z_k)$ and partial derivatives $\frac{\partial{\cF}}{\partial z_i}$ are defined in an analogous manner as in the monotype case.


\subsection{Enriched trees}
\label{sec:enriched}
The concept of enriched tree was introduced by \cite{MR642392}. Given a weighted class $\cR$ an $\cR$-enriched tree consists of a rooted unordered labelled tree $T$ together with function $\alpha$ that assigns to vertex $v \in T$ with offspring set $M_v$ an element $\alpha(v) \in \cR[M_v]$ as decoration. In the algebra of species the class $\cA_\cR$ of $\cR$-enriched trees may be specified by the decomposition
\begin{align}
	\cA_\cR = \cX \cdot \cR(\cA_\cR),
\end{align}
because any such elements consists of a root vertex (accounting for the factor $\cX$) together with an $\cR$-structure whose labels may be identified with the enriched fringe subtrees dangling from the root. There are two key facts that we are going to use in our proofs:

\begin{lemma}
	\label{le:encoding}
	\emph{Any} class of structures satisfying a decomposition of the form $\cF \simeq \cX \cdot \cR(\cF)$ admits a canonical (weight-preserving) isomorphism  $\cF \simeq \cA_\cR$. 
\end{lemma}

In particular, if we want to sample a labelled  structure from $\cF_n$  with probability proportional to its weight, we may sample an element from $(\cA_\cR)_n$  in this way, and  apply the combinatorial bijection. Sampling a random enriched tree may be done by adorning a simply generated plane tree with independent decorations:

\begin{lemma}
	\label{le:sampling}
	The following sampling procedure draws an $n$-sized enriched tree from $(\cA_\cR)_n$ with probability proportional to its weight.
	\begin{enumerate}
		\item Select a random simply generated plane tree $\mT_n$ with weight-sequence $([z^k]\cR(z))_{k \ge 0}$.
		\item For each vertex $v \in \mT_n$ an $\cR$-structure $\beta_n(v)$ from the set $\cR_{d^+_{\cT_n}(v)}$ with probability proportional to its weight.
		\item Relabel this enriched plane tree $(\mT_n, \beta_n)$ by choosing  a bijection from its vertex set to $\{1, \ldots, n\}$ uniformly at random.
	\end{enumerate}
\end{lemma}

We are going to refer to way of chosing the random decorations in the second step as the $\emph{canonical way}$ of decorating the plane tree $\mT_n$. The random decoration $\beta_n$ is also referred to as the $\emph{canonical decoration}$.

Lemma~\ref{le:encoding} is a special case of Joyal's Implicit Species, given in \cite[Thm. 6]{MR633783}.
The idea is to \emph{unroll} the isomorphism via a recursive procedure $\Gamma$: Let's say we are given an $\cF$-object $F \in \cF[U]$ over some finite non-empty set $U$. Applying the bijection from the isomorphism to $F$ yields a single atom $a \in U$ (corresponding to the factor $\cX$), and an $\cR$-object $R$ whose atoms correspond some collection $M$ of $\cF$-objects. The underlying sets of structures in $M$ partition the set $U \setminus \{a\}$. If $M=\emptyset$ we stop and return as output $\Gamma(F)$ a tree consisting of a single root vertex with label $a$. Otherwise we apply the procedure $\Gamma$ recursively to each element from the collection $M$, resulting in a collection $M'$ of enriched trees. The output $\Gamma(F)$ is then formed by an enriched tree with root labelled with the atom $a$, with $M$ being the collection of branches dangling from $a$, and the $\cR$-structure of $a$ formed by relabelling $R$ canonically with the labels of the root vertices of the enriched trees from $M'$.

As for the second key fact, Lemma~\ref{le:sampling}, such a sampling procedure was given in \cite[Lem. 6.1] {2016arXiv161202580S} and, in a  less general setting, in \cite[Prop. 3.6]{PaStWe2016}. If the species $\cR$ is asymmetric, then so is the species $\cA_\cR$. In this case, we may actually work with unlabelled $\cA_\cR$-objects, that correspond bijectively to pairs $(T, \beta)$ of a planted plane tree $T$ and a function $\beta$ that assigns to each vertex $v \in T$ with outdegree $d_T^+(v)$ an unlabelled $\cR$-object  $\beta(v) \in \tilde{\cR}_{d^+_T(v)}$.

\section{Giant components in random compound structures}

\subsection{Gibbs partitions}

Let $\cF$ and $\cG$ be weighted species with $\cG(0)=0$. For any integer $n>0$ with $[z^n]\cF(\cG(z)) >0$ we may draw an element $\mS_n$ from $(\cF(\cG))_n$ with probability proportional to its weight. This random compound structure comes with a partition of the underlying set $[n]$, which is called a Gibbs partition. The term was coined in the  comprehensive work \cite{MR2245368}. 

\paragraph*{The convergent case}

If we remove one of the largest components of $\mS_n$ and replace it by a placeholder, we are left with an $\cF'(\cG)$-object. We let $\mS_n'$ denote the corresponding unlabelled object. 

Suppose that $0<\rho_\cG<\infty$ and $\cF'(\cG(\rho_\cG))<\infty$ (with $\rho_\cG$ denoting the radius of convergence of $\cG(z)$) and that  $[z^n]\cF(\cG(z)) >0$ for infinitely many $n$.  We say $\cF \circ \cG$ has convergent type, if $\mS_n'$ converges weakly to $\emph{the}$ Boltzmann distribution of $\cF'(\cG)$. This implies that the size $\Delta(\mS_n)$ of the largest component of $\mS_n$ equals $n - O_p(1)$.

As shown in \cite[Thm. 3.1]{doi:10.1002/rsa.20771}, a sufficient condition for this behaviour is when $\rho_\cF > \cG(\rho_\cG)$ and the coefficients $g_n := [z^n]\cG(z)$ satisfy
\begin{align}
	\label{eq:subexp}
	\frac{g_n}{g_{n+1}} \to \rho_\cG \qquad \text{and} \qquad \frac{1}{g_n} \sum_{i+j=n} g_i g_j \to 2 \cG(\rho_\cG) < \infty.
\end{align}
Condition~\eqref{eq:subexp} means that the size of a Boltzmann distributed $\cG$-object has a subexponential density \cite{MR3097424}. This entails (see \cite[Thm. 1]{MR0348393}, \cite[Thm. C]{MR772907})
\begin{align}
	\label{eq:gibconv}
	[z^n]\cF(\cG(z)) \sim \cF'(\cG(\rho_\cG))[z^n] \cG(z).
\end{align}
Condition~\eqref{eq:subexp} is satisfied when $g_n \rho_\cG^n$ varies regularly with an index smaller than $-1$. Another sufficient condition was given in  \cite[Lem. 3.3]{doi:10.1002/rsa.20771}:
\begin{lemma}
\label{lem:gibconv}
Suppose that there is a power series $\phi(z) = \sum_{k \ge 0} \omega_k z^k$ with non-negative coefficients and positive radius of convergence such that
	$\cG(z) = z \phi(\cG(z))$,
and $\omega_0>0$, $\omega_k>0$ for at least one $k \ge 2$,  and $\gcd\{ k \mid \omega_k >0\} = 1$. Then Equation~\eqref{eq:gibconv} holds and $\mS_n$ is convergent.
\end{lemma}
The difference $n - \Delta(\mS_n)$ admits a limit distribution if $\cF \circ \cG$ has convergent type. We will also require the following bound. The proof is by identical arguments as for a bound given in \cite[last display before Eq. (3.15)]{2019arXiv190104603S}.
\begin{proposition}
	\label{prop:gibbsbound}
	Suppose that $\rho_\cF > \cG(\rho_\cG)$ and that \eqref{eq:subexp} holds. 
	Let $X$ denote the size of a Boltzmann-distributed $\cG$-object. Then there are constants $C,c>0$ such that for all $k \ge 0$
	\begin{align}
		\Pr{n-\Delta(\mS_n) = k} \le C \frac{\Pr{X=n-k}  \Pr{X=k} }{\Pr{X=n}}\exp\left(-c\frac{k}{n-k}\right).
	\end{align}
\end{proposition}

\paragraph*{Convergence in the superexponential case}

Suppose that $\rho_\cG = 0$ and $\rho_\cF > 0$. Furthermore, suppose that $[z^1]\cF(z)>0$ and $g_n := [z^n]\cG(z)>0$ for sufficiently large integers $n$. It follows from \cite[Cor. 6.19, Eq. (6.25)]{2016arXiv161202580S} that if
\begin{align}
	\label{eq:condsuperexp}
	\sum_{i=1}^{n-1} g_i g_{n-i} = o(g_n).
\end{align}
then $\mS_n$ consists with high probability of a single $\cG$-object. In this case we say $\mS_n$ is super-convergent. A sufficient condition was given in \cite[Lem. 6.17]{2016arXiv161202580S}:
\begin{lemma}
	\label{lem:gibsuper}
Suppose that 
$\cG(z) = z \phi(\cG(z))$,
for a non-analytic series $\phi(z) = \sum_{k \ge 0} \omega_k z^k$ such that  $\omega_0>0$, $\omega_k>0$ for at least one $k \ge 2$,  and $\gcd\{ k \mid \omega_k >0\} = 1$. Then~\eqref{eq:condsuperexp} holds and $\mS_n$ is super-convergent.
\end{lemma}

\subsection{Random product structures}
Let $\cF$ and $\cG$ denote weighted species satisfying $f_n := [z^n]\cF(z)>0$ and $g_n := [z^n]\cG(z)>0$ for all sufficiently large $n$. As before we let $\rho_\cF$ and $\rho_\cG$ denote the radii of convergence of the generating series $\cF(z)$ and $\cG(z)$. We may draw a random pair $\mS_n$ with probability proportional to its weight among all $n$-sized $\cF \cG$-objects, and look at its $\cF$- and $\cG$-components $\cF(\mS_n)$ and $\cG(\mS_n)$. We describe two observations where it is unlikely for both components to be large at the same time. Both are immediate consequences of  standard properties of random variables with subexponential densities, see \cite{MR3097424}.

\begin{proposition}
	\label{pro:asymmm}
	Suppose that the coefficients $(g_n)_{n \ge 0}$ satisfy Equation~\eqref{eq:subexp} with $\rho_\cG>0$.
If $f_n = o(g_n)$, then  up to relabelling the $\cF$-component $\cF(\mS_n)$ converges  in distribution  to a Boltzmann distributed $\cF$-object with parameter $\rho_\cG$. Moreover,
\begin{align}
	\label{eq:asymas}
	[z^n] \cF(z) \cG(z) \sim \cF(\rho_\cG) g_n.
\end{align}
\end{proposition}
\begin{proof}
	 Let $X$ and $Y$ denote the sizes of  Boltzmann distributed $\cF$-objects and $\cG$-objects with parameter $\rho_\cG$. In order to verify the first claim, it suffices to show that $|\cF(\mS_n)| \convd X$. It is elementary that 
	 \begin{align}
	 		\label{eq:dareal0}
	 (|\cF(\mS_n)|, |\cG(\mS_n)|) \eqdist ((X,Y) \mid X + Y = n).
	 \end{align} Equation~\eqref{eq:subexp} ensures that $\Pr{Y= n} \sim \Pr{Y= n+1}$. By the result \cite[Thm. 4.23]{MR3097424} it follows that \begin{align}
	 \label{eq:inter}
	 \Pr{X+Y=n} \sim \Pr{Y=n}.
	 \end{align} Hence for any constant $k \ge 0$ we get
	\begin{align}
		\label{eq:dareal1}
		\Pr{X=k \mid X+Y=n} = \frac{\Pr{X=k}\Pr{Y=n-k}}{\Pr{X+Y=n}} \to \Pr{X=k}.
	\end{align}	
	This shows weak convergence of  $\cF(\mS_n)$ (up to relabelling). Moreover, Equation~\eqref{eq:asymas} follows directly from Equation~\eqref{eq:inter}.
\end{proof}

The following proposition describes the asymptotic behaviour of random $\cF \cG$-structure where either the $\cF$-component or the $\cG$-component has macroscopic size, but not both at the same time.

\begin{proposition}
	\label{pro:bothasymmm}
	Suppose that the coefficients $(g_n)_{n \ge 0}$ satisfy Equation~\eqref{eq:subexp} with $\rho := \rho_\cG>0$, and $f_n / g_n \to \lambda$ for some constant $0< \lambda<\infty$. Then
	\begin{align}
	\label{eq:assymser}
	[z^n] \cF(z)\cG(z) \sim \cF(\rho) g_n + \cG(\rho) f_n.
	\end{align}
	Moreover,
	\begin{align}
		\label{eq:thefinalend}
		d_{\mathrm{TV}}(\mS_n, \hat{\mS}_n) \to 0
	\end{align}
	for a random object  $\hat{\mS}_n$ constructed by the following procedure: 
	\begin{enumerate}
		\item  Set $p = \cF(\rho)/(\cF(\rho) + \lambda \cG(\rho))$, and flip a biased coin that shows head with probability $p$. 
		\item If it shows head, sample a Boltzmann distributed $\cF$-object $\mF$ with parameter $\rho$. If $|\mF|\le n$ and $g_{n-|\mF|}>0$, let $\hat{\mS}_n$ be the $\cF \cG$-structure  with $\cF$-component $\mF$ and  an $(n -|\mF|)$-sized  $\cG$-component drawn with probability proportional to its weight. If $|\mF| > n$ or $g_{n-|\mF|}=0$, we set $\hat{\mS}_n$ to some placeholder value $\diamond$.
		\item If the coin flip shows tails, we sample a Boltzmann distributed $\cG$-object $\mG$ with parameter~$\rho$. If $|\mG|\le n$ and $f_{n-|\mF|}>0$, we let $\hat{\mS}_n$ be the $\cF \cG$-structure with $\cG$-component $\mG$ and  an $(n -|\mG|)$-sized  $\cF$-component drawn with probability proportional to its weight. If $|\mG| > n$ or $f_{n-|\mG|}=0$, we set $\hat{\mS}_n$ to some placeholder value $\diamond$.
	\end{enumerate}
\end{proposition}
\begin{proof}
	We let $X$ and $Y$ denote the sizes of independent Boltzmann distributed $\cF$-objects and $\cG$-objects with parameter $\rho$.  
	Equation~\eqref{eq:dareal0} reduces the entire problem to comparing the sizes
	\[
		(X_n, Y_n) := (|\cF(\mS_n)|, |\cG(\mS_n)|) \eqdist ((X,Y) \mid X+Y = n)
	\]
	with the sizes
	\[
		(\hat{X}_n, \hat{Y}_n) := (|\cF(\hat{\mS}_n)|, |\cG(\hat{\mS}_n)|).
	\]
	Note that
	\[
	\Pr{X=n} \sim \lambda \frac{\cG(\rho)}{\cF(\rho)} \Pr{Y=n}.
	\]
	By the result \cite[Thm. 4.23]{MR3097424} it follows that
	\[
		\Pr{X + Y = n} \sim \Pr{X=n} + \Pr{Y=n}.
	\]
	This verifies Equation~\eqref{eq:assymser}. Equation~\eqref{eq:subexp} ensures that $\Pr{Y= n} \sim \Pr{Y= n+1}$. Hence for any constant $k \ge 0$ we get
	\begin{align}
	\label{eq:darealas1}
	\Pr{X=k \mid X+Y=n} = \frac{\Pr{X=k}\Pr{Y=n-k}}{\Pr{X+Y=n}} \sim \Pr{X=k} p \sim \Pr{\hat{X}_n = k}.
	\end{align}	
	Likewise, $\Pr{X=n} \sim \Pr{X=n+1}$ and 
	\begin{align}
	\label{eq:darealas2}
	\Pr{Y=k \mid X+Y=n} = \frac{\Pr{Y=k}\Pr{X=n-k}}{\Pr{X+Y=n}} \sim \Pr{Y=k} (1-p) \sim \Pr{\hat{Y}_n = k}.
	\end{align}
	This entails that there is a sequence $(t_n)_n$ of integers that tends to infinity sufficiently slowly such that
	\begin{align}
		\Pr{ (X_n,Y_n) = (k, n-k)} = (1+o(1)) \Pr{ (\hat{X}_n,\hat{Y}_n) = (k, n-k)} 
	\end{align}
	uniformly for all integers $k$ with $0 \le k \le t_n$ or $n-t_n \le k \le n$. As $\Pr{ t_n \le \hat{X}_n \le n - t_n} \to 0$ and 
	\[
	\Pr{ t_n \le X_n \le n - t_n } = O(1) \sum_{k=t_n}^{n-t_n} \frac{\Pr{X=k}\Pr{X=n-k}}{\Pr{X=n}} \to 0
	\] by subexponentiality,  this implies~\eqref{eq:thefinalend}.
\end{proof}

\section{Tree-like graph decompositions and convergent Gibbs partitions}
\label{sec:graphs}

Given a class of simple graphs $\cG$, we may form the subclass $\cC$ of connected graphs in $\cG$. Graphs in $\cC$ must have at least one vertex. We let $\cB$ denote the subclass of $2$-connected graphs in $\cC$. Here we consider the graph $K_2$ consisting of two vertices joined by a single edge as $2$-connected, such that $2$-connected graphs must have at least $2$ vertices. We let $\cF$ denote the subclass of $3$-connected graphs in  $\cB$. We require $3$-connected graphs to have at least $4$ vertices.

Throughout this section we assume that $\cG$ is stable under Tutte's decomposition, see \cite{MR586989,MR2524178,MR0327391,MR1546002,MR0140094,MR0210617,MR746795}. That is, a simple graph lies in $\cG$ if and only if all $3$-connected components lie in $\cF$. Note that this implies that $K_2$ belongs to $\cG$. We also consider $\cG=\cG(x,y)$ and its subclasses as $2$-sort species, with $x$ counting vertices and $y$ edges.

\subsection{Decomposition into connected components}

A graph consists of a collection of connected components. Hence the species $\cG$ and $\cC$ are related by the well-known decomposition
\begin{align}
\label{eq:connectedcomp}
\cG = \Set(\cC).
\end{align}
Since $\cG$ is stable under Tutte's decomposition, it is also stable under taking $2$-connected components in the block decomposition. Hence by \cite[Thm. 4.1]{doi:10.1002/rsa.20771} and \cite[Cor. 6.33]{2016arXiv161202580S} the Gibbs partition obtained by taking a uniform random graph from $\cG$ with $n$ vertices is convergent or super-convergent. This behaviour was established earlier for the class of planar graphs, general minor-closed addable classes, and many related classes by \cite{MR2418771,MR2507738,MR3084594}.


\subsection{The block decomposition}

If we root a connected graph at a vertex, then this marked vertex is contained in some set of blocks, that is, maximal $2$-connected subgraphs.  The entire graph may be decomposed into this collection of vertex-marked blocks and rooted graphs attached to the non-marked vertices. This block-decomposition \cite{MR633783, MR1629341,MR0357214} is expressed as follows:
\begin{align}
	\label{eq:blockdecomp}
	x \frac{\partial \cC}{\partial x} = x \Set \left( \frac{\partial \cB}{\partial x}\left(x \frac{\partial \cC}{\partial x},y\right) \right).
\end{align}
Hence Lemma~\ref{le:encoding} applies, yielding that the class $x \frac{\partial \cC}{\partial x}$ may be identified with $\Set\left( \frac{\partial \cB}{\partial x}\right)$-enriched trees.

\subsection{Tutte's decomposition}

\subsubsection{The $3$-decomposition for simple graphs}
\label{sec:3comb}

We recall the $3$-decomposition grammar of simple graphs by \cite{MR2524178}. 

A \emph{network} is a graph with $2$ distinguished unlabelled vertices $*_0$ and $*_1$, such that adding the edge $*_0 *_1$ (if absent) yields a $2$-connected graph. Here we do not exclude the case that the graph was already $2$-connected without this edge. For ease of reference we call $*_0$ the south pole and $*_1$ the north pole. We let $\cN$ denote the class of all networks whose $3$-connected components in the Tutte decomposition lie in $\cF$, and that are not equal to the trivial network consisting of  $*_0$ and $*_1$ without an edge between them. (However, $\cN$ does contain the graph consisting of $*_0$ and $*_1$ joined by an edge.). Hence
\begin{align}
\label{eq:NB}
\cN = (1+y)\frac{2}{x^2}\frac{\partial \cB}{\partial y}-1.
\end{align}
There are three types of networks in $\cN$, yielding
\begin{align}
\label{eq:sumdecomp}
\cN &= \cS  + \cP +  \cH
\end{align}
for the following subclasses of $\cN$: Networks from $\cS \subset \cN$ are series networks, obtained by identifying the north pole of one network with the south pole of another, and labelling this vertex and the remaining non-pole vertices. Any series network may be decomposed uniquely into a sequence of non-series networks, yielding
\begin{align}
\label{eq:series}
\cS = (\cH + \cP) \cdot \Seq_{\ge 1}(x(\cH + \cP)).
\end{align}
Networks from $\cP \subset \cN$ are parallel networks, obtained by identifying the south poles of two networks with each other, as well as the north poles.  We additionally declare the network consisting of $*_0$ and $*_1$ joined by a single edge as a parallel network. Hence parallel networks with non-adjacent poles correspond in a unique way to an unordered collection of at least $2$ non-parallel networks. Parallel networks with adjacent poles correspond to a possibly empty unordered collection of non-parallel networks. Hence
\begin{align}
\label{eq:parallel}
\cP = \Set_{\ge 2}(\cH + \cS) + y \Set(\cH + \cS).
\end{align}
Networks from $\cH \subset \cN$ are so called $h$-networks, obtained as follows. The class $\cF_{0,1}$ is formed by taking a graph from $\cF$, removing an edge, and making its endpoints poles. Thus
\begin{align}
	\cF_{0,1} = \frac{2}{x^2} \frac{\partial \cF}{\partial y}
\end{align}
The class $\cH$ is obtained by taking networks from $\cF_{0,1}$ and replacing their edges by arbitrary networks from $\cN$ (in any possible way). Hence
\begin{align}
\label{eq:Hdecomp}
\cH = \cF_{0,1}(x,\cN).
\end{align}
This leads to
\begin{align}
\label{eq:decomp}
\cN = (1+y)\Set \left( \cF_{0,1}(x,\cN) + \cS \right) -1
\end{align}

\subsubsection{An enriched tree encoding}

\label{eq:enrichedmaps}

We want to obtain an enriched tree encoding for $\cN$. The class $\cS$ could be expressed by $\cN$ by overcounting and substracting:
\begin{align}
\cS = \frac{x\cN^2}{1 + x \cN}.
\end{align}
Applied to Equation~\eqref{eq:decomp} this yields a recursive equation for $\cN$ involving substraction operations. However, we require a \emph{substraction-free} decomposition. For this reason, we take a different path, and  define
\begin{align}
\label{eq:k}
\cK := \cH + \cP.
\end{align}
Combining Equations~\eqref{eq:sumdecomp} and \eqref{eq:series} yields
\begin{align}
\label{eq:nbyk}
\cN = \cK \Seq(x\cK) \qquad \text{and} \qquad x\cN = \Seq_{\ge 1}(x \cK).
\end{align}
By Lemma~\ref{lem:gibconv}, Lemma~\ref{lem:gibsuper}, and Equation~\eqref{eqK:decomp} below it follows that the corresponding Gibbs partition obtained by taking a random $\cN$-object with $n$ edges and weight $x>0$ at vertices is convergent or super-convergent. This reduces the study of $\cN$ to the study of $\cK$. 

We proceed to show that $\cK$ admits an enriched tree encoding. Equation~\eqref{eq:series} may be rewritten by
\begin{align}
\label{eq:sbyk}
\cS = \cK \Seq_{\ge 1}(x\cK) \qquad \text{and} \qquad x\cS = \Seq_{\ge 2}(x \cK).
\end{align}
Combining Equations~\eqref{eq:Hdecomp} and ~\eqref{eq:nbyk} yields
\begin{align}
\label{eq:hbyk}
\cH =  \cF_{0,1}(x, \cK \Seq(x\cK)).
\end{align}
Combining Equations~\eqref{eq:k}, \eqref{eq:parallel}, \eqref{eq:sbyk}, and \eqref{eq:hbyk} yields
\begin{align}
\cK = &\cF_{0,1}(x, \cK \Seq(x\cK)) + \Set_{\ge 2}(\cF_{0,1}(x, \cK \Seq(x\cK)) + \cK \Seq_{\ge 1}(x\cK)) \\ &+y \Set(\cF_{0,1}(x, \cK \Seq(x\cK)) + \cK \Seq_{\ge 1}(x\cK)). \nonumber
\end{align}
We may write this as
\begin{align}
	\label{eq:kpre}
	\cK = \cI(x,\cK) + y \cJ(x,\cK),
\end{align}
with $\cI$ and $\cJ$ representing combinatorial species. 

We are not done yet, but the final step that we are going to take is more delicate than it appears. Isomorphisms or identities of combinatorial species are \emph{always} required to be compatible with relabelling. This is what allows us to deduce equations of cycle index sums and ordinary generating series from a single isomorphism between two combinatorial species. It is clear that, for example,  an unordered collection of at least two $\cK$-objects (that is, an element of $\Set_{\ge 2}(\cK)$), has no canonically distinguished $\cK$-object. Hence there is no species $\cD$ with an isomorphism $\Set_{\ge 2}(\cK) = \cK \cD(\cK)$. That being said, we may still define a weighted species $\cD$ with a single object of size $k$ and weight $1/(k+1)$ for each $k \ge 1$. This way, the product $\cK \cD(\cK)$ of weighted species has the same exponential generating series as $\Set_{\ge 2}(\cK)$. Moreover, for each finite set $U$, the objects from $(\cK \cD(\cK))[U]$ may be grouped into disjoint subsets, such that each of these classes has sum of weights $1$ and corresponds bijectively to an element from $\Set_{\ge 2}(\cK)[U]$. We denote this fact by $\Set_{\ge 2}(\cK) \equiv \cK \cD(\cK)$. In the same way we may form a weighted species $\cI^*$ and an analogous correspondence
\begin{align}
\label{eq:makestar}
\cI(x,\cK) \equiv \cK \cI^*(x,\cK).
\end{align}
We stress that in order to study random unlabelled objects we may not use this equation directly. Instead we would have to put additional effort into understanding the identity of cycle index sums derived from \eqref{eq:kpre}. Since the present work concerns itself exclusively with random labelled graphs this will not be an issue at all, but we want to stress this point due to ongoing research on random unlabelled planar graphs.

Equation~\eqref{eq:makestar} allows us to unwind~Equation~\eqref{eq:kpre}, yielding by induction
\begin{align*}
\cK &\equiv y \cJ(x,\cK) + \cK  \cI^*(x,\cK)  \\
	&=y \cJ(x,\cK) + (y \cJ(x,\cK) + \cK  \cI^*(x,\cK))  \cI^*(x,\cK) \\
	&= y \cJ(x,\cK) + y \cJ(x,\cK)\cI^*(x,\cK) + \cK  \cI^*(x,\cK)^2 \\
	&= \ldots \\
	&= y \cJ(x,\cK) \sum_{k \ge 0} \cI^*(x,\cK)^k.
\end{align*}
Thus
\begin{align}
	\label{eqK:decomp}
	\cK \equiv y \cR(x, \cK),
\end{align}
with
\begin{align}
	\label{eq:expforR}
	\cR(x,y) = \cJ(x,y) \Seq(\cI^*(x,y)).
\end{align}
Hence Lemma~\ref{le:encoding} yields a correspondence between the class $\cK$ and the class of $\cR$-enriched trees.

\section{Tree-like planar map decompositions and convergent Gibbs partitions}
\label{sec:mapdecomp}

We consider planar maps that are rooted at a corner, or equivalently a half-edge. An exception is made only for the map consisting of a single vertex and no edges. This map has no corners to be rooted at, but we count it as corner-rooted nevertheless.

Throughout this Section we let $\cM$ denote a class of planar maps that is closed under re-rooting and corners, and stable under Tutte's decomposition. We write $\cM = \cM(x,z)$ with $x$ marking \emph{non-root} vertices and $z$ marking corners.

\subsection{From connected to non-separable}


A planar map is termed \emph{separable}, if its edge-set may be partitioned into two disjoint subsets $E_1$ and $E_2$ such that there is precisely one vertex that is incident with both a member of $E_1$ and a member of $E_2$.  Note that the only non-separable map containing a loop is the map consisting of a single vertex with a loop. The map consisting of a single vertex with no edges is also non-separable.

\cite[Sec. 6]{MR0146823} described that a corner-rooted map consists of a non-separable map with arbitrary maps inserted at each corner.  This may be expressed in combinatorial language. Let $\cV$ denote the subclass of all non-separable maps in $\cM$. 
Then
\begin{align}
\label{eq:maps1to2}
z\cM = z\cV(x, z \cM),
\end{align}
yielding by Lemma~\ref{le:encoding} that the class $z \cM$ is isomorphic to $\cV$-enriched trees.

\subsection{From non-separable to $3$-connected}
\subsubsection{The $3$-decomposition for maps}
Let $\cD$ denote the class of all plane networks obtained by taking a non-separable map from $\cV$ with at least two vertices, removing the root-edge, and distinguishing its origin and destination as the south pole and north pole of the network. This way we make the original root-edge ``invisible''. We additionally forbid the network consisting of two poles and no edges between them. We write $\cD= \cD(x,y)$ with $x$ marking \emph{non-pole} vertices and $y$ marking edges, and likewise for any subclass of $\cD$. We let $\bar{\cF}_{0,1}$ denote the subclass  of $\cD$ obtained in the same way from all maps in $\cV$ that are $3$-connected (and in particular simple, as the definition of $k$-connectedness for multigraphs additionally requires this). The classes $\cV$ and $\cD$ are related by
\begin{align}
	\label{eq:relVD}
	\cV(x,z) = 1 + z^2 + z^2x\cD(x,z^2),
\end{align}
with $1$ representing the map consisting of a single vertex, and $z^2$ representing the map with a single vertex and a loop-edge.  The class $\cD$ has a known decomposition (see \cite{MR586989,MR0327391,MR1546002,MR0140094,MR0210617,MR746795}) into parallel networks $\bar{\cP}$, series networks $\bar{\cS}$, and $h$-networks $\bar{\cH}$:
\begin{align}
	\label{eq:decompD}
	\cD &= \bar{\cS} + \bar{\cP} + \bar{\cH}, \\
	\label{eq:decompS}
	\bar{\cS} &= (\bar{\cP} + \bar{\cH}) \Seq_{\ge 1}(x(\bar{\cP} + \bar{\cH})), \\
	\label{eq:decompP}
	\bar{\cP} &= y + (y + \bar{\cH} + \bar{\cS} )\cD,\\
	\label{eq:decompH}
	\bar{\cH} &= \bar{\cF}_{0,1}(x, \cD).
\end{align}

\subsubsection{An enriched tree encoding}
We proceed similarly as for the enriched tree encoding of $2$-connected graphs. Setting
\begin{align}
	\label{eq:decK}
	\bar{\cK} = \bar{\cH} + \bar{\cP},
\end{align}
we obtain from Equations~\eqref{eq:decompD} and~\eqref{eq:decompS} that
\begin{align}
	\label{eq:mapgibb}
	\cD= \bar{\cK} \Seq(x \bar{\cK}) \qquad \text{and} \qquad x\cD = \Seq_{\ge 1}(x\bar{\cK}).
\end{align}
Lemmas~\ref{lem:gibconv} and ~\ref{lem:gibsuper}, and Equation~\eqref{eq:decompKbar} below imply that the corresponding Gibbs partition obtained by taking a random $\cD$-object with $n$ edges and weight $x>0$ at vertices is convergent or super-convergent. 
This reduces the study of $\cD$ to the study of $\bar{\cK}$. 

We are going to show that $\bar{\cK}$ admits an enriched tree encoding. Combining Equations~\eqref{eq:decompS} and \eqref{eq:decompH}--\eqref{eq:mapgibb} yields
\begin{align}
	\bar{\cH} &= \bar{\cF}_{0,1}(x, \bar{\cK}\Seq(x \bar{\cK})), \\
	\label{eq:untilhere1}
	\bar{\cS} &= \bar{\cK}\Seq_{\ge 1}(x \bar{\cK}).
\end{align}
Equations~\eqref{eq:decK} and~\eqref{eq:decompP} yield
\begin{align}
	\bar{\cK} &=  \bar{\cS}\cD + \bar{\cH}(1 + \cD) + y(1 + \cD).
\end{align}
Substituting $\cD$, $\bar{\cS}$, and $\bar{\cH}$ by their expressions in terms of $\bar{\cK}$  (Equations~\eqref{eq:mapgibb}--\eqref{eq:untilhere1}) yields
\begin{align}
	\bar{\cK} &= \bar{\cK}^2\Seq_{\ge 1}(x \bar{\cK})\Seq(x \bar{\cK}) + \bar{\cF}_{0,1}(x, \bar{\cK}\Seq(x \bar{\cK}))(1 + \bar{\cK} \Seq(x \bar{\cK})) + y(1 + \bar{\cK} \Seq(x \bar{\cK})).
\end{align}
We may write this as
\begin{align*}
	\bar{\cK} &= \bar{\cK} \bar{\cI}^*(x, \bar{\cK}) + y \bar{\cJ}(x,\bar{\cK}),
\end{align*}
with $\bar{\cI}^*$ and $\bar{\cJ}$ representing combinatorial species. We may unroll this identity using induction, yielding
\begin{align}
	\bar{\cK} = y \bar{\cJ}(x,\bar{\cK}) \Seq(\bar{\cI}^*(x,\bar{\cK})).
\end{align}
and hence
\begin{align}
	\label{eq:decompKbar}
	\bar{\cK} = y\bar{\cR}(x, \bar{\cK})
\end{align}
for
\begin{align}
	\label{eq:forRbar}
	\bar{\cR} = \bar{\cJ}(x,y)\Seq(\bar{\cI}^*(x,y)).
\end{align}
Hence the class $\bar{\cK}$ may be identified with $\bar{\cR}$-enriched trees by Lemma~\ref{le:encoding}.


\section{Asymptotic enumeration using random walks with negative drift}
\label{sec:asympenum}

From here on  we let $\cM$ denote the class of all planar maps, and define the subclasses considered in Section~\ref{sec:mapdecomp}  accordingly. Likewise $\cG$ denotes the class of all planar graphs and the subclasses of Section~\ref{sec:graphs} are defined accordingly. Hence by Whitney's theorem, given in \cite{MR1506961},
\begin{align}
\cF_{0,1} = \frac{1}{2} \bar{\cF}_{0,1} = \frac{2}{x^2} \frac{\partial \cF}{\partial y}.
\end{align}
Hence the complete grammar from the previous two sections may be summarized as follows:
\begin{tcolorbox}
	{\footnotesize
		\begin{align*}
		\cG &= \Set(\cC) \\\\
		x \frac{\partial \cC}{\partial x} &= x \Set \left( \frac{\partial \cB}{\partial x}\left(x \frac{\partial \cC}{\partial x},y\right) \right) & z\cM &= z\cV(x, z \cM) \\
		\cN &= (1+y)\frac{2}{x^2}\frac{\partial \cB}{\partial y}-1 & \cV &= 1 + z^2 + z^2x\cD(x,z^2) \\
		\cN &= \cK \Seq(x\cK) & \cD &= \bar{\cK} \Seq(x \bar{\cK})\\
		\cK &\equiv y \cR(x, \cK) & \bar{\cK} &= y\bar{\cR}(x, \bar{\cK})
		\\&&\\
		\cR &= \cJ \Seq(\cI^*),&\bar{\cR} &= \bar{\cJ}\Seq(\bar{\cI}^*)\\
		\cI^* &= \frac{1}{y}\cF_{0,1}(x, y \Seq(xy)) & \bar{\cI}^* &= y\Seq_{\ge 1}(x y)\Seq(x y)\\
		&+ \frac{1}{y} \Set_{\ge 2}\left(\cF_{0,1}(x, y \Seq(xy)) + y \Seq_{\ge 1}(xy)\right), & &+ \frac{1}{y}\bar{\cF}_{0,1}(x, y\Seq(x y))(1 + y \Seq(x y)),\\
		\cJ &= \Set(\cF_{0,1}(x, y \Seq(xy)) + y \Seq_{\ge 1}(xy)),&\bar{\cJ} &=1 + y \Seq(x y), \\\\
		\cF_{0,1} &= \frac{1}{2} \bar{\cF}_{0,1} = \frac{2}{x^2} \frac{\partial \cF}{\partial y}.
		\end{align*}
	}%
\end{tcolorbox}

\subsection{Planar graphs}
\label{sec:subpla}
\cite{MR1946145} proved that
\begin{align}
\label{eq:Basymp}
[x^n]\cB(x,1) \sim c_\cB \rho_{\cB}^{-n} n^{-7/2},
\end{align}
with $c_\cB \approx 0.37042 \cdot 10^{-5}$ and $\rho_\cB \approx 0.03819$.
Setting $\phi_\cC(x) := \exp\left(\frac{\partial \cB}{\partial x}(x,1)\right)$ it holds that 
\begin{align}
\nu_\cC := \lim_{t \nearrow \rho_\cB} \phi_\cC'(x)x/\phi_{\cC}(x) = \rho_\cB \frac{\partial^2 \cB}{\partial x^2}(\rho_\cB,1).
\end{align}
Any connected graph with $n$ vertices has at least $n-1$ edges. Using this it is elementary that
\begin{align}
\label{eq:part}
\nu_{\cC} < \frac{\partial^2 \cB}{\partial x \partial y}(\rho_\cB,1).
\end{align}
In their proof, \cite{MR1946145} also obtained a singular expansion
\begin{align*}
\cN(x,1) = D_0 + D_2 X^2 + D_3 X^3 + O(X^4), \qquad X=\sqrt{1 - x/\rho_\cB},
\end{align*}
with analytic expressions for the constants $D_0 \approx 1.09417$ and $D_2 \approx -0.13749$.\footnote{When verifying these approximations, note that the expression for $D_2$ in~\cite{MR1946145} contains a small typo: it lacks a factor $t$. See \cite{MR2476775}. } Together with Equation~\eqref{eq:NB} this allows us to evaluate the upper bound in Inequality~\eqref{eq:part}, yielding
\begin{align}
\label{eq:Cnu}
\nu_\cC < \rho_\cB\left( \frac{1+D_0}{2}-1 \right) - \frac{D_2 \rho_\cB}{4} \approx 0.041302 < 1.
\end{align}
By Equation~\eqref{eq:gibconv} it follows that 
\begin{align}
\label{eq:phiC}
[x^n]\phi_\cC(x) \sim \exp\left(\frac{\partial \cB}{\partial x}(\rho_\cB,1)\right) [x^n]\frac{\partial \cB}{\partial x}(x,1) \sim \rho_\cB^{-1} \exp\left(\frac{\partial \cB}{\partial x}(\rho_\cB,1)\right) c_\cB \rho_\cB^{-n} n^{-5/2}.
\end{align}
Since $\nu_\cC < 1$ and since the coefficients of $\phi_{\cC}(\rho_\cB x)$ are regularly varying with index $-5/2$, we may apply Equation~\eqref{eq:awesome} to obtain
\begin{align}
\label{eq:numcon}
[x^n]\cC(x,1) &\sim c_\cC \rho_{\cC}^{-n} n^{-7/2}
\end{align}
with 
\begin{align}
\label{eq:thefirst}
\rho_\cC = \rho_\cB/ \phi_\cC(\rho_\cB)
\end{align} and $c_\cC= c_\cB (1 - \nu_\cC)^{-5/2}$. Equation~\eqref{eq:gibconv} implies
\begin{align}
\label{eq:numdisc}
[x^n]\cG(x,1) &\sim c_\cG \rho_{\cC}^{-n} n^{-7/2}
\end{align}
with \begin{align}
\label{eq:thesecond}
c_\cG= c_\cC \exp(\cC(\rho_\cC,1)).
\end{align} Equations~\eqref{eq:numcon}--\eqref{eq:thesecond} were obtained by \cite[Thm. 1]{MR2476775} using analytic methods, that additionally yield singular expansions of all involved generating series.






\subsection{$3$-connected planar networks}

\cite{MR0218275}  obtained the expression
\begin{align}
\label{eq:F01series}
\bar{\cF}_{0,1}(x,y) = y\left( \frac{1}{1+xy} + \frac{1}{1+y} - 1 - \frac{(1 + u)^2(1+v)^2}{(1+u + v)^3}\right)
\end{align}
where $u=u(x,y)$ and $v=(x,y)$ are specified by the system
\begin{align}
\label{eq:useries}
u = xy(1 + v)^2 \qquad \text{and} \qquad v = y(1+ u))^2.
\end{align}
This system yields the asymptotic enumeration of $3$-connected planar networks.  We follow the presentation by \cite[Thm. 9.13]{MR2484382}. Let $t>0$ be a constant. Applying \cite[Thm. 2.19]{MR2484382} yields square root singular expansions for $y \mapsto u(t, y)$ and $y \mapsto v(t, y)$. This leads to the representation
\begin{align}
\label{eq:usefulexp}
\frac{(1 + u)^2(1+v)^2}{(1+u + v)^3} = E_0(t)+E_2(t)Z^2 + E_3(t)Z^3 + O(Z^4) \qquad \text{with} \qquad  Z = \sqrt{1 - \frac{y}{\rho_\cF(t)}}.
\end{align}
Here $E_0(t), E_2(t), E_3(t), \rho_\cF(t)$ denote non-zero constants that depend only on $t$ and admit explicit expressions. Specifically, if $u_0$ is the positive solution of
\begin{align}
	\label{eq:deft}
	t = \frac{(1+u_0)(3u_0-1)^3}{16u_0},
\end{align}
then
\begin{align}
\label{eq:intermsofu0}
	\rho_\cF(t) = \frac{1}{(u_0+1)(3u_0-1)}, \quad E_0(t) = \frac{16(3u_0-1)}{27u_0(u_0+1)}, \quad E_2(t) = \frac{16(3u_0^2+1)(3u_0-1)}{81u_0^2(u_0+1)^2}.
\end{align}
Hence
\begin{align}
\label{eq:Fsing}
\bar{\cF}_{0,1}(t,y) = F_0(t) + F_2(t) Z^2 + F_3(t) Z^3 + O(Z^4)
\end{align}
for some non-zero constants $F_0(t), F_2(t), F_3(t), \rho_\cF(t)$ that depend only on $t$. Hence, by transfer theorems given in \cite[Ch. 6]{MR2483235},
\begin{align}
[y^n]\bar{\cF}_{0,1}(t,y) \sim c_\cF(t) \rho_\cF(t)^{-n} n^{-5/2},
\end{align}
with $c_\cF(t) = F_3(t) \frac{3}{4\sqrt{\pi}}$. 

\subsection{Non-separable planar maps}
The functions $y \mapsto \cR(t, y)$ and $y \mapsto \bar{\cR}(t, y)$ both have radius of convergence
\begin{align}
\label{eq:eqforrhoF}
\rho_\cR(t) =  \frac{\rho_\cF(t)}{1 + t\rho_\cF(t)}.
\end{align}
Simplifying Expression ~\eqref{eq:forRbar}  yields
\begin{align}
	\label{eq:expRbar}
	\bar{\cR}(t,y) = \frac{(1-t y) (1+\tilde{u}+\tilde{v})^3}{(1+\tilde{u})^2 (1+\tilde{v})^2} \quad \text{with} \quad  \tilde{u} = u\left(t,\frac{y}{1-ty}\right), \quad \tilde{v} = v\left(t,\frac{y}{1-ty}\right).
\end{align}
Expansion~\eqref{eq:usefulexp} allows us also to deduce a singular expansion of $\bar{\cR}(t,y)$, yielding
\begin{align}
\label{eq:Rbarasympt}
[y^n]\bar{\cR}(t,y) \sim c_{\bar{\cR}}(t) \rho_\cR(t)^{-n} n^{-5/2}
\end{align}
for some constant $c_{\bar{\cR}}(t)>0$. The $\nu$-parameter from Equation~\eqref{eq:nu} corresponding to $\bar{\cR}(t,y)$ is given by
\begin{align}
\label{eq:defnu}
\nu_{\bar{\cK}}(t) := \frac{\rho_{\cR}(t) \frac{\partial \bar{\cR}}{\partial y}(t, \rho_{\cR}(t))}{\bar{\cR}\left(t,\rho_{\cR}(t) \right)} = \rho_{\cR}(t)  \frac{\partial}{\partial y}\log(\bar{\cR}(t,y))\big|_{y=\rho_\cR(t)}.
\end{align}
Using Expansion~\eqref{eq:usefulexp} it follows that
\begin{align}
\label{eq:fornuK}
	\nu_{\bar{\cK}}(t) = \frac{E_2}{E_0}(1+t\rho_\cF(t)) - t \rho_\cF(t) = \frac{21u_0^2 + 6u_0+1}{48 u_0^2}.
\end{align}
The rational function $r \mapsto \frac{(1+r)(3r-1)^3}{16r}$ increases strictly on $\ndR_{>0}$ and assumes the value $0$ at $r=1/3$. As $t>0$, it follows from Equation~\eqref{eq:deft} that 
\begin{align}
\label{eq:boundu0}
u_0 = u_0(t) > 1/3
\end{align} 
The function $s \mapsto \frac{21s^2 + 6s+1}{48 s^2}$ decreases strictly on $\ndR_{>0}$, and assumes the value $1$ at $s=1/3$. Hence
\begin{align}
\label{eq:Kbarsubcrit}
\nu_{\bar{\cK}}(t)<1
\end{align}
for all $t>0$. This allows us to apply Equation~\eqref{eq:awesome}, yielding
\begin{align}
\label{eq:dooo}
[y^n]\bar{\cK}(t,y) &\sim c_{\bar{\cK}}(t) {\rho_{\bar{\cK}}(t)}^{-n} n^{-5/2}
\end{align}
with  $\rho_{\bar{\cK}}(t) = \rho_{\cR}(t)/ \bar{\cR}(t,\rho_\cR)$ and $c_{\bar{\cK}}(t) = \rho_{\cR}(t) c_{\bar{\cR}}(t) (1- \nu_{\bar{\cK}}(t))^{-5/2}$. Equation~\eqref{eq:gibconv} implies
\begin{align}
\label{eq:numnetw}
[y^n]\cD(t,y) &\sim c_\cD(t) {\rho_{\bar{\cK}}(t)}^{-n} n^{-5/2}
\end{align}
with $c_{\cD}(t) = c_{\bar{\cK}}(t)/(1-t\bar{\cK}(t,\rho_{\bar{\cK}}(t)))$. Hence we may deduce the known asymptotic formula
\begin{align}
\label{eq:nuasymptotic}
[z^{2n}]\cV(t,z) &\sim c_\cV(t) {\rho_{\bar{\cK}}(t)}^{-n} n^{-5/2}
\end{align}
with $c_\cV(t) = t c_{\cD}(t) \rho_{\bar{\cK}}(t)$.

\subsection{Planar maps}
By Equation~\eqref{eq:nuasymptotic}, the function $z \mapsto \cV(t,z)$ has radius of convergence $\rho_\cV(t)= \sqrt{\rho_{\bar{\cK}}(t)}$. Hence the  $\nu$-parameter from Equation~\eqref{eq:nu} corresponding to $\cV(t,z)$ is given by
\begin{align}
\label{eq:tmp11}
	\nu_{\cM}(t) := \frac{\rho_\cV(t) \frac{\partial \cV}{\partial y}\left(t, \rho_\cV(t) \right)  }{\cV\left(t,\rho_\cV(t) \right)} = \frac{2\rho_{\bar{\cK}}(t)(1+t\cD(t,\rho_{\bar{\cK}}(t)) + t\rho_{\bar{\cK}}(t)\frac{\partial \cD}{\partial y}(t,\rho_{\bar{\cK}}(t)))}{1+\rho_{\bar{\cK}}(t)(1 + t\cD(t,\rho_{\bar{\cK}}(t)))}.
\end{align}
It holds that $\bar{\cK}(t,\rho_{\bar{\cK}}(t)) = \rho_\cR(t)$ by Lemma~\ref{le:simplygen}. Hence, using Equation~\eqref{eq:eqforrhoF},
\begin{align}
	\label{eq:last2}
	\cD(t,\rho_{\bar{\cK}}(t)) = \frac{\bar{\cK}(t,\rho_{\bar{\cK}}(t))}{1 - t\bar{\cK}(t,\rho_{\bar{\cK}}(t))} = \frac{\rho_\cR(t)}{1 - t\rho_\cR(t)} =\rho_\cF(t).
\end{align}
Using $\rho_{\bar{\cK}}(t) = \rho_{\cR}(t)/ \bar{\cR}(t,\rho_\cR(t))$, $\bar{\cK}(t,\rho_{\bar{\cK}}(t)) = \rho_\cR(t)$, and Equation~\eqref{eq:eqforrhoF}, we obtain
\begin{align}
\label{eq:simplifyme}
\rho_{\bar{\cK}}(t)\frac{\partial \cD}{\partial y}(t,\rho_{\bar{\cK}}(t)) = \frac{\rho_{\bar{\cK}}(t)\frac{\partial \bar{\cK}}{\partial y}(t,\rho_{\bar{\cK}}(t))}{1 - t \bar{\cK}(t,\rho_{\bar{\cK}}(t))} = \rho_{\cF}(t) \frac{\frac{\partial \bar{\cK}}{\partial y}(t,\rho_{\bar{\cK}}(t))}{\bar{\cR}(t,\rho_\cR)}
\end{align}
Differentiating $\bar{\cK}(t,y) = y\bar{\cR}(t, \bar{\cK}(t,y))$ yields
\begin{align}
\label{eq:last1}
\frac{\partial \bar{\cK}}{\partial y}(t,y) = \frac{\bar{\cR}(t, \bar{\cK}(t,y))}{1 - y \frac{\partial \bar{\cR}}{\partial y}(t, \bar{\cK}(t,y))}.
\end{align}
Using  the definition of $\nu_{\bar{\cK}}(t)$ in Equation~\eqref{eq:defnu}, it follows that Equation~\eqref{eq:simplifyme} simplifies to 
\begin{align}
\rho_{\bar{\cK}}(t)\frac{\partial \cD}{\partial y}(t,\rho_{\bar{\cK}}(t)) = \frac{\rho_{\cF}(t)}{1 - \nu_{\bar{\cK}}(t)}.
\end{align}
Using Equation~\eqref{eq:last2}, it follows that Equation~\eqref{eq:tmp11} simplifies to
\begin{align}
\label{eq:simplifymereally}
\nu_{\cM}(t) = \frac{2 \rho_{\bar{\cK}}(t)\left( 1 + t \rho_{\cF}(t)(1 + 1/(1- \nu_{\bar{\cK}}(t))) \right)}{1 + \rho_{\bar{\cK}}(t)(1 + t \rho_{\cF}(t))}.
\end{align}
Note that by Equations~\eqref{eq:expRbar}, \eqref{eq:eqforrhoF} and \eqref{eq:usefulexp}
\begin{align}
	\rho_{\bar{\cK}}(t) = \frac{\rho_{\cR}(t)}{\bar{\cR}(t,\rho_\cR(t))} = \rho_{\cF}(t)\frac{(1 + u(1,\rho_{\cF}(t)))^2(1+v(\rho_{\cF}(t)))^2}{(1 + u(t,\rho_{\cF}(t)) + v(t, \rho_{\cF}(t)))^3 } = \rho_\cF(t) E_0(t).
\end{align}
Plugging this equation into~\eqref{eq:simplifymereally} yields an expression of $\nu_{\cM}(t)$ in terms of $t$, $\rho_\cF(t)$, $E_0(t)$, and $\nu_{\bar{\cK}}(t)$.  Equations~\eqref{eq:intermsofu0} and \eqref{eq:fornuK} yield Expressions of $\rho_\cF(t)$, $E_0(t)$, and $\nu_{\bar{\cK}}(t)$ in terms of $u_0$. Hence we obtain an expression of $\nu_{\cM}(t)$ in terms of $t$ and $u_0$. Using~\eqref{eq:deft} we may simplify this to an expression in terms of $u_0$ alone, yielding
\begin{align}
	\nu_{\cM}(t)  = \frac{2 \left(1+19 u_0+51 u_0^2+225 u_0^3\right)}{(1+u_0) (1+3 u_0)^3 (1+9 u_0)}.
\end{align}
This rational expression is strictly decreasing in $u_0$, and assumes the value $1$ at the point $1/3$.  As $u_0 = u_0(t) > 1/3$ by~\eqref{eq:boundu0}, this implies
\begin{align}
	\label{eq:nuM}
	\nu_{\cM}(t) < 1.
\end{align}
Using Equations~\eqref{eq:lem} and~\eqref{eq:doit}, it follows that
\begin{align}
	[z^n] \cM(t,z) \sim c_{\cM}(t) n^{-5/2} \rho_{\cM}(t)^{-n} \quad \text{ }
\end{align}
as $n \in 2\ndN_0$ becomes large, with $\rho_\cM(t) = \frac{\rho_\cV(t)}{\cV(t, \rho_\cV(t))}$ and $c_\cM(t) = \frac{c_\cV(t) \rho_{\bar{\cK}}(t)}{\cV(t, \rho_\cV(t)) (1 - \nu_\cM(t))^{5/2}}$.

\subsection{Planar networks}

Simplifying Expression~\eqref{eq:expforR} yields
\begin{align}
	\cR(t,y) = \frac{e^{\cF\left(t,\frac{y}{1-ty}\right)+\frac{t y^2}{1-t y}} y (1-t y)}{1+y-t y-e^{\cF\left(t,\frac{y}{1-ty}\right)+\frac{t y^2}{1-t y}} (1-t y)}.
\end{align}
The $\nu$-parameter from Equation~\eqref{eq:nu} corresponding to $\cR(t,y)$ may be expressed by
\begin{align}
\nu_{\cK}(t) := \frac{\rho_{\cR}(t) \frac{\partial {\cR}}{\partial y}(t, \rho_{\cR}(t))}{{\cR}\left(t,\rho_{\cR}(t) \right)} = \rho_{\cR}(t)  \frac{\partial}{\partial y}\log({\cR}(t,y))\big|_{y=\rho_\cR(t)}.
\end{align}
Setting 
\begin{align}
F(y) := \cF\left(t,\frac{y}{1-ty}\right)+\frac{t y^2}{1-t y} 
\end{align} we obtain
\begin{align}
\frac{\partial}{\partial y}\log({\cR}(t,y)) = F'(y) + \frac{1}{y} - \frac{t}{1-ty} - \frac{1-t+\exp(F(y))(t-  F'(y)(1-ty))}{1 + y(1-t) - \exp(F(y))(1-ty)}.  
\end{align}
Similar as for planar maps, this allows us to express $\nu_K(t)$ in terms of $u_0$, and  using Equation~\eqref{eq:boundu0} it follows that
\begin{align}
	\label{eq:ksubcrit}
	\nu_{\cK}(t) < 1.
\end{align}
Expansion~\eqref{eq:usefulexp} implies a singular expansion of $\cR(t,y)$, yielding
\begin{align}
\label{eq:Rasymp}
[y^n]\cR(t,y) \sim c_{\cR}(t) \rho_\cR(t)^{-n} n^{-5/2}
\end{align}
for some constant $c_{\cR}(t)>0$. By Equation~\eqref{eq:awesome} it follows that
\begin{align}
[y^n]\cK(t,y) &\sim c_{\cK}(t) {\rho_{\cK}(t)}^{-n} n^{-5/2}
\end{align}
with  $\rho_{\cK}(t) = \rho_{\cR}(t)/ \cR(t,\rho_\cR)$ and $c_{\cK}(t) = \rho_{\cR}(t) c_{\cR}(t) (1- \nu_{\cK}(t))^{-5/2}$. Hence by  Equation~\eqref{eq:gibconv} 
\begin{align}
[y^n]\cN(t,y) &\sim c_\cN(t) {\rho_{\cK}(t)}^{-n} n^{-5/2}
\end{align}
with $c_{\cN}(t) = c_{\cK}(t)/(1-t\cK(t,\rho_{\cK}(t)))$.


\section{Quenched local convergence}

\subsection{Weighted non-separable maps}

We use the notation $\ve(\cdot)$, $\ed(\cdot)$, and $\co(\cdot)$ for the number of vertices, edges, and corners. Let $t>0$ be a constant. We let $\mM_n^t$ denote a random planar map with $n$ edges that is drawn with probability proportional to  $t^{\ve(M)}$.

As corner-rooted planar maps are asymmetric, it follows from Equation~\eqref{eq:maps1to2} and Section~\ref{sec:enriched} that any planar map $M$ with $n$ corners corresponds bijectively to a pair $(T, \beta)$ of a (planted) plane tree $T$ with $2n+1$ vertices and a function $\beta$ that assigns to each vertex $v \in T$ a non-separable map with $d^+_{T}(v)$ corners. Here the decoration $\beta(o)$ of the root vertex $o$ of $T$ corresponds to the non-separable component of $M$ containing the root-edge. The non-root corners of $\beta(o)$ correspond to the offspring vertices of $o$ in $T$. The decorated fringe subtree at such an offspring vertex represents the map inserted at the corresponding corner of $\beta(o)$. Here $\emph{corresponding}$ means according to a fixed canonical ordering of the non-root corners of $\beta(o)$. We choose this ordering according to a breadth-first-search exploration, so that the distance to the root-corner is non-decreasing in the ordering.

The  enriched tree $(\mT_{n}^\cM, \beta_{n}^\cM)$ corresponding to the random map $\mM_n^t$ admits an easy description. By Lemma~\ref{le:sampling} and subsequent remarks on asymmetric species in Section~\ref{sec:enriched}, the random plane tree $\mT_{n}^\cM$ is a simply generated tree with weight-sequence $(\omega_k^\cM)_{k \ge 0}$ given by
\begin{align}
	\omega_k^\cM= [z^k] \cV(t, z), \quad k \ge 0.
\end{align}
Given $\mT_n^\cM$, each decoration $\beta_n^\cM(v)$, $v \in \mT_n^\cM$, gets drawn with probability proportional to its weight among all non-separable maps with $d^+_{\mT_n^\cM}(v)$ corners, independently from the remaining decorations. Here the weight of such a map $V$ is $t^{\ve(V) -1}$.

By Inequality~\eqref{eq:nuM}, the asymptotic expression~\eqref{eq:nuasymptotic}, and Lemma~\ref{le:simplygen} it follows that $\mT_n^\cM$ is distributed like a $\xi^\cM$-Galton--Watson tree $\mT^\cM$ conditioned on having $2n+1$ vertices, with offspring distribution $\xi^\cM$ satisfying
\begin{align}
	\label{eq:tailmaps}
	 \Ex{\xi^\cM}< 1 \qquad \text{and} \qquad \Pr{\xi^\cM = 2n} \sim \frac{c_{\cV}(t)}{\cV(t,\rho_{\bar{\cK}}(t))} n^{-5/2}.
\end{align}
This offspring distribution $\xi^\cM$ is a random even integer. Hence $\xi^\cM/2$ satisfies condition \eqref{eq:prerequisite}, but $\xi^\cM$ does not. However, it is easy to see that Lemma~\ref{le:maxllt} may be extended to this setting. The reason for this is that the proof of Lemma~\ref{le:maxllt} given in \cite[Thm. 1.1]{2019arXiv190104603S} uses the well-known fact that $\Delta(\mT_n^\cM)$ corresponds to the largest jump in an $n$-step random walk with step-distribution $\xi^\cM$ conditioned to arrive at $n-1$ after $n$-steps. Hence the proof of Lemma~\ref{le:maxllt} may be adapted to this setting by rescaling by the factor $\frac{1}{2}$. Keeping in mind that $\mT_n^\cM$ has $2n+1$ vertices, it follows that the (with high probability unique) largest non-separable component $\cV(\mM_n^t)$ satisfies
\begin{align}
	\label{eq:sepcoresize}
	\Pr{\ed(\cV(\mM_n^t)) = \ell} = \frac{1}{g_{\cM}(t) n^{2/3}}\left(h\left(\frac{ (1-\Ex{\xi^\cM})n - \ell}{g_{\cM}(t) n^{2/3}}   \right) + o(1)\right)
\end{align}
uniformly for all $\ell \in \ndZ$ with $g_{\cM}(t) >0$ a constant. A similar probabilistic approach to the block sizes in random planar maps was used by \cite{AddBe}, and the local limit theorem itself is a celebrated result by \cite{MR1871555}. Likewise, Lemma~\ref{le:fringe} also holds in this setting despite the periodicity, as its proof given in \cite[Thm. 1.2]{2019arXiv190104603S} may be adapted analogously.
  
  The following quenched limit was shown recently in \cite{2019arXivsubmit} by establishing quenched limits for extended fringe subtrees of re-rooted multi-type Galton--Watson trees and applying the Bouttier--Di Francesco--Guitter bijection, see \cite[Sec. 2]{MR2097335}. An annealed version may be deduced by applying planar duality to  the earlier annealed convergence of face-weighted random planar maps by \cite{MR3769811}, who established local convergence of such multi-type trees close to the fixed root.
   
  \begin{lemma}[\cite{2019arXivsubmit}]
  	\label{le:quenchmap}
  	 Let  $c_n$ denote a uniformly selected corner of $\mM_n^t$. There is a random infinite planar map $\hat{\mM}^t$ with
  	\begin{align}
  	\Pr{ (\mM_n^t, c_n) \mid \mM_n^t} \convp \mfL(\hat{\mM}^t).
  	\end{align}
  \end{lemma}

   Using this convergence we deduce the following Lemma, which is a quenched version of a more general argument on random block-weighted planar maps with no tail assumptions by \cite[Thm. 6.59]{2016arXiv161202580S}.

\begin{lemma}
	\label{le:map1to2}
	The maximal non-separable component $\cV(\mM_n^t)$ admits a distributional limit~$\hat{\mV}^t$ in the local topology. Letting  $c_n$ denote a uniformly selected corner of $\cV(\mM_n^t)$, it holds that
	\begin{align}
	\Pr{ (\cV(\mM_n^t), c_n) \mid \cV(\mM_n^t)} \convp \mfL(\hat{\mV}^t).
	\end{align}
\end{lemma}
\begin{proof}
Recall that we enumerated the corners of $\cV(\mM_n^t)$ in a canonical order, starting with the root-corner. The map $\mM_n^t$ consists of $\cV(\mM_n^t)$ together with maps $(\cM_i(\mM_n^t))_{1 \le i \le \Delta(\mT_n^\cM)}$ attached to each of its corners. The map $\cM_1(\mM_n^t)$ has an additional marked corner, corresponding to the root-corner of $\mM_n^t$. For each $2 \le i \le \Delta(\mT_n^\cM)$ the map $\cM_i(\mM_n^t)$ attached to the $i$th corner is determined by the fringe subtree $F_i(\mT_n^\cM)$ and the restriction $\beta_n^\cM|_{F_i(\mT_n^\cM)}$ to its vertex set. The root corner of $\cV(\mM_n^t)$ differs, as there are two maps attached to it: the map determined by $(F_1(\mT_n^\cM), \beta_n^\cM|_{F_1(\mT_n^\cM)})$, and the bi-corner-rooted map determined by the marked plane tree $F_0(\mT_n^\cM)$ and the restriction of $\beta_n^\cM$ to its non-marked vertices.

Let $\mT^{\bullet \cM}$ and $\mT^{\circ \cM}$ be defined for the offspring distribution $\xi^\cM$ analogously as $\mT^\bullet$ and $\mT^\circ$ were defined for $\xi$ in Section~\ref{sec:condens}. 
For each $i \ge 2$ let $\mM(i)$ be  an independent copy of the planar map~$\mM$ corresponding to the tree $\mT^\cM$ with a canonical random decoration. (See Section~\ref{sec:enriched} for the definition of \emph{canonical decorations}.)
 Let $\mM(1)$ denote the random map with a second marked corner that corresponds to canonically decorated versions  of $\mT_1^\cM$ and $\mT^{\circ \cM}$ in the same way as $\cM_1(\mM_n^t)$ does to decorated versions of $F_0(\mT_n^\cM)$ and $\cF_1(\mT_n^\cM)$. By (our adaption of) Lemma~\ref{le:fringe}, we obtain that there is a constant $C>0$ such that for any sequence of integers $(t_n)_n$ with $t_n\to \infty$ and $t_n=o(n)$
\begin{align}
	\label{eq:attached1}
	\left( \cM_i(\mM_n^t) \right)_{1 \le i \le \Delta(\mT_n^\cM)-t_n} \atv \left(\mM(i)\right)_{1 \le i \le \Delta(\mT_n^\cM)-t_n}
\end{align}
and with high probability
\begin{align}
	\label{eq:attached2}
	\sum_{i=\Delta(\mT_n^\cM)-t_n}^{\Delta(\mT_n^\cM)} \co(\cM_i(\mM_n^t)) \le C t_n.
\end{align}
The corners of $\mM_n^t$ correspond bijectively to the corners of $\left( \cM_i(\mM_n^t) \right)_{1 \le i \le \Delta(\mT_n^\cM)}$ and the $\Delta(\mT_n^\cM)$ corners of $\cV(\mM_n^t)$. Let us select two corners $v_1, v_2$ of $\mM_n^t$ uniformly at random. It will be notationally convenient to refer to $v_1$ as a red corner, and $v_2$ as a blue corner. If $v_1$ or $v_2$ corresponds to a corner in a component $\cM_i(\mM_n^t)$ or to the corner of $\cV(\mM_n^t)$ where $\cM_i(\mM_n^t)$ is attached, then we let $\bar{\cM}_i(\mM_n^t)$ denote $\cM_i(\mM_n^t)$ together with the location(s) and colour(s) of $v_1$ and/or $v_2$. Otherwise we just set $\bar{\cM}_i(\mM_n^t) = \cM_i(\mM_n^t)$.  Note that this / these location(s) may either be a corner of  $\cM_i(\mM_n^t)$ or an additional placeholder corner, referring to the corner of $\cV(\mM_n^t)$ where $\cM_i(\mM_n^t)$ is attached. 

Let  $\mM^\bullet$ denote  the random bi-corner-rooted map corresponding to a canonical decoration of the tree $\mT^{\bullet, \cM}$. Let $\mM_1^\bullet$ and $\mM_2^\bullet$ denote independent copies of $\mM^\bullet$. We colour the marked corner of $\mM_1^\bullet$ red, and the marked corner of $\mM_2^\bullet$ blue. Note that  $|\mT^{\bullet \cM}| = \co(\mM^\bullet)+1$. If the marked vertex of $\mT^{\bullet \cM}$ coincides with its root-vertex, then we view $\mM^\bullet$ as marked at a placeholder location.

 Let $j_1, j_2$ be a uniformly selected pair of distinct integers between $2$ and $\Delta(\mT_n^\cM) - t_n$. Furthermore, for each $i \ge 1$ set $\tilde{\mM}(i) = \mM_k^\bullet$ if $i=j_k$, $k\in\{1,2\}$, and $\tilde{\mM}(i) = \mM(i)$ otherwise. It follows by Corollary~\ref{co:sizebias} that
\begin{align}
\label{eq:reallyawesome}
\left( \bar{\cM}_i(\mM_n^t) \right)_{1 \le i \le \Delta(\mT_n^\cM)-t_n} \atv \left(\tilde{\mM}(i)\right)_{1 \le i \le \Delta(\mT_n^\cM)-t_n}.
\end{align}

Let $r_1$ and $r_2$ denote fixed non-negative integers. By Proposition~\ref{pro:sparse} and Lemma~\ref{le:quenchmap} it follows that the neighbourhoods $U_{r_1}(\mM_n^t, v_1)$ and $ U_{r_2}(\mM_n^t, v_2)$ are with high probability disjoint. Repeated application of Proposition~\ref{pro:sparse} also entails that we may select $(t_n)_n$ to converge sufficiently slowly to infinity such that with high probability neither of these neighbourhoods contains the first or any of the  last $t_n$ corners of $\cV(\mM_n^t)$ (with respect to the canonical ordering of its corners).

Let $M_1$ and $M_2$  be finite corner rooted planar maps with radii  $r_1$ and $r_2$. For $k =1,2$ let $v_k'$ be the corner of $\cV(\mM_n^t)$ where the component containing $v_k$ is attached. If the distance between $v_k$ and $v_k'$ is at least $r_k$, then $U_{r_k}(\mM_n^t, v_k)$ is fully contained in the component containing $v_k$. If the distance equals some $h <  r_k$, then $U_{r_k}(\mM_n^t, v_k)$ may be patched together from the $r_k$-neighbourhood of $v_k$ in that component (with additional knowledge of the location of $v_k'$ within that neighbourhood), the neighbourhood $U_{r_k-h}(\cV(\mM_n^t), v_k')$, and neighbourhoods in the components attached to corners $c \in U_{r_k-h}(\cV(\mM_n^t), v_k') \setminus \{v_k'\}$ with distance less than $r_k-h$ from $v_k'$. By Equation~\eqref{eq:reallyawesome} and the observations made in the penultimate paragraph, we know that jointly and asymptotically the component containing $v_k$ behaves like $\mM^\bullet$, and the components attached to the corners $c$ behave like independent copies of a map $\mM$ corresponding to a canonical decoration of $\mT^\cM$. This allows us to write
\begin{align}
	\label{eq:recur}
	\Pr{U_{r_k}(\mM_n^t, v_k)=M_k} = o(1) +  C_{r_k}(M_k)  + \sum_{h=0}^{r_k-1} p_h \sum_H c_{h,H} \Pr{U_{r_k-h}(\cV(\mM_n^t), v_k') = H}.
\end{align}
Here  $C_{r_k}(M_k)$ denotes the probability for the event that jointly the distance between the  root corner and marked corner in $\mM^\bullet$ is at least $r_k$, and that the $r_k$-neighbourhood of the marked corner in $\mM^\bullet$ equals~$M_k$. The constant $p_h>0$ denotes the probability solely for the event that this distance in $\mM^\bullet$ equals $h$. The constant $c_{h,H}$ represents the sum of product  probabilities for the finitely many ways of patching $M_k$ together out of a rooted $r_k$-neighbourhood in $\mM^\bullet$ (conditioned on having distance $h$ between the roots),  and neighbourhoods in independent $\mM$-distributed components attached to non-root corners $c$ of the sum index map $H$. Note that the sum index $H$ ranges over specific submaps of $M_k$, and the case $H=M_k$ occurs only once and for $h=0$. Specifically, letting $\emptyset$ denote the empty map with no vertices at all,
\begin{align}
	\label{eq:c}
	c_{0,M_k} = \Pr{\mM^\bullet= \emptyset } \Pr{\mM= \emptyset}^{s(M_k)-1}>0,
\end{align}
with $s(M_k)$ denoting the number of corners in $M_k$ with distance less than $r_k$ from the root-corner.
The case $H=M_k$ and $h=0$ is the only summand on the right-hand side of Equation~\eqref{eq:recur} where  $|H| + (r_k-h)$ attaines its maximum. As the left-hand side of Equation~\eqref{eq:recur} converges by Lemma~\ref{le:quenchmap},  it follows by induction on $r_k + |M_k|$ that the probability $\Pr{U_{r_k}(\cV(\mM_n^t), v_k')=M_k}$ converges to some constant $p_{r_k, M_k}$. (The base case $r_k=0$ is trivial.) 

In order to deduce distributional convergence of the neighbourhood it remains to verify  $\sum_{M_k} p_{{r_k}, M_k} = 1$.
Suppose that $1 - \sum_{M_k} p_{r_k, M_k} =: \epsilon >0$. Then for any $s>0$ 
\[
\Pr{\co(U_{r_k}(\cV(\mM_n^t), v_k')) > s} = 1 - \sum_{M_k, \co(M_k) \le s} \Pr{U_{r_k}(\cV(\mM_n^t), v_k') = M_k}  \to 1 - \sum_{M_k, \co(M_k) \le s} p_{{r_k}, M_k} \ge \epsilon.
\]
This implies that there is a sequence $s_n \to \infty$ with $\Pr{\co(U_{r_k}(\cV(\mM_n^t), v_k')) > s_n} \ge \epsilon/2$ for all $n$. As the distance $d_{\mM_n^t}(v_k, v_k')$ between the corners $v_k$ and $v_k'$ admits the limit distribution $(p_h)_{h \ge 0}$,  this implies
\begin{align*}
	\Pr{\co(U_{r_k}(\mM_n^t), v_k) > s_n} \ge \Pr{d_{\mM_n^t}(v_k, v_k') = 0} \Pr{\co(U_{r_k}(\cV(\mM_n^t), v_k')) > s_n} 
	\ge  p_0 \epsilon + o(1).
\end{align*}
As $p_0 >0$, this contradicts the distributional convergence of $U_{r_k}(\mM_n^t, v_k)$. Consequently, it must hold that $\sum_{M_k} p_{{r_k}, M_k} = 1$.

Summing up, there is random infinite graph $\hat{\mV}^t$ which is the distributional limit of $\cV(\mM_n^t)$ rooted according to the stationary distribution, and satisfies
\begin{align}
	\label{eq:ing1}
	\Pr{U_{r_k}(\hat{\mM}^t) = M_k} = C_{r_k}(M_k) + \sum_{h=0}^{r_k-1} p_h \sum_H c_{h,H} \Pr{U_{r_k-h}(\hat{\mV}^t) = H}.
\end{align}
Using the fact that with high probability $U_{r_1}(\mM_n^t, v_1)$ and $U_{r_2}(\mM_n^t, v_2)$ do not intersect,  we obtain analogously as for  Equation~\eqref{eq:recur} that
\begin{align}
	\label{eq:recurjoint}
	&\Pr{U_{r_1}(\mM_n^t, v_1)=M_1, U_{r_2}(\mM_n^t, v_2)=M_2} = o(1) + C_{r_1}(M_1)C_{r_2}(M_2) \\
	&+ C_{r_1}(M_1)\sum_{h=0}^{{r_2}-1} p_h \sum_{H_2} c_{h,H_2} \Pr{U_{r_2-h}(\cV(\mM_n^t), v_2') = H_2} \nonumber \\
	&+ C_{r_2}(M_2)\sum_{h=0}^{r_1-1} p_h \sum_{H_1} c_{h,H_1} \Pr{U_{r_1-h}(\cV(\mM_n^t), v_1') = H_1} \nonumber \\
	&+ \sum_{\substack{0 \le h_1 < r_1 \\ 0 \le  h_2 < r_2}} p_{h_1}p_{h_2} \sum_{H_1, H_2} c_{h_1,H_1} c_{h_2,H_2} \Pr{U_{r_1-h_1}(\cV(\mM_n^t), v_1') = H_1, U_{r_2-h_2}(\cV(\mM_n^t), v_2') = H_2}.\nonumber
\end{align}
By Lemma~\ref{le:quenchmap} and Proposition~\ref{prop:char} we know that the left-hand side of this equation satisfies
\begin{align}
\label{eq:ing2}
\Pr{U_{r_1}(\mM_n^t, v_1)=M_1, U_{r_2}(\mM_n^t, v_2)=M_2} \to \Pr{U_{r_1}(\hat{\mM}^t) = M_1}\Pr{U_{r_2}(\hat{\mM}^t) = M_2}.
\end{align}  Using convergence of the marginals $\Pr{U_{r_k-h}(\cV(\mM_n^t), v_k') = H_k}$ and a similar inductive argument as before it follows that the joint probability $\Pr{U_{r_1}(\cV(\mM_n^t), v_1') = M_1, U_{r_2}(\cV(\mM_n^t), v_2') = M_2}$ converges to some constant $p_{r_1, r_2, M_1, M_2}$. Using Equations~\eqref{eq:ing1} and \eqref{eq:ing2} it follows that
\begin{multline*}
\sum_{\substack{0 \le h_1 < r_1 \\ 0 \le  h_2 < r_2}} p_{h_1}p_{h_2} \sum_{H_1, H_2} c_{h_1,H_1} c_{h_2,H_2} \Pr{U_{r_1-h_1}(\hat{\mV}^t) = H_1}\Pr{U_{r_2-h_2}(\hat{\mV}^t) = H_2} = \\
\sum_{\substack{0 \le h_1 < r_1 \\ 0 \le  h_2 < r_2}} p_{h_1}p_{h_2} \sum_{H_1, H_2} c_{h_1,H_1} c_{h_2,H_2} p_{r_1, r_2, M_1, M_2}.
\end{multline*}
Hence, again by induction (on $r_1 + r_2 + \co(M_1) + \co(M_2)$, with the base case being trivial) 
\begin{align}
	p_{r_1, r_2, M_1, M_2} = \Pr{U_{r_1}(\hat{\mV}^t) = M_1}\Pr{U_{r_2}(\hat{\mV}^t) = M_2}.
\end{align}
This verifies that if $c_n^{(1)}$ and $c_n^{(2)}$ are uniform independent corners of $\cV(\mM_n^t)$, then
\begin{align}
	\left( \left(\cV(\mM_n^t), c_n^{(1)}\right), \left(\cV(\mM_n^t), c_n^{(2)}\right)\right) \convd \left(\hat{\mV}^{t, (1)}, \hat{\mV}^{t, (2)}\right),
\end{align}
with $\hat{\mV}^{t, (1)}, \hat{\mV}^{t, (2)}$ denoting independent copies of $\hat{\mV}^t$. Hence by Proposition~\ref{prop:char}
	\begin{align}
\Pr{ (\cV(\mM_n^t), c_n) \mid \cV(\mM_n^t)} \convp \mfL(\hat{\mV}^t).
\end{align}
\end{proof}

\subsection{$\bar{\cK}$-networks}

Let $\mD_n^t$ and $\bar{\mK}_n^t$ denote random $\cD$- and $\bar{\cK}$-networks with $n$ edges, drawn with probability proportional to their weight (given by $t^{\ve(\cdot)}$).  The relation $\cV(x,z) = 1 + z^2 + z^2x\cD(x,z^2)$ from Equation~\eqref{eq:relVD} entails that, for $n\ge 2$, $\mV_n$ may be sampled by converting the ``invisible'' root-edge of $\mD_{n-1}$ into a regular one.

The Gibbs partition $\cD= \bar{\cK} \Seq(x \bar{\cK})$  from Equation~\eqref{eq:mapgibb} represents the fact that any $\cD$-network consists of a series composition of a positive number of $\bar{\cK}$-networks.  The discussion in Section~\ref{eq:enrichedmaps} entails that $\cD= \bar{\cK} \Seq(x \bar{\cK})$ has convergent type. Hence, identifying $\cD$-networks with sequences of $\bar{\cK}$-networks,
\begin{align}
	\label{eq:Dtoseq}
	\mD_n^t \atv \left(\bar{\mK}(1), \ldots, \bar{\mK}(N), \bar{\mK}_{n - A}^t, \bar{\mK}'(1), \ldots, \bar{\mK}'(N')\right), 
\end{align}
with
\begin{align}
 A := \sum_{i=1}^{N} \ed(\bar{\mK}(i)) +  \sum_{i=1}^{N'} \ed(\bar{\mK}'(i)).
\end{align}
Here $N$ and $N'$ denote independent identically distributed non-negative integers with geometric distribution 
\begin{align}
	\Pr{N = k} = \bar{\cK}(t, \rho_{\bar{\cK}})^k ( 1 - \bar{\cK}(t,\rho_{\bar{\cK}})), \qquad k \ge 0.
\end{align}
The networks $\bar{\mK}(i)$ and  $\bar{\mK}'(i)$, $i \ge 1$, denote independent copies of a random $\bar{\cK}$-network $\bar{\mK}$ with distribution given by
\begin{align}
	\label{eq:fin}
	\Pr{\bar{\mK} = \bar{K}} = t^{\ve(\bar{K})} \rho_{\bar{\cK}}^{\ed(\bar{K})} / \bar{\cK}(t,\rho_{\bar{\cK}}).
\end{align}

Equation~\eqref{eq:Dtoseq} tells us that a large $\cD$-network has with high probability a giant $\bar{\cK}$-component. We let $\bar{\cK}(\mM_n^t)$ denote the with high probability unique  $\bar{\cK}$-core  of the $\cD$-network $\cD(\mM_n^t)$ corresponding to the non-separable core $\cV(\mM_n^t)$. 

\begin{corollary}
	\label{co:doitf}
	Lemma~\ref{le:map1to2} holds analogously for $\bar{\cK}(\mM_n^t)$, when we treat the ``invisible'' root-edge of $\bar{\cK}(\mM_n^t)$ as a real one. 
\end{corollary}
\begin{proof}
	Lemma~\ref{le:map1to2} and the fourth characterization of quenched local convergence in Proposition~\ref{prop:conv} tells us that for any $r\ge 0$ the percentage of corners whose $r$-neighbourhoods has a fixed shape $M$ in $\cV(\mM_n^t)$ concentrates around the limit probability $\Pr{U_r(\hat{\mV}^t) = M}$. Equation~\eqref{eq:Dtoseq} and Proposition~\ref{pro:sparse} entail that all but a stochastically bounded number of these corners have  the property that their $M$-shaped $r$-neighbourhood lies entirely in $\bar{\cK}(\mM_n^t)$. Using again Proposition~\ref{prop:conv}, it follows that $\hat{\mV}^t$ is also the quenched local limit of $\bar{\cK}(\mM_n^t)$. 
\end{proof}

\begin{corollary}
	\label{co:tillicollapse}
	Uniformly for $\ell \in \ndZ$
	\begin{align}
		\label{eq:tillicollapse}
			\Pr{\ed(\bar{\cK}(\mM_n^t)) = \ell} = \frac{1}{g_{\cM}(t) n^{2/3}}\left(h\left(\frac{ (1-\Ex{\xi^\cM})n - \ell}{g_{\cM}(t) n^{2/3}}   \right) + o(1)\right).
	\end{align}	
\end{corollary}
\begin{proof}
	As $\lim_{x \to \infty} h(x) = 0$, it suffices to show this for integers $\ell \ge 0$. The density function $h$ is bounded and uniformly continuous. Hence for any $k \ge 0$ it holds by Equations \eqref{eq:sepcoresize} and \eqref{eq:Dtoseq} that
	\[
		g_{\cM}(t) n^{2/3} \Prb{\ed(\bar{\cK}(\mM_n^t)) = \ell , \ed(\bar{\cD}(\mM_n^t)) = \ell+k} = h\left(\frac{ (1-\Ex{\xi^\cM})n - \ell}{g_{\cM}(t) n^{2/3}}   \right)\Pr{A=k} + o(1),
	\]
	with the $o(1)$ term being uniform in $\ell$.
	 Hence for any sequence $t_n \to \infty$ tending sufficiently slowly to infinity
	 \begin{align}
	 	\label{eq:drdre}
	 		g_{\cM}(t) n^{2/3} \sum_{k =0}^{t_n} \Prb{\ed(\bar{\cK}(\mM_n^t)) = \ell , \ed(\cD(\mM_n^t)) = \ell +k} = h\left(\frac{ (1-\Ex{\xi^\cM})n - \ell}{g_{\cM}(t) n^{2/3}}   \right) + o(1),
	 \end{align}
	 with an uniform $o(1)$ term. Equation~\eqref{eq:dooo} implies that
	 	$\Pr{\ed(\bar{\mK})=n} \sim c_{\bar{\cK}}(t) n^{-5/2}$.
	 Using Equation \eqref{eq:sepcoresize},  Proposition~\ref{prop:gibbsbound}, and the fact that $h$ is bounded it follows that there are constants $C,c>0$ such that
	\begin{align*}
		g_{\cM}(t) n^{2/3} &\sum_{t_n \le k \le n} \Pr{\ed(\bar{\cK}(\mM_n^t)) = \ell , \ed(\cD(\mM_n^t)) = \ell+k} \\
			&\le  C \sum_{t_n \le k \le n} \left(h\left(\frac{ (1-\Ex{\xi^\cM})n - \ell-k}{g_{\cM}(t) n^{2/3}}   \right) + o(1)\right)
			 \frac{\Pr{\ed(\bar{\mK})=\ell} \Pr{\ed(\bar{\mK})=k} }{\Pr{\ed(\bar{\mK})=\ell+k}}\exp\left(-c\frac{k}{\ell}\right) \\
			 &\le O(\Pr{ \ed(\bar{\mK}) \ge t_n}).
	\end{align*}
	Together with Equation~\eqref{eq:drdre} this verifies Claim~\eqref{eq:tillicollapse}.
\end{proof}

\subsection{$\bar{\cR}$-networks}
Equation~\eqref{eq:decompKbar}, that is $\bar{\cK} = y \bar{\cR}(x, \bar{\cK})$, tells us that a $\bar{\cK}$-network may be recursively described as an $\bar{\cR}$-network where we insert an additional edge (corresponding to the factor $y$) at a specified location, and substitute all other regular edges (if there are any) by $\bar{\cK}$-networks.

It follows from Equation~\eqref{eq:decompKbar} and the general results of Section~\ref{sec:enriched} that any $\bar{\cK}$-network with $n$ edges corresponds bijectively to a pair $(T, \beta)$ of a planted plane tree $T$ with $n$ vertices and a function $\beta$ that assigns to each vertex $v \in T$ an $\bar{\cR}$-network with $d^+_T(v)$ regular edges.

The bijection may be recursively described as follows: The decoration $\beta(o)$ of the root vertex $o$ of $T$ is a network with an ``invisible'' edge connecting the poles, $d^+_T(o)$ regular edges and an additional ``terminal'' edge that we label with $o$. The regular edges may be enumerated in some canonical way (let's say in a breadth-first-search manner), so that each corresponds bijectively to one of the $d^+_T(o)$ offspring vertices $v$ of $o$. The decorated fringe subtree rooted at such a vertex $v$ corresponds bijectively to a network, and we replace the edge of $\beta(o)$ that corresponds to $v$ by the ``invisible'' root-edge of that network, hence creating the final network by a glueing operation.

The enriched tree $(\mT_n^{\bar{\cK}}, \beta_n^{\bar{\cK}})$ corresponding to a random $\bar{\cK}$-network $\bar{\mK}_n^t$ with $n$ edges (drawn with probability proportional to its weight $t^{\ve(\cdot)}$) is a canonically decorated simply generated tree with weight-sequence $(\omega_k^{\bar{\cK}})_{k \ge 0}$ given by
\begin{align}
	\omega_k^{\bar{\cK}} = [y^k]\bar{\cR}(t,y).
\end{align}

It follows from Inequality~\eqref{eq:Kbarsubcrit}, Equation~\eqref{eq:Rbarasympt}, and Lemma~\ref{le:simplygen} that $\mT_n^{\bar{\cK}}$ is distributed like a Galton--Watson tree $\mT^{\bar{\cK}}$ conditioned on having  $n$ vertices, with offspring distribution $\xi^{\bar{\cK}}$ satisfying
\begin{align}
\label{eq:refnowme}
\Ex{\xi^{\bar{\cK}}} < 1 \qquad \text{and} \qquad \Pr{\xi^{\bar{\cK}} = n} \sim \frac{c_{\bar{\cR}}(t)}{\bar{\cR}(t, \rho_\cR(t))} n^{-5/2}.
\end{align}

By Lemma~\ref{le:maxllt} it follows that the (with high probability unique) largest $\bar{\cR}$-component $\bar{\cR}(\bar{\mK}_n^t)$ satisfies
\begin{align}
\label{eq:lltRt}
\Pr{\ed(\bar{\cR}(\bar{\mK}_n^t)) = \ell} = \frac{1}{g_{\bar{\cK}}(t) n^{2/3}}\left(h\left(\frac{ (1-\Ex{\xi^{\bar{\cK}}})n - \ell}{g_{\bar{\cK}}(t) n^{2/3}}   \right) + o(1)\right)
\end{align}
uniformly for all $\ell \in \ndZ$ with $g_{\bar{\cK}}(t) >0$ a constant.

We let $\bar{\cR}(\mM_n^t) := \bar{\cR}(\bar{\cK}(\mM_n^t))$ denote the largest $\bar{\cR}$-component within the largest $\bar{\cK}$-component of $\cV(\mM_n^t)$. Note the subtlety of this definition. It is clear that with probability tending to $1$ the component $\bar{\cR}(\mM_n^t)$ actually equals the largest $\bar{\cR}$-component within the entire map $\mM_n^t$. However, determining the speed of convergence is something that would require additional work. For this reason, getting a local limit theorem for the size $\ed(\bar{\cR}(\mM_n^t))$ requires a little less effort than getting a local limit theorem for the size of the largest $\bar{\cR}$-component in $\mM_n^t$.

A local limit theorem for the size of the largest $3$-connected component of $\mM_n^t$  was established by \cite[Thm. 6.4]{MR3068033}, and we are going to argue analogously.

\begin{corollary}
	\label{co:Rbarsize}
	It holds uniformly for all $\ell \in \ndZ$ 
\begin{align}
	\label{eq:getthelimit}
	\Pr{\ed(\bar{\cR}(\mM_n^t)) = \ell} = \frac{1}{\tilde{g}_{\bar{\cK}}(t) n^{2/3}}\left(h\left(\frac{ (1-\Ex{\xi^{\bar{\cK}}})(1-\Ex{\xi^{{\cM}}})n - \ell}{\tilde{g}_{\bar{\cK}}(t) n^{2/3}}   \right) + o(1)\right),
\end{align}
with
\begin{align}
	\tilde{g}_{\bar{\cK}}(t) = \left(\left(g_{\cM}(t)(1-\Ex{\xi^{\bar{\cK}}})\right)^{3/2} + g_{\bar{\cK}}(t)^{3/2}(1 - \Ex{\xi^\cM}) \right)^{2/3}.
\end{align}
\end{corollary}
\begin{proof}
	It follows from the definition of $\bar{\cR}(\mM_n^t)$ that 
	\begin{align}
		\Pr{\ed(\bar{\cR}(\mM_n^t)) = \ell} &= \sum_{N=1}^n \Pr{ \ed(\bar{\cK}(\mM_n^t)) = N}\Pr{\ed(\bar{\cR}(\bar{\mK}_N^t)) = \ell}.
	\end{align}
	Recall that $\bar{\cK}(\mM_n^t) \eqdist \frac{1}{2} \Delta(\mT_n^\cM)$ (and that $\mT_n^\cM$ has $2n+1$ vertices). It follows by \cite[Eq. (2.10)]{2019arXiv190104603S} that there is a constant $\epsilon_1>0$ such that 
	\begin{align}
		\Pr{\Delta(\mT_n^\cM) \le \epsilon_1 n / \log n} = o(n^{-2/3}).
	\end{align}
	By \cite[Eq. (2.11), Eq. (2.12)]{2019arXiv190104603S} it follows that for any $0< \epsilon_2 < 1-\Ex{\xi^{{\cM}}}$
	\begin{align}
		\Pr{\epsilon_1 n / \log n \le \Delta(\mT_n^\cM) \le \epsilon_2 n} = O(n^{-2/3}).
	\end{align}
	 Equation~\eqref{eq:lltRt} implies that  $\Pr{\ed(\bar{\cR}(\bar{\mK}_N^t)) = \ell} = O(N^{-2/3})$, hence 
	\begin{align*}
		\sum_{N = 1}^{\lfloor \epsilon_2 n \rfloor} \Pr{ \ed(\bar{\cK}(\mM_n^t)) = N}\Pr{\ed(\bar{\cR}(\bar{\mK}_N^t)) = \ell} = o(n^{-2/3}).
	\end{align*}
	Hence by Corollary~\ref{co:tillicollapse} it follows that for each $\epsilon >0$ we may select a constant $M_1>0$ large enough such that the interval $I_n:= (1-\Ex{\xi^{{\cM}}})n \pm M_1 n^{2/3}$ satisfies
	\begin{align}
	\label{eq:doit1}
	\limsup_{n \to \infty} \sup_{\ell \in \ndZ} n^{2/3} \sum_{N \notin I_n } \Pr{ \ed(\bar{\cK}(\mM_n^t)) = N}\Pr{\ed(\bar{\cR}(\bar{\mK}_N^t)) = \ell} \le \epsilon, \quad 
	\end{align}
	Using again Equation~\eqref{eq:lltRt}, it follows that there is a constant $M_2$ (depending only on $\epsilon$) such that the interval  $J_n := (1-\Ex{\xi^{\bar{\cK}}})(1-\Ex{\xi^{{\cM}}})n \pm M_2 n^{2/3}$ satisfies
	\begin{align}
	\label{eq:doit2}
	\limsup_{n \to \infty} \sup_{\ell \in J_n} n^{2/3} \Pr{\ed(\bar{\cR}(\mM_n^t)) = \ell} \le 2\epsilon.
	\end{align}
	
	Hence it suffices to verify Equation~\eqref{eq:getthelimit} uniformly for $\ell \in I_n$ as $n$ becomes large. We may write $N \in I_n$ as $N = (1-\Ex{\xi^{{\cM}}})n + x_N  n^{2/3}$ with $|x_N| \le M_1$. Likewise, we may write $\ell \in J_n$ as $\ell= (1-\Ex{\xi^{\bar{\cK}}})(1-\Ex{\xi^{{\cM}}})n + y_\ell n^{2/3}$ with $|y_\ell| \le M_2$.  Equations~\eqref{eq:tillicollapse}, \eqref{eq:lltRt}, and the fact that $h$ is bounded and uniformly continuous imply that uniformly for $\ell \in J_n$
	\begin{multline*}
		n^{2/3} \sum_{N \in I_n} \Pr{ \ed(\bar{\cK}(\mM_n^t)) = N}\Pr{\ed(\bar{\cR}(\bar{\mK}_N^t)) = \ell} \\
		= \frac{1+o(1)}{g_{\cM}(t) g_{\bar{\cK}}(t) (1 - \Ex{\xi^\cM})^{2/3}n^{2/3}}\sum_{N \in I_n}  h\left(\frac{-x_N}{g_{\cM}(t)} \right)h\left(\frac{(1 - \Ex{\xi^{\bar{\cK}}})x_N - y_\ell  }{ g_{\bar{\cK}}(t)(1 - \Ex{\xi^\cM})^{2/3}} \right).
	\end{multline*}
	Taking $M_1$ large enough and using $|y_\ell| < M_2$, it follows that this expression lies in the interval
	\begin{align*}
	 \pm \epsilon + o(1) + \frac{1}{g_{\cM}(t) g_{\bar{\cK}}(t) (1 - \Ex{\xi^\cM})^{2/3}} \int_{-\infty}^{\infty}  h\left(\frac{-z}{g_{\cM}(t)} \right)h\left(\frac{(1 - \Ex{\xi^{\bar{\cK}}})z - y_\ell  }{ g_{\bar{\cK}}(t)(1 - \Ex{\xi^\cM})^{2/3}} \right) \mathrm{d}z.
	\end{align*}Setting $a=g_{\cM}(t)(1-\Ex{\xi^{\bar{\cK}}})$ and $b = g_{\bar{\cK}}(t)(1 - \Ex{\xi^\cM})^{2/3}$, and making a linear variable transform, we may rewrite the last summand by
	$
		\int_{-\infty}^\infty \frac{1}{a}h\left(-\frac{z}{a}\right) \frac{1}{b}h\left(\frac{z - y_\ell}{b}\right) \mathrm{d}z.
	$ Recall that $h$ is the density of a $3/2$-stable random variable $X$ with Laplace transform $\Ex{e^{-\lambda X}} = e^{\lambda^{3/2}}$. Hence the integral is the density of a sum $-aX_1 -bX_2$ evaluated at the point $y_\ell$, with $X_1$ and $X_2$ denoting independent copies of $X$. By comparing Laplace transforms it holds that $-aX_1 -bX_2 \eqdist -cX$ with $c = \left(a^{3/2} + b^{3/2} \right)^{2/3}$. Thus
	\[
	\int_{-\infty}^\infty \frac{1}{a}h\left(-\frac{z}{a}\right) \frac{1}{b}h\left(\frac{z - y_\ell}{b}\right) \mathrm{d}z = \frac{1}{c} h\left( -\frac{y_\ell}{c} \right) = \frac{1}{c} h\left(\frac{ (1-\Ex{\xi^{\bar{\cK}}})(1-\Ex{\xi^{{\cM}}})n - \ell}{c n^{2/3}}   \right).
	\]
\end{proof}

\begin{lemma}
	\label{le:convRbar}
	The $\bar{\cR}$-component $\bar{\cR}(\mM_n^t)$ admits a distributional limit~$\hat{\bar{\mR}}^t$ in the local topology. Letting  $c_n$ denote a uniformly selected corner of $\bar{\cR}(\mM_n^t)$, it holds that
	\begin{align}
	\Pr{ (\bar{\cR}(\mM_n^t), c_n) \mid \bar{\cR}(\mM_n^t)} \convp \mfL(\hat{\bar{\mR}}^t).
	\end{align}
\end{lemma}


\begin{proof}
	The (regular and ``invisible'') edges of the core $\bar{\cR}(\bar{\mK}_n^t)$ (that is, the canonically selected  largest $\bar{\cR}$-component of $\bar{\mK}_n^t$) may be enumerated from $0$ to $\ed(\bar{\cR}(\bar{\mK}_n^t))$ in a canonical way, starting with the invisible edge. We give each edge an orientation according to a fair independent coin flip. The map $\bar{\mK}_n^t$ is constructed from the core $\bar{\cR}(\bar{\mK}_n^t)$ by replacing the $i$th (oriented) edge by a network $\bar{\cK}_i(\bar{\mK}_n^t)$ for all $0 \le i \le \ed(\bar{\cR}(\bar{\mK}_n^t))$. Here \emph{replacing} means deleting the edge and identifying its start vertex with the south pole and its end vertex with the north pole of the network. 
	
	 The network $\bar{\cK}_0(\bar{\mK}_n^t)$ inserted at the ``invisible'' edge of the core $\bar{\cR}(\bar{\mK}_n^t)$ carries a second pair of poles that correspond to the poles of~$\bar{\mK}_n^t$. For each $1 \le i \le \ed(\bar{\cR}(\bar{\mK}_n^t))$ the network $\bar{\cK}_i(\bar{\mK}_n^t)$ is fully described by the fringe subtree $F_i(\mT_n^{\bar{\cK}})$ and the restriction $\beta_n^{\bar{\cK}}|_{F_i(\mT_n^{\bar{\cK}})}$. The network $\bar{\cK}_i(\bar{\mK}_n^t)$ (and its second pair of poles) is fully described by the marked tree $F_0(\mT_n^{\bar{\cK}})$ and the restriction of $\beta_n^{\bar{\cK}}$ to all unmarked vertices of~$F_0(\mT_n^{\bar{\cK}})$.

	We define $\mT^{\bar{\cK}}$, $\mT^{\bullet \bar{\cK}}$ and $\mT^{\circ \bar{\cK}}$  for the offspring distribution $\xi^{\bar{\cK}}$ analogously as $\mT$, $\mT^\bullet$ and $\mT^\circ$ were defined for $\xi$ in Section~\ref{sec:condens}. For each integer $i \ge 1$ we let $\bar{\mK}(i)$ denote an independent copy of  the network $\bar{\mK}$ corresponding to $\mT^{\bar{\cK}}$ with canonical random $\bar{\cR}$-decorations (as defined in Section~\ref{sec:enriched}). We let $\bar{\mK}(0)$ denote the network (with two sets of poles) corresponding to a canonical decoration of the random  marked tree $\mT^{\circ \bar{\cK}}$. It follows from Lemma~\ref{le:fringe} that there is a constant $C>0$ such that for any sequence of integers $(t_n)_n$ with $t_n\to \infty$ and $t_n=o(n)$
	\begin{align}
	\label{eq:Kattached1}
	\left( \bar{\cK}_i(\bar{\mK}_n^t) \right)_{0 \le i \le \ed(\bar{\cR}(\bar{\mK}_n^t))-t_n} \atv \left(\bar{\mK}(i)\right)_{0 \le i \le \ed(\bar{\cR}(\bar{\mK}_n^t))-t_n}
	\end{align}
	and with probability tending to $1$ as $n$ becomes large
	\begin{align}
	\label{eq:Kattached2}
	\sum_{i=\ed(\bar{\cR}(\bar{\mK}_n^t))-t_n}^{\ed(\bar{\cR}(\bar{\mK}_n^t))} \ed(\bar{\cK}_i(\mM_n^t)) \le C t_n.
	\end{align}

	The corners of $\bar{\mK}_n^t$ (counting the ``invisible'' edge between the poles as a real edge) correspond bijectively to the corners of the collection $(\bar{\cK}_i(\mK_n^t))_{0 \le i \le \ed(\bar{\cR}(\bar{\mK}_n^t))}$ (treating the ``invisible'' edge between the poles of $\bar{\cK}_0(\mK_n^t)$, that correspond to the poles of $\bar{\mK}_n^t$, like a real edge). Let us select a red corner $v_1$ and a blue corner $v_2$  of $\bar{\mK}_n^t$ uniformly and independently at random. This may be done by uniformly selected two independent edges, and flipping fair coins for each to determine which of the corresponding half-edges to use. 
	For each $0 \le i \le \ed(\bar{\cR}(\bar{\mK}_n^t))$ we set   $\tilde{\bar{\cK}}_i(\bar{\mK}_n^t) = {\bar{\cK}}_i(\bar{\mK}_n^t)$ if neither $v_1$ nor $v_2$ lies in this component, and otherwise we let $\tilde{\bar{\cK}}_i(\bar{\mK}_n^t)$ be given by ${\bar{\cK}}_i(\bar{\mK}_n^t)$ with an additional marked red and/or blue corner corresponding to the location(s) of $v_1$ and/or $v_2$.
	
	We let $\bar{\mK}^\bullet$ denote the network with a marked corner obtained by taking the network corresponding to a canonical decoration of $\mT^{\bullet \bar{\cK}}$ and flipping a fair coin on which of the two half-edges corresponding to the marked edge to distinguish. We let $\bar{\mK}^\bullet_1$ and $\bar{\mK}^\bullet_2$ denote independent copies of $\bar{\mK}^\bullet$ where we colour the corners red and blue, respectively. Let $j_1, j_2$ denote a uniformly selected pair of distinct integers between $1$ and $\ed(\bar{\cR}(\bar{\mK}_n^t))-t_n$. For each $i \ge 0$ we set $\tilde{\bar{\mK}}_i^t = {\bar{\mK}}_i^t$ if $i \notin \{ j_1,j_2\}$, and $\tilde{\bar{\mK}}_i^t = \bar{\mK}^\bullet_k$ if $i = j_k$ for $k \in \{1,2\}$. By Corollary~\ref{co:sizebias} it follows that
	\begin{align}
	\label{eq:reallyawesomeKattached}
	\left( \tilde{\bar{\cK}}_i(\bar{\mK}_n^t) \right)_{0 \le i \le \ed(\bar{\cR}(\bar{\mK}_n^t))-t_n} \atv \left(\tilde{\bar{\mK}}(i)\right)_{0 \le i \le \ed(\bar{\cR}(\bar{\mK}_n^t))-t_n}.
	\end{align}

	Since $\ed(\bar{\cK}(\mM_n^t)) \convd \infty$ and $\bar{\cK}(\mM_n^t) \eqdist \bar{\mK}_{\ed(\bar{\cK}(\mM_n^t))}^t$, everything so far (in particular Equations~\eqref{eq:Kattached1},\eqref{eq:Kattached2} and \eqref{eq:reallyawesomeKattached}) holds analogously when we replace $\bar{\mK}_n^t$ by $\bar{\cK}(\mM_n^t)$. Hence we may write
	\begin{align}
	\label{eq:truedat444}
	\left( \tilde{\bar{\cK}}_i(\bar{\cK}(\mM_n^t)) \right)_{0 \le i \le \ed(\bar{\cR}(\mM_n^t))-t_n} \atv \left(\tilde{\bar{\mK}}(i)\right)_{0 \le i \le \ed(\bar{\cR}(\mM_n^t))-t_n},
	\end{align}
	with the implicitly used indices $j_1, j_2$ referring this time to a uniformly selected pair of distinct integers between $1$ and $\ed(\bar{\cR}(\mM_n^t))-t_n$.

	Let $r \ge 1$ be a constant.  For each $k\in\{1,2\}$ let $e_k$ be the \emph{oriented} edge of $\bar{\cR}(\mM_n^t)$ where the component containing $v_k$ is inserted. (Recall that we substitute edges of $\bar{\cR}(\mM_n^t)$ by networks, and  each edge of $\bar{\cR}(\mM_n^t)$  was given a direction so that there is no ambiguity which of the end vertices gets identified with which pole.)

	If the distance between $v_k$ and (both endpoints of) $e_k$ is at least~$r$, the neighbourhood $U_r(\bar{\mK}_n^t, v_k)$ is fully contained in the $\bar{\cR}$-component containing the corner~$v_k$.  If the distance $a_k$ of $v_k$ from the origin $e_k(1)$ and the distance $b_k$ of $v_k$ from the destination $e_k(2)$ satisfy $\min(a_k,b_k) < r$, then $U_r(\bar{\mK}_n^t, v_k)$ may be assembled canonically from the following parts:
	\begin{enumerate}
		\item The $r$-neighbourhood of the corner $v_k$ in the component containing it, with additional knowledge of the location of $e_k$ relative to that neighbourhood.
		\item The (connected) submap $U(k)$ of $\bar{\cR}(\bar{\mK}_n^t)$ (rooted at $e_k$) induced by all edges of $\bar{\cR}(\bar{\mK}_n^t)$ whose components contain edges from $U_r(\bar{\mK}_n^t, v_k)$. 
		\item Neighbourhood(s) of one or both poles (possibly with different radii) within the components inserted at edges $e \neq e_k$ in $U(k)$.
	\end{enumerate}
Let us define a \emph{semi-network} analogously as a planar network, with the only different requirement being that adding the ``invisible'' oriented root-edge must make the semi-network connected (instead of non-separable, as in the case of networks). The necessity for this notion stems from the fact that in the case $\min(a_k,b_k) < r$ it holds that if a network $K$ gets inserted at an edge $e \ne e_k$ of $U(k)$, then its contribution to (or intersection with)  the $r$-neighbourhood of $v_k$ in $\bar{\cK}(\mM_n^t)$ may have two different shapes: Either it consists of a neighbourhood in $K$ of only one of the poles (and the other is too far away.) Or it consists of the union of  neighbourhoods with possibly different radii of the south pole and north pole. These neighbourhoods may or may not overlap. Hence the need for this terminology, to describe how the $r$-neighbourhood of $v_k$ in $\bar{\cK}(\mM_n^t)$ gets assembled by inserting semi-networks at edges of $U(k)$.

It follows from Proposition~\ref{pro:sparse} that the neighbourhoods $U_r(\bar{\cK}(\mM_n^t), v_1)$ and $ U_r(\bar{\cK}(\mM_n^t), v_2)$ are with high probability disjoint. Applying Proposition~\ref{pro:sparse} repeatedly also entails that we may choose the sequence $(t_n)_n$ so that it converges sufficiently slowly to infinity such that with high probability neither of these neighbourhoods contains the ``invisible'' edge between the poles of $\bar{\cR}(\mM_n^t)$ or any of the  last $t_n$ edges of $\bar{\cR}(\mM_n^t)$ (with respect to the canonical ordering of its edges). (If there would exist a subsequence $(n')$ along which the probability for $v_1$ to have distance less than $r$ from the ``invisible'' edge or some $i$th last edge of $\bar{\cR}(\mM_n^t)$ is bounded away from zero, then so is the probability that this happens jointly for $v_1$ and $v_2$ (as they are i.i.d.), contradicting that $d_{\bar{\cK}(\mM_n^t)}(v_1, v_2) \convd \infty$.)

By Equation~\eqref{eq:truedat444} and the discussion of the previous paragraph it follows that jointly and asymptotically the components containing $v_1$ and $v_2$  behave like independent copies of $\bar{\mK}^\bullet$, and that  the components inserted at edges $\ne e_1, e_2$ of $U(1)$ and $U(2)$ behave jointly like independent copies of $\bar{\mK}$. It also follows that $U(1) \cap U(2) = \emptyset$ with high probability. (We are going to use this fact below when studying joint probabilities.)

We would like to establish an analogon of Equation~\eqref{eq:recur} and perform a similar proof by induction  as we did in the proof of Lemma~\ref{le:map1to2}, that uses convergence of the $r$-neighbourhood of $v_k$ in  $\bar{\cK}(\mM_n^t)$ (ensured by Corollary~\ref{co:doitf}) to deduce convergence of the  $r$-neighbourhood of $e_k$ in $\bar{\cR}(\mM_n^t)$. However, there is a problem if we work directly with probabilities for events that neighbourhoods of fixed radii have given shapes: The map $U_{r}(\bar{\cR}(\mM_n^t), e_k)$ may have more edges than $U(k)$, hence breaking the induction step.

For this reason, we are going to perform the induction with a different family of convergence-determining events. For any planar maps $R$ and $H$, any subset $A$ of vertices from $H$, and any half-edge $v$ of $R$ we let $\cE(H, A, R, v)$ denote the indicator variable (and, by abuse of notation, also the  corresponding event when $R$ is random) that $H$ may be embedded as a submap of $R$ with the root-edge corresponding to $v$, such that $R$ has no additional edges that are incident to~$A$. We are going to refer to the pair $(H, A)$ as a \emph{community}, and to the subset $A$ as the \emph{conservative} members of the community. (The analogy is that some community members are open to form new connections to others within and outside of their community, whereas members of the subset $A$ are more conservative.)

Let $(M_1, A_1)$ and $(M_2, A_2)$ be given finite communities, with $M_1$ and $M_2$ having radii $r_1, r_2 < r$. In the event $\cE(M_k, A_k, \bar{\cK}(\mM_n^t), v_k)$ there are finitely many possible shapes $H_k$ of the submap of the core $\bar{\cR}(\mM_n^t)$ induced by edges whose components contain edges of the embedding of $M_k$ in $\bar{\cK}(\mM_n^t)$. (Here we consider $H_k$ as rooted at the oriented edge corresponding to $e_k$.)  For example, if $H_k$ consists of single oriented edge, then the image of $M_k$ lies entirely in the component inserted at $e_k$. The marked corner in that component corresponds to the root-edge of $M_k$. Furthermore, the poles of this component may not correspond to conservative members. As the component inserted at $e_k$ asymptotically behaves like an independent copy of $\bar{\mK}^\bullet$, it follows that the limiting probability for this subevent is a constant $C(M_k,A_k)$ determined by the probability of some event for $\bar{\mK}^\bullet$. If $H_k$ consists of more than its root-edge, then $M_k$ gets assembled by substituting each edge $e$ of $H_k$ by a semi-network $K_e(k)$. The semi-network inserted at the root-edge of $H_k$ has a marked corner, that corresponds to the root corner of $M_k$. There may be multiple (but only finitely many) choices for such families $(K_e(k))_e$ for assembling $M_k$ in this way. If we know the shape $H_k$ then we know the subset $B_k$ of vertices of $H_k$ that correspond to conservative members of $M_k$, and if we additionally know  the family $(K_e(k))_e$ then for each $e$ we know the subset  of non-pole vertices of the inserted network $K_e$ that correspond to conservative members. Hence, this subevent is characterized by requiring the event $\cE(H_k,B_k, \bar{\cR}(\mM_n^t), e_k)$ to take place, and additionally for each edge $e$ of $H_k$ the corresponding component needs to have the semi-network $K_e(k)$ as sub-semi-network and no further edges incident to conservative members of $K_e(k)$. We know that jointly and  asymptotically the component corresponding to the root-edge of $H_k$ behaves like an independent copy of $\bar{\mK}^\bullet$,  and the components corresponding to all other edges behave like independent copies of $\bar{\mK}$. Hence the probability for the entire subcase corresponding to $H_k$ may be expressed by
\[
 o(1) + C_{(H_k,B_k)} \Pr{\cE(H_k,B_k, \bar{\cR}(\mM_n^t), e_k)}.
\]
Here $C_{(H_k,B_k)} \ge 0$ denotes a constant that corresponds to a sum (over all choices for the family $(K_e(k))_e$) of product probabilities for events of independent copies of $\bar{\mK}$ (and one copy of $\bar{\mK}^\bullet$).
As there are only finitely many choices for $H_k$, this allows to write
\begin{align}
\label{eq:recurdd}
	\Pr{ \cE(M_k, A_k, \bar{\cK}(\mM_n^t), v_k)} = o(1) + C(M_k,A_k) +  \sum_{\substack{(H_k, B_k) \\ \ed(H_k) \ge 2} } C_{(H_k,B_k)} \Pr{\cE(H_k,B_k, \bar{\cR}(\mM_n^t), e_k)}.
\end{align}
Here the sum index $(H_k, B_k)$  ranges over some finite set of communities, each having the properties $\ed(H_k) \ge 2$ and $\ve(H_k) \le \ve(M_k)$ and $\ed(H_k) \le \ed(M_k)$. There are at most two cases where jointly  $\ve(H_k) = \ve(M_k)$ and $\ed(H_k) = \ed(M_k)$: $H_k = M_k$ and  $H_k = \Inv{M_k}$, the map obtained by reversing the orientation of the root-edge of~$M_k$. (It is possible that $M_k = \Inv{M_k}$, entailing that these cases coincide.) 

Note that both  conditions $\ve(H_k) = \ve(M_k)$ and $\ed(H_k) = \ed(M_k)$ are necessary to nail the cases of $H_k$ down to $M_k$ / $\Inv{M_k}$.
There are potentially many more cases of $H_k$ with $\ed(H_k) = \ed(M_k)$, but all of them satisfy $\ve(H_k) < \ve(M_k)$.  For example, consider the case where $M_k$ is a path of length $4$ with the root-edge incident to and pointing away from the middle vertex. We could assemble $M_k$ by replacing the edges of a $4$-cycle by semi-networks - by replacing two of the square edges with a network consisting of two poles joined by a single regular edge (plus the ``invisible'' edge that we discard when substituting), and the other two by a semi-network where there is a single regular edge incident to one of the poles but not to the other. There are also potentially many more cases for $H_k$ with  $\ve(H_k) = \ve(M_k)$, but all of them satisfy $\ed(H_k) < \ed(M_k)$. For example when $M_k$ consists of two vertices joined by $3$ edges.

Let us focus on the special case $H_k = M_k$. We may assume that $M_k$ has at least two edges. 
For $H_k = M_k$ the semi-network $K_e(k)$ inserted at an edge $e$ of $H_k$ to form $M_k$ must have precisely one regular edge (plus the ``invisible'' edge that we discard when substituting). Hence there are three possible choices for $K_e(k)$.
Either the regular edge connects the two poles (we denote this by $*_S-*_N$), or it is only incident to the south pole ($*_S-$) or only to the north pole ($-*_N$). 
\begin{enumerate}
	\item If $e$ is an edge of $H_k$ with both ends having degree at least $2$, then $K_e(k)$ needs to equal $*_S-*_N$. If $e$ is the root-edge, then the edge of $K_e(k)$ needs to be oriented to point from south pole to north pole.
	\item If $e$ points from $e(1)$ to $e(2)$ such that $e(2)$ has degree $1$ (and hence $e(1)$ doesn't, since $M_k$ was assumed to have  at least two edges), then $K_e(k)$ may either be $*_S-*_N$ or $*_S-$. If $e$ is the root-edge, then the edge of $K_e(k)$ is oriented  to point away from the south pole.
	\item   If $e(1)$ has degree $1$ (and hence $e(2)$ doesn't), then $K_e(k)$ may either be $*_S-*_N$ or $-*_N$. If $e$ is the root-edge, then the edge of $K_e(k)$ is oriented and needs to point towards the north pole.
\end{enumerate}
This entails that \[
C_{(M_k, B_k)}=0 \quad \text{for} \quad  B_k \ne A_k.
\]
Indeed, if $B_k$ is a strict subset of $A_k$, then at least one of the networks $(K_e(k))_e$ has a conservative member. But the probability is zero for $\bar{\mK}$ or $\bar{\mK}^\bullet$ to have one of the three described shapes with an additional conservative member that may not be incident to further edges.
 It also follows that 
\begin{align}
\label{eq:positiveattitude}
 C_{(M_k, A_k)} >0.
\end{align}
Let $\Inv{e_k}$ denote the result of reversing the direction of $e_k$. By symmetry, it follows that
\begin{align*}
C_{(\Inv{M_k},B_k)} \Pr{\cE(\Inv{M_k},B_k, \bar{\cR}(\mM_n^t), e_k)} &= C_{(M_k,B_k)} \Pr{\cE(M_k,B_k, \bar{\cR}(\mM_n^t), \Inv{e_k})} \\
&= C_{(M_k,B_k)} \Pr{\cE(M_k,B_k, \bar{\cR}(\mM_n^t), e_k)}.
\end{align*}
This allows us to express Equation~\eqref{eq:recurdd} by
\begin{align}
\label{eq:recurddtuned}
\Pr{ \cE(M_k, A_k, \bar{\cK}(\mM_n^t), v_k)} = o(1)  &+  \sum_{(H_k, B_k) \in \mathfrak{C}(M_k,A_k) } C_{(H_k,B_k)} \Pr{\cE(H_k,B_k, \bar{\cR}(\mM_n^t), e_k)} \\
&+  C(M_k,A_k) + D_{(M_k, A_k)} \Pr{\cE(M_k,A_k, \bar{\cR}(\mM_n^t), e_k)}, \nonumber
\end{align}
with
\[
D_{(M_k, A_k)} = C_{(M_k, A_k)}(1 + \one_{M_k \ne \Inv{M_k}}).
\]
and $\mathfrak{C}(M_k,A_k)$ denoting a finite collection of communities $(H_k, B_k)$, all satisfying $\ed(H_k) \ge 2$, and $\ve(H_k) < \ve(M_k)$ or $\ed(H_k) < \ed(M_k)$.

The left-hand side of Equation~\eqref{eq:recurddtuned} converges by Corollary~\ref{co:doitf}. As $D_{(M_k, A_k)}>0$ by Equation~\eqref{eq:positiveattitude}, it follows by induction on $\ve(M_k) + \ed(M_k)$ (with the base case  being trivial) that there is a constant $p_{M_k, A_k} \ge 0$ with
\begin{align}
	\label{eq:dohhhh}
	\lim_{n \to \infty} \Pr{\cE(M_k,A_k, \bar{\cR}(\mM_n^t), e_k)} = p_{M_k, A_k}.
\end{align}

Given a planar map $M$ with a specified corner $c$, we define the  \emph{edge neighbourhood} $E_r(M,c)$ as the planar map (rooted at $c$) induced by all edges where at least one end-point has distance at most $r-1$ from $c$. Hence $E_r(M,c)$ may be obtained from $U_r(M,c)$ by removing all edges where both end-points have distance $r$ from $c$. It is clear that for any sequence $(\mX_n)_{n \ge 1}$ of random corner-rooted maps, weak convergence of $E_{r+1}(\mX_n)$ implies weak convergence of $U_r(\mX_n)$. Conversely, weak convergence of $U_r(\mX_n)$ implies weak convergence of $E_r(\mX_n)$.

Given a planar map $H$, the event $E_r(M,c)=H$ is equivalent to $\cE(H, U_{r-1}(H), M,c)$. Recalling that we assumed $M_k$ to have radius $r_k$, it follows from Equation~\eqref{eq:dohhhh} that 
\begin{align}
		\lim_{n \to \infty} \Pr{E_{r_k}(\bar{\cR}(\mM_n^t), e_k) = M_k} = p_{M_k, U_{r_k-1}(M_k)} =: p_{r_k, M_k}.
\end{align}
In order to deduce weak convergence of $E_{r_k}(\bar{\cR}(\mM_n^t), e_k)$, we need to show that $\sum_{M_k} p_{r_k, M_k} = 1$. We verify this using a proof by contradiction. Suppose that $1 - \sum_{M_k} p_{r_k, M_k} =: \epsilon >0$. Then for any constant $s>0$ 
\begin{align*}
\Pr{\ed(E_{r_k}(\bar{\cR}(\mM_n^t), e_k)) > s} &= 1 - \sum_{M_k, \ed(M_k) \le s} \Pr{U_{r_k}(E_{r_k}(\bar{\cR}(\mM_n^t), e_k)) = M_k} \\ &\to 1 - \sum_{M_k, \ed(M_k) \le s} p_{r_k, M_k} \ge \epsilon.
\end{align*}
It follows that there is a sequence $s_n \to \infty$ with $\Pr{\ed(E_{r_k}(\bar{\cR}(\mM_n^t), e_k)) > s_n} \ge \epsilon/2$ for all $n$. The component containing $v_k$ (inserted at $e_k$) admits $\bar{\mK}^\bullet$ as weak limit, hence the probability for $v_k$ to correspond to $e_k$ converges to a constant $p>0$. It follows that
\begin{align*}
\Pr{\ed(E_{r_k}(\bar{\cK}(\mM_n^t), v_k)) > s_n} \ge (p + o(1))\Pr{\ed(E_{r_k}(\bar{\cR}(\mM_n^t), e_k)) > s_n} = p \epsilon + o(1).
\end{align*}
But this contradicts the distributional convergence of $E_{r_k}(\bar{\cK}(\mM_n^t), v_k)$  ensured by Corollary~\ref{co:doitf}. It follows that
\begin{align}
	\sum_{M_k} p_{r_k, M_k} = 1.
\end{align}
As this holds for all $r_k$, there is a random infinite graph $\hat{\bar{\mR}}^t$ (with a root corner $\hat{e}_{\hat{\bar{\mR}}^t}$ which is the distributional limit of $\bar{\cR}(\mM_n^t)$ rooted according to the stationary distribution. Letting $\hat{\bar{\mK}}^t$ denote the local limit of $\bar{\cK}(\mM_n^t)$ (and $\hat{e}_{\hat{\bar{\mK}}^t}$ its root-corner), it follows from Equation~\eqref{eq:recurddtuned} that
\begin{align}
\label{eq:t333dios}
\Pr{ \cE(M_k, A_k, \hat{\bar{\mK}}^t, \hat{e}_{\hat{\bar{\mK}}}^t)} &=   \sum_{(H_k, B_k) \in \mathfrak{C}(M_k,A_k) } C_{(H_k,B_k)} \Pr{\cE(H_k,B_k, \hat{\bar{\mR}}^t, \hat{e}_{\hat{\bar{\mR}}^t})} \\
&+  C(M_k,A_k) + D_{(M_k, A_k)} \Pr{\cE(M_k,A_k, \hat{\bar{\mR}}^t, \hat{e}_{\hat{\bar{\mR}}^t})}. \nonumber
\end{align}

As stated above, for any $r \ge 1$ the neighbourhoods $U_r(\bar{\cK}(\mM_n^t), v_1)$ and $U_r(\bar{\cK}(\mM_n^t), v_2)$ do not intersect with high probability. Hence jointly and asymptotically the components inserted at $e_1$ and $e_2$ behave like independent copies of $\bar{\mK}$, and the components inserted at the remaining edges of $U(1)$ and $U(2)$ like independent copies of $\bar{\mK}$. Hence, analogously as for Equation~\eqref{eq:recurddtuned}, we obtain
\begin{align}
\label{eq:toobig}
&\Pr{ \cE(M_1, A_1, \bar{\cK}(\mM_n^t), v_1) \text{ and }  \cE(M_2, A_2, \bar{\cK}(\mM_n^t), v_2)   } = o(1) + C(M_1,A_1)C(M_2,A_2) \\
&+  C(M_1,A_1) \left( \quad \quad \quad \sum_{\mathclap{(H_2, B_2) \in \mathfrak{C}(M_2,A_2) }} C_{(H_2,B_2)} \Pr{\cE(H_2,B_2, \bar{\cR}(\mM_n^t), e_2)} + D_{(M_2, A_2)} \Pr{\cE(M_2,A_2, \bar{\cR}(\mM_n^t), e_2)} \right)\nonumber \\
&+  C(M_2,A_2) \left(  \quad \quad \quad \sum_{\mathclap{(H_1, B_1) \in \mathfrak{C}(M_1,A_1) }} C_{(H_1,B_1)} \Pr{\cE(H_1,B_1, \bar{\cR}(\mM_n^t), e_1) + D_{(M_1, A_1)} \Pr{\cE(M_1,A_1, \bar{\cR}(\mM_n^t), e_1)} } \right)\nonumber \\
&+ \sum_{\substack{(H_1, B_1) \in \mathfrak{C}(M_1,A_1)\\(H_2, B_2) \in \mathfrak{C}(M_2,A_2)} } C_{(H_1,B_1)} C_{(H_2,B_2)}  \Pr{\cE(H_1,B_1, \bar{\cR}(\mM_n^t), e_1) \text{ and } \cE(H_2,B_2, \bar{\cR}(\mM_n^t), e_2)} \nonumber \\
&+ D_{(M_1, A_1)} D_{(M_2, A_2)} \Pr{\cE(M_1,A_1, \bar{\cR}(\mM_n^t), e_1) \text{ and } \cE(M_2,A_2, \bar{\cR}(\mM_n^t), e_2)}. \nonumber 
\end{align}
Corollary~\ref{co:doitf} and Proposition~\ref{prop:char} entail that the left-hand side satisfies
\begin{align}
\label{eq:thatonetoo}
\Pr{ \cE(M_1, A_1, \bar{\cK}(\mM_n^t), v_1) \text{ and }  \cE(M_2, A_2, \bar{\cK}(\mM_n^t), v_2)   } \to \Pr{ \cE(M_1, A_1, \hat{\bar{\mK}}^t, \hat{e}_{\hat{\bar{\mK}}^t})}\Pr{ \cE(M_2, A_2, \hat{\bar{\mK}}^t, \hat{e}_{\hat{\bar{\mK}}})}.
\end{align}
Since the marginal probabilities $\Pr{\cE(H_k,B_k, \bar{\cR}(\mM_n^t), e_k)}$ and $\Pr{\cE(M_k,A_k, \bar{\cR}(\mM_n^t), e_k)}$ converge, and since $C_{(M_1, A_1)} C_{(M_2, A_2)}  >0$,  it follows by induction on $\ve(M_1) + \ed(M_1) + \ve(M_2) + \ed(M_2)$ (with the base case being trivial) that there is a constant $p_{M_1, A_1, M_2, A_2} \ge 0$ with 
\[
\lim_{n \to \infty} \Pr{\cE(M_1,A_1, \bar{\cR}(\mM_n^t), e_1) \text{ and } \cE(M_2,A_2, \bar{\cR}(\mM_n^t), e_2)} = p_{M_1, A_1, M_2, A_2}.
\]
It follows by Equations~\eqref{eq:t333dios}, \eqref{eq:thatonetoo} and \eqref{eq:toobig} that
\begin{align*}
 &\sum_{\substack{(H_1, B_1) \in \mathfrak{C}(M_1,A_1)\\(H_2, B_2) \in \mathfrak{C}(M_2,A_2)} } C_{(H_1,B_1)} C_{(H_2,B_2)}  p_{H_1,B_1,H_2,B_2}  
+ D_{(M_1, A_1)} D_{(M_2, A_2)} p_{M_1,A_1,M_2,A_2}  \\
 &= \sum_{\substack{(H_1, B_1) \in \mathfrak{C}(M_1,A_1)\\(H_2, B_2) \in \mathfrak{C}(M_2,A_2)} } C_{(H_1,B_1)} C_{(H_2,B_2)}  \Pr{ \cE(H_1, B_1, \hat{\bar{\mR}}^t, \hat{e}_{\hat{\bar{\mR}}^t})}\Pr{ \cE(H_2, B_2, \hat{\bar{\mR}}^t, \hat{e}_{\hat{\bar{\mR}}^t})} \\
&+ D_{(M_1, A_1)} D_{(M_2, A_2)} \Pr{ \cE(M_1, A_1, \hat{\bar{\mR}}^t, \hat{e}_{\hat{\bar{\mR}}^t})}\Pr{ \cE(M_2, A_2, \hat{\bar{\mR}}^t, \hat{e}_{\hat{\bar{\mR}}^t})}.
\end{align*}
By induction on $\ve(M_1) + \ed(M_1) +  \ve(M_2) + \ed(M_2)$ it follows that
\[
 p_{M_1,A_1,M_2,A_2} = \Pr{ \cE(M_1, A_1, \hat{\bar{\mR}}^t, \hat{e}_{\hat{\bar{\mR}}^t})}\Pr{ \cE(M_2, A_2, \hat{\bar{\mR}}^t, \hat{e}_{\hat{\bar{\mR}}^t})}.
\]
Thus, if $c_n^{(1)}$ and $c_n^{(2)}$ are uniform independent corners of $\bar{\cR}(\mM_n^t)$, then
\begin{align}
\left( \left(\bar{\cR}(\mM_n^t), c_n^{(1)}\right), \left(\bar{\cR}(\mM_n^t), c_n^{(2)}\right)\right) \convd \left(\hat{\bar{\mR}}^{t, (1)}, \hat{\bar{\mR}}^{t, (2)}\right),
\end{align}
with $\hat{\bar{\mR}}^{t, (1)}, \hat{\bar{\mR}}^{t, (2)}$ denoting independent copies of $\hat{\bar{\mR}}^t$. It follows  by Proposition~\ref{prop:char} that
\begin{align}
	\Pr{ (\bar{\cR}(\mM_n^t), c_n) \mid \bar{\cR}(\mM_n^t)} \convp \mfL(\hat{\bar{\mR}}^t).
\end{align}
\end{proof}

\subsection{$\bar{\cO}$-networks}
\label{sec:details}
We define the class of networks $\bar{\cO}$ by
\begin{align}
\bar{\cO} := \bar{\cF}_{0,1}(x, y \Seq(xy)),
\end{align} with $x$ marking vertices (not counting the poles) and $y$ marking regular edges (not counting the ``invisible'' edge between the poles). That is, it is obtained from a $3$-connected map by declaring the oriented root-edge ``invisible'', it's origin the south pole, it's destination the north pole, and substituting all remaining edges by paths of positive length. 

We let $\bar{\cO}^*$ denote the class obtained making a canonical choice of an  edge (with a canonical orientation) in  $\bar{\cO}$, and declaring it invisible. We may think of the ends of this edge as the \emph{second} pair of poles of the network. That is, the species $\bar{\cO}$ and $\bar{\cO}^*$ are related by
\begin{align}
	\bar{\cO} = y\bar{\cO}^*.
\end{align}
Let us recall the decomposition of $\bar{\cR}$:
\begin{align*}
\bar{\cR} &= \bar{\cJ}\Seq(\bar{\cI}^*), \\
\bar{\cI}^* &= y\Seq_{\ge 1}(x y)\Seq(x y) + \bar{\cO}^*(x,y)(1 + y \Seq(x y)), \\
\bar{\cJ} &=1 + y \Seq(x y).
\end{align*}
That is, an $\bar{\cR}$-network consists of a $\bar{\cJ}$-component and a possibly empty ordered sequence of $\bar{\cI}^*$-components. We are going to describe this decomposition in detail. A network from the species $\bar{\cJ}$ may have two different shapes:
\begin{enumerate}
	\item  It may be the trivial network consisting of a south pole, a north pole, and only the ``invisible'' edge between them. This accounts for the summand $1$.
	\item It may consist of two poles joined by a single path of positive length (and, in parallel, the ``invisible'' edge between the poles). This accounts for the summand $y \Seq(x y)$.
\end{enumerate}
Recall that  Equation~\eqref{eq:decompKbar}, that is $\bar{\cK} = y \bar{\cR}(x, \bar{\cK})$, may be interpreted as a recursive description of $\bar{\cK}$-networks. It tells us that a $\bar{\cK}$-network consists of an $\bar{\cR}$-network where we insert an additional edge (corresponding to the factor $y$) between the poles of its $\bar{\cJ}$-component, and substitute all other regular edges (if there are any) by $\bar{\cK}$-networks. 
In particular, when we interpret $\bar{\cR}(\mM_n^t)$ as a planar map, we have to replace the ``invisible'' edge between its poles by a regular edge, and add an additional edge between the poles of its $\bar{\cJ}$-component.

An element from $\bar{\cI}^*$ is a network with a second pair of poles joined by a second ``invisible'' edge. It may have the following shapes:
\begin{enumerate}
	\item It may be an $\bar{\cO}^*$-network. Or it is the parallel composition of an $\bar{\cO}^*$-network with a path of positive length. These cases account for the summand $\bar{\cO}^*(x,y)(1 + y \Seq(x y))$.
	\item It may be constructed as follows: Take the parallel composition of a path of length at least $2$ with a path of positive length. Declare the first edge of the first path as ``invisible'' and its ends as the second pair of poles. This accounts for the summand $y\Seq_{\ge 1}(x y)\Seq(x y)$.
\end{enumerate}
Finally, a network from
\[
\bar{\cR} = \bar{\cJ}\Seq(\bar{\cI}^*) = \bar{\cJ} + \bar{\cJ}\bar{\cI}^* + \bar{\cJ}\left(\bar{\cI}^*\right)^2 + \bar{\cJ}\left(\bar{\cI}^*\right)^3 \ldots
\]
may have the following shapes:
\begin{enumerate}
	\item It may consist of a $\bar{\cJ}$-network. This accounts for the summand $\bar{\cJ}$.
	\item It may be constructed as follows. Take an integer $k \ge 1$. Choose arbitrary  $\bar{\cI}^*$-networks $I_1, \ldots, I_k$ and  a $\bar{\cJ}^*$-network~$J$. We substitute the second pair of poles of $I_1$ by $I_2$, then the second pair of poles of $I_2$ by $I_3$, and so on. Finally we substitute the second pair of poles of $I_k$ by $J$. This accounts for the summand $\bar{\cJ} \left(\bar{\cI}^*\right)^k$.
\end{enumerate}

We let $\bar{\cO}(\mM_n^t)$ denote the largest $\bar{\cO}$-component in the decomposition of $\bar{\cR}(\mM_n^t)$. 
\begin{lemma}
	\label{le:laststep}
	\begin{enumerate}
		\item  The $\bar{\cO}$-component $\bar{\cO}(\mM_n^t)$ admits a distributional limit~$\hat{\bar{\mO}}^t$ in the local topology. Letting  $c_n$ denote a uniformly selected corner of $\bar{\cO}(\mM_n^t)$, it holds that
			\begin{align}
			\label{eq:fi1}
			\Pr{ (\bar{\cO}(\mM_n^t), c_n) \mid \bar{\cO}(\mM_n^t)} \convp \mfL(\hat{\bar{\mO}}^t).
			\end{align}
				\item
		It holds uniformly for all $\ell \in \ndZ$ 
		\begin{align}
		\label{eq:fi3}
		\Pr{\ed(\bar{\cO}(\mM_n^t)) = \ell} = \frac{1}{\tilde{g}_{\bar{\cK}}(t) n^{2/3}}\left(h\left(\frac{ (1-\Ex{\xi^{\bar{\cK}}})(1-\Ex{\xi^{{\cM}}})n - \ell}{\tilde{g}_{\bar{\cK}}(t) n^{2/3}}   \right) + o(1)\right).
		\end{align}
	\end{enumerate}
\end{lemma}
\begin{proof}
	The singular expansion~\eqref{eq:Fsing} entails that
	\begin{align}
		[y^n] \bar{\cO}^*(t,y) \sim c_{\bar{\cO}}(t) \rho_{\cR}(t)^{-n} n^{-5/2}
	\end{align}
	for some constant $c_{\bar{\cO}}(t)>0$. The constant \begin{align}\rho_{\cR}(t) < 1/t\end{align} is given in Equation~\eqref{eq:eqforrhoF}. The summand $y\Seq_{\ge 1}(t y)\Seq(t y)$ has radius of convergence strictly larger than $\rho_{\cR}(t)$.
	Using Equation~\eqref{eq:asymas} it follows that 
	\begin{align}
	\label{eq:goodforsth2}
	 		[y^n] \bar{\cI}^*(t,y) \sim c_{\bar{\cO}}(t)(1 + \rho_{\cR}(t)/(1- t\rho_{\cR}(t)))  \rho_{\cR}(t)^{-n} n^{-5/2}.
	\end{align}
	Hence by Equation~\eqref{eq:gibconv} 
	\begin{align}
	\label{eq:goodforsth}
		[y^n] \Seq(\bar{\cI}^*(t,y)) \sim (1 - \bar{\cI}^*(t,\rho_{\cR}(t)))^{-2} [y^n] \bar{\cI}^*(t,y).
	\end{align}
	The factor $\bar{\cJ}(t,y)$ has radius of convergence strictly larger than $\rho_{\cR}(t)$. By Proposition~\eqref{pro:asymmm} it follows that the $\bar{\cJ}$-component $\bar{\cJ}(\mM_n^t)$ of $\bar{\cR}(\mM_n^t)$ converges to a random $\bar{\cJ}(t, y)$-object following a Boltzmann distribution with parameter $\rho_\cR(t)$. By Lemma~\ref{lem:gibconv} it follows that the $\Seq(\bar{\cI}^*)$-component $\Seq(\bar{\cI}^*)(\mM_n^t)$ of $\bar{\cR}(\mM_n^t)$ has a giant component and the small fragments converge (analogously as in Equation~\eqref{eq:Dtoseq}) to a Boltzmann distributed $\Seq(\bar{\cI}^*(t,y))^2$-object with parameter $\rho_{\cR}(t)$. 
	The generating series $y\Seq_{\ge 1}(t y)\Seq(t y)$ has radius of convergence strictly larger than $\rho_\cR(t)$. Hence the (canonically selected) maximal $\bar{\cI}^*$-component $\bar{\cI}^*(\mM_n^t)$ belongs to $\bar{\cO}^*(t,y)(1 + y \Seq(t y))$ with probability tending exponentially fast to $1$ as $n$ becomes large. It follows from Proposition~\eqref{pro:asymmm} that the $(1 + y \Seq(t y))$-component of $\bar{\cI}^*(\mM_n^t)$ admits a Boltzmann limit distribution with parameter $\rho_\cR(t)$. Summing up, it holds that
	\begin{align}
		\ed(\bar{\cR}(\mM_n^t)) = \ed(\bar{\cO}(\mM_n^t)) + O_p(1).
	\end{align}
	Equation~\eqref{eq:fi1}  now follows from Lemma~\ref{le:convRbar}, by entirely analogous arguments as in the proof of~Corollary~\ref{co:doitf}.
	
	It remains to verify the local limit theorem. To this end, it suffices to verify~\eqref{eq:fi3} for all $\ell \ge 1$. We set $\lambda := (1-\Ex{\xi^{\bar{\cK}}})(1-\Ex{\xi^{{\cM}}})$. Let $X_1$ denote the size of a Boltzmann distributed $\bar{\cJ}(t,y)$-object with parameter $\rho_\cR(t)$. Note that $X_1$ has finite exponential moments. Using Corollary~\ref{co:Rbarsize} and the fact that $h$ is bounded and uniformly continuous, it follows that for each constant $k \ge 0$
	\begin{align*}
		\tilde{g}_{\bar{\cK}}(t) n^{2/3} \Pr{\ed(\Seq(\bar{\cI}^*)(\mM_n^t)) = \ell, \ed(\bar{\cR}(\mM_n^t)) = \ell+k} = h\left(\frac{ \lambda n - \ell}{\tilde{g}_{\bar{\cK}}(t) n^{2/3}}   \right)\Pr{X_1 = k} + o(1).
	\end{align*}
	Here the $o(1)$ term is uniform in $\ell$. Hence for any sequence $(t_n)_n$ of integers that tends sufficiently slowly to infinity, it holds that
	\begin{align}
		\label{eq:maximilian}
		\tilde{g}_{\bar{\cK}}(t) n^{2/3} \sum_{k=0}^{t_n} \Pr{\ed(\Seq(\bar{\cI}^*)(\mM_n^t)) = \ell, \ed(\bar{\cR}(\mM_n^t)) = \ell+k} = h\left(\frac{ \lambda n - \ell}{\tilde{g}_{\bar{\cK}}(t) n^{2/3}}   \right) + o(1).
	\end{align}
	Again, the $o(1)$-term is uniform in $\ell$. Let $Y_1$ denote the size of a Boltzmann distributed $\Seq(\bar{\cI}^*)$-object with parameter $\rho_\cR(t)$. Note that by Equations~\eqref{eq:goodforsth2} and \eqref{eq:goodforsth} it holds that \[
	\Pr{Y_1 = n} \sim c_1 n^{-5/2}
	\] for some constant $c_1>0$.
	 Using Corollary~\ref{co:Rbarsize}, Equation~\eqref{eq:dareal0} and \eqref{eq:dareal1}, and the fact that $h$ is bounded, it follows  that
	\begin{align}
	\label{eq:moritz}
	\tilde{g}_{\bar{\cK}}(t) n^{2/3} &\sum_{t_n \le k \le n} \Pr{\ed(\Seq(\bar{\cI}^*)(\mM_n^t)) = \ell, \ed(\bar{\cR}(\mM_n^t)) = \ell+k}  \\
	&=  \sum_{t_n \le k \le n} \left(h\left(\frac{ \lambda n - \ell}{\tilde{g}_{\bar{\cK}}(t) n^{2/3}}   \right) + o(1)\right)
	 \frac{\Pr{X_1=k}\Pr{Y_1=\ell}}{\Pr{X_1+Y_1=\ell+k}} \nonumber \\
	 &\le  O(1) \sum_{t_n \le k \le n}
	 \Pr{X_1=k}\left(1 + \frac{k}{\ell}\right)^{5/2}. \nonumber
	\end{align}
	This bound tends to zero uniformly for all $\ell \ge 1$, since $X_1$ has finite exponential moments. Combining Equations~\eqref{eq:maximilian} and \eqref{eq:moritz} yields
	\begin{align}
	\label{eq:wilhelmbusch}
	 \Pr{\ed(\Seq(\bar{\cI}^*)(\mM_n^t)) = \ell} = \frac{1}{\tilde{g}_{\bar{\cK}}(t) n^{2/3}} \left(h\left(\frac{ \lambda n - \ell}{\tilde{g}_{\bar{\cK}}(t) n^{2/3}}   \right) + o(1) \right).
	\end{align}
	By identical arguments as in the proof of Corollary~\ref{co:tillicollapse}, it follows that
	\begin{align}
		\Pr{\ed(\bar{\cI}^*(\mM_n^t)) = \ell} = \frac{1}{\tilde{g}_{\bar{\cK}}(t) n^{2/3}} \left(h\left(\frac{ \lambda n - \ell}{\tilde{g}_{\bar{\cK}}(t) n^{2/3}}   \right) + o(1) \right).
	\end{align}
	Using identical arguments as for Equations~\eqref{eq:maximilian}, \eqref{eq:moritz}, and~\eqref{eq:wilhelmbusch}, it follows that 
	\begin{align}
		\Pr{\ed(\bar{\cO}(\mM_n^t)) = \ell} = \frac{1}{\tilde{g}_{\bar{\cK}}(t) n^{2/3}} \left(h\left(\frac{ \lambda n - \ell}{\tilde{g}_{\bar{\cK}}(t) n^{2/3}}   \right) + o(1) \right).
	\end{align}
\end{proof}

\subsection{Transfer between different mixtures}

In the preceding arguments, we transferred properties  of $\mM_n^t$ to different cores: $\cV(\mM_n^t)$, $\bar{\cK}(\mM_n^t)$, $\bar{\cR}(\mM_n^t)$, and $\bar{\cO}(\mM_n^t)$. It is an important subtlety of these arguments that each of these cores  has a \emph{random} number of edges, for which we deduced a local limit theorem with a $3/2$-stable limit law. Conditioned on having a fixed number $k$ of edges, each core gets drawn with probability proportional to its weight (defined by putting weight $t$ at vertices) among all $k$-edge elements of its corresponding class. That is, each core is a mixture of random weighted objects.


\subsubsection{An absolute continuity relation}
Let us observe that we have a certain degree of freedom in changing these mixtures. To this end, suppose that $S$ is a  space and let $\mfB(S)$ denote its Borel $\sigma$-algebra. Let $(\mS_n)_n$ denote a sequence of $S$-valued random variables. Let $X_n$ and $Y_n$ denote random integers, each being independent from  $(\mS_n)_{n}$. Suppose that there are constants $\mu_X, \mu_Y, g_X, g_Y>0$ such that 
\begin{align}
\label{eq:llt11}
\Pr{ X_n =  \ell} = \frac{1}{g_X n^{2/3}}\left(h\left(\frac{ \mu_X n - \ell}{g_X n^{2/3}}   \right) + o(1)\right)
\end{align}
and
\begin{align}
\label{eq:llt22}
\Pr{ Y_n =  \ell} = \frac{1}{g_Y n^{2/3}}\left(h\left(\frac{ \mu_Y n - \ell}{g_Y n^{2/3}}   \right) + o(1)\right)
\end{align}
uniformly for all $\ell \in \ndZ$.

\begin{lemma}
	\label{le:megatransfer}
	Let $s_n = {\left\lfloor n \frac{\mu_Y}{\mu_X} \right\rfloor}$.
	\begin{enumerate}
		\item For each $\epsilon>0$ there are constants $0<c<C$ and $N_0>0$ such that for all $n \ge n_0$ and all events $\cE \in \mfB(S)$
		\begin{align}
			\label{eq:toshowfff}
			 c\Pr{\mS_{X_{s_n}} \in \cE} - \epsilon \le \Pr{\mS_{Y_n} \in \cE} \le  C\Pr{\mS_{X_{s_n}} \in \cE} + \epsilon.
		\end{align}
			\item If $(\mS_{X_n})_{n \ge 1}$ is uniformly tight, then so is $(\mS_{Y_n})_{n \ge 1}$.
		\item  If $\mS_{X_n} \in \cE$ holds with high probability, then so does $\mS_{Y_n} \in \cE$.
		\item In the case where $S=\mathfrak{G}$ or $S=\mathfrak{M}$ (defined in Section~\ref{sec:localconv}), quenched local convergence of $\mS_{X_n}$ to a deterministic law $\mu$ implies quenched local convergence of $\mS_{Y_n}$ to $\mu$.
	\end{enumerate}
\end{lemma}
\begin{proof}
	Using Equations~\eqref{eq:llt11} and \eqref{eq:llt22} (and the fact that the density function $h$ is bounded, uniformly continuous and positive) it follows that for each constant $M>0$ there are constants $0 < c_M < C_M$ such that  uniformly for all integers $k = \mu_Y n + x n^{2/3}$ with $|x| \le M$ 
	\begin{align}
		\frac{\Pr{X_{s_n} = k}}{\Pr{Y_n = k}} =  \frac{g_Yh\left(\frac{x}{g_X \left( \mu_Y/ \mu_X\right)^{2/3}} \right) + o(1)}{g_X \left( \frac{\mu_Y}{ \mu_X}\right)^{2/3}h\left( \frac{x}{g_Y}\right) +o(1)} \in [c_M + o(1), C_M + o(1)].
	\end{align}
	This yields
	\begin{align*}
		\Pr{\mS_{Y_n} \in \cE} &\le \Pr{|Y_n - \mu_Yn| \ge M n^{2/3}} + \sum_{k \in n \mu_Y \pm M n^{2/3}} \Pr{Y_n = k} \Pr{ \mS_k \in \cE  } \\
		&\le \Pr{|Y_n - \mu_Yn| \ge M n^{2/3}}  + (C_M + o(1)) \sum_{k \in n \mu_Y \pm M n^{2/3}} \Prb{ X_{s_n} = k  } \Pr{ \mS_k \in \cE  } \\
		&\le o(1) + \Pr{|Y_n - \mu_Yn| \ge M n^{2/3}}  + C_M\Pr{\mS_{X_{s_n}} \in \cE},
	\end{align*}
	with an $o(1)$ term that only depends on $M$ and $n$. Likewise
	\begin{align*}
		\Pr{\mS_{Y_n} \in \cE} & \ge o(1) - \Pr{|X_{s_n} - \mu_Yn| \ge M n^{2/3}} + c_M\Pr{\mS_{X_{s_n}} \in \cE}.
	\end{align*}
	Given $\epsilon>0$, it follows from Equations~\eqref{eq:llt11} and \eqref{eq:llt22} that we may select $M$ sufficiently large such that
	\[
		\Pr{|Y_n - \mu_Y| \ge M n^{2/3}} \le \epsilon/2 \qquad \text{and} \qquad \Pr{|X_{s_n} - \mu_X| \ge M n^{2/3}} \le \epsilon/2.
	\]
	This verifies~Inequality~\eqref{eq:toshowfff}.
	
	As for the second claim, suppose that $(\mS_{X_n})_{n \ge 1}$ is uniformly tight. Let $\epsilon>0$ be given. Then there is a compact subset $K_0 \subset X$ with $\Pr{\mS_{X_n} \notin K_0} \le \epsilon/2$ for all $n$. By Inequality~\eqref{eq:toshowfff} it follows that there is a constant $N_0$ with $\Pr{\mS_{Y_n} \notin K_0} \le \epsilon$ for all $n \ge N_0$. For each $1 \le i \le N_0$ there is a compact subset $K_i \subset S$ with $\Pr{\mS_{Y_i} \notin K_i} \le \epsilon$, hence $K= \bigcup_{i=0}^{N_0} K_i$ is compact and satisfies $\Pr{\mS_{Y_n} \notin K} \le \epsilon$ for all $n\ge 1$.
	
	The third claim follows directly from Inequality~\eqref{eq:toshowfff}. 
	
	As for the fourth, Proposition~\ref{prop:conv} implies that quenched local convergence corresponds to convergence in probability of the percentage of vertices / corners with an (arbitrary but fixed) radius $r \ge 1$ neighbourhood having an (arbitrary but  fixed) shape $M$. Hence if $\mS_{X_n}$ converges in the quenched sense to a deterministic limit law $\mu$, then this percentage of specified points in $\mS_{X_n}$ converges in probability to a constant $p_{r, M}$ given by the corresponding $\mu$-probability. The third claim now implies that the same holds for $\mS_{Y_n}$, yielding quenched convergence of $\mS_{Y_n}$.
\end{proof}

\subsection{An application to random $2$-connected planar maps}

Let $\mV_n^t$ denote the random  non-separable planar map with $n$ edges drawn with probability proportional to the weight $t^{\ve(\cdot)}$. Recall that in Lemma~\ref{le:map1to2} we established a quenched local limit $\hat{\mV}^t$ of the core $\cV(\mM_n^t)$.

\begin{theorem}
	\label{te:map2}
	Letting  $c_n$ denote a uniformly selected corner of $\mV_n^t$, it holds that
	\begin{align}
	\Pr{ (\mV_n^t, c_n) \mid \mV_n^t} \convp \mfL(\hat{\mV}^t).
	\end{align}
	We call $\hat{\mV}^t$ the uniform infinite non-separable planar map with weight $t$ at vertices.
\end{theorem}

The limit $\hat{\mM}^t$ may be constructed from the uniform infinite non-separable map $\hat{\mV}^t$ by inserting independent random planar maps (with explicit distributions) at each corner of the uniform infinite non-separable planar map, see   \cite[Thm. 6.59]{2016arXiv161202580S}. The asymptotic degree distribution of $\mV_n^1$ was established in prior works by~\cite{MR3071845}.

\begin{proof}[Proof of Theorem~\ref{te:map2}]
By Lemma~\ref{le:laststep}, the $\bar{\cO}$-core $\bar{\cO}(\mM_n^t)$ admits a quenched  limit~$\hat{\bar{\mO}}^t$ in the local topology, and
\[
\Pr{\ed(\bar{\cO}(\mM_n^t)) = \ell} = \frac{1}{\tilde{g}_{\bar{\cK}}(t) n^{2/3}}\left(h\left(\frac{ (1-\Ex{\xi^{\bar{\cK}}})(1-\Ex{\xi^{{\cM}}})n - \ell}{\tilde{g}_{\bar{\cK}}(t) n^{2/3}}   \right) + o(1)\right).
\]
uniformly for $\ell \in \ndZ$. We may define the cores 
$\bar{\cK}(\mV_n^t)$, $\bar{\cR}(\mV_n^t)$ and $\bar{\cO}(\mV_n^t)$ analogously as for $\mM_n$, and by analogous arguments it follows that
\begin{align}
\label{eq:vcorellt}
\Pr{\ed(\bar{\cO}(\mV_n^t)) = \ell} = \frac{1}{g(t) n^{2/3}}\left(h\left(\frac{ (1-\Ex{\xi^{\bar{\cK}}})n - \ell}{g(t) n^{2/3}}   \right) + o(1)\right)
\end{align}
for some constant $g(t)>0$. By Lemma~\ref{le:megatransfer} it follows that $\hat{\bar{\mO}}^t$ is also the quenched local limit of $\bar{\cO}(\mV_n^t)$. The arguments in the proof of Lemma~\ref{le:laststep}, that pass quenched local convergence of a large random $\bar{\cR}$-structure down to its giant $\bar{\cO}$-core, also entail, conversely, that convergence of such a $\bar{\cO}$-core entails convergence of the $\bar{\cR}$ structure. Hence $\hat{\bar{\mR}}^t$ is also the quenched local limit of~$\bar{\cR}(\mV_n^t)$. Likewise, the arguments in the proof of Lemma~\ref{le:convRbar}, that pass convergence from a large random $\bar{\cK}$-structure down to its $\bar{\cR}$-core, easily imply that convergence of the $\bar{\cR}$-core implies convergence of the $\bar{\cK}$-structure. Hence $\bar{\cK}(\mV_n^t)$ admits $\hat{\bar{\mK}}^t$ as quenched local limit. The same goes for Corollary~\ref{co:doitf}, yielding that $\hat{\mV}^t$ is the quenched local limit of $\mV_n^t$.
\end{proof}

\subsection{$\cK$-networks}

We let $\mK_n^t$ denote a random $\cK$-network, drawn with probability proportional to its weight given by $t^{\ve(\cdot)}$. Equation~\eqref{eqK:decomp}, that is 	$\cK \equiv y \cR(x, \cK)$, and the discussion in Section~\ref{sec:enriched} imply that  $\cR$-enriched plane trees may be transformed into $\cK$-networks. The network corresponding to such an enriched tree $(T, \beta)$ with $n$ vertices has $n$  edges and gets constructed as follows. The $\cR$-structure $\beta(o)$ corresponding to the root-vertex $o$ of $T$ is a network with $d^+_T(o)$ regular edges and an additional ``terminal'' edge. The regular edges correspond bijectively to the fringe subtrees dangling from $o$. The total network gets constructed recursively by replacing each regular edge by the network corresponding to its fringe subtree. The terminal edge corresponds to the factor $y$ in $\cK \equiv y \cR(x, \cK)$.

We let $\mT_n^\cK$ denote the simply generated tree with weight-sequence $(\omega_k^{\cK})_{k \ge 0}$ given by
\begin{align}
\omega_k^{\cK} = [y^k]\cR(t,y).
\end{align}
For each vertex $v$ of $\mT_n^\cK$ we draw a $d^+_{\mT_n^\cK}$-sized $\cR$-structure $\beta_n^\cK(v)$ with probability proportional to its weight. Lemma~\ref{le:sampling} implies that the random $\cK$-structure corresponding to the random enriched plane tree $(\mT_n^\cK, \beta_n^\cK)$ is distributed like $\mK_n^t$.

Inequality~\eqref{eq:ksubcrit}, Equation~\eqref{eq:Rasymp}, and Lemma~\ref{le:simplygen} imply that the simply generated tree $\mT_n^{{\cK}}$ is distributed like a Galton--Watson tree $\mT^{{\cK}}$ conditioned on having  $n$ vertices, with offspring distribution $\xi^{{\cK}}$ satisfying
\begin{align}
\Ex{\xi^{{\cK}}} < 1 \qquad \text{and} \qquad \Pr{\xi^{{\cK}} = n} \sim \frac{c_{{\cR}}(t)}{{\cR}(t, \rho_\cR(t))} n^{-5/2}.
\end{align}

Lemma~\ref{le:maxllt} entails that there is a constant $g_{\cK}(t) >0$ such that the largest ${\cR}$-component ${\cR}({\mK}_n^t)$ satisfies
\begin{align}
\label{eq:lltRtnb}
\Pr{\ed({\cR}({\mK}_n^t)) = \ell} = \frac{1}{g_{{\cK}}(t) n^{2/3}}\left(h\left(\frac{ (1-\Ex{\xi^{{\cK}}})n - \ell}{g_{{\cK}}(t) n^{2/3}}   \right) + o(1)\right)
\end{align}
uniformly for all $\ell \in \ndZ$.

Similarly as we defined the class of networks $\bar{\cO}$, we let
\begin{align}
\cO := {\cF}_{0,1}(x, y \Seq(xy)) 
\end{align}
denote the pendant of networks obtained  by blowing up regular edges of $\cF_{0,1}$-networks into paths. Whitney's theorem, see \cite{MR1506961}, yields a $1:2$ correspondence between $\mathcal{O}$-networks and $\bar{\mathcal{O}}$-networks, as up to reflection any $3$-connected graph has a unique embedding into the $2$-sphere (and any $3$-connected map has at least $4$ vertices and differs from its mirror-image).  Thus
\begin{align}
\label{eq:halfit}
\cO = \frac{1}{2}\bar{\cO}.
\end{align}

Recall that $\cR$ admits the decomposition
\begin{align*}
\cR &= \cJ \Seq(\cI^*),\\
y \cI^* &=  \cO + \Set_{\ge 2}\left(\cO + \cL\right), \\
\cJ &= \Set(\cO + \cL).
\end{align*}
with
\begin{align}
	\cL:= y \Seq_{\ge 1}(xy).
\end{align}

This means an $\cR$-network consists of a $\cJ$-component and a possibly empty sequence of $\cI^*$-components. We explain this in detail:

The $\cJ$-component is a network consisting of the parallel composition of a (possibly empty) unordered collection of networks that are either $\cO$-networks or paths of length at least $2$ (corresponding to $\cL$). If the collection is empty, we interpret this as the network consisting of two poles and no regular edges.

A $y\cI^*$-network is either an $\cO$-network, or the parallel composition of an unordered collection of at least two networks, each being either an $\cO$-network or a path of length at least two.  An $\cI^*$-network is a weighted network (weighted by both the fact that we have weight $t$ at vertices, and that we divided by $y$) that, in addition to the ``invisible'' edge between the poles, has a second distinguished ``invisible'' edge. 

A network from
\[
	\cR = \cJ \Seq(\cI^*) = \cJ + \cJ \cI^* + \cJ (\cI^*)^2 + \cJ(\cI^*)^3 \ldots
\]
is either a $\cJ$-network, or it belongs to  $\cJ (\cI^*)^k$ for some $k \ge 1$. That is, it gets constructed as follows. Take $k$ $\cI^*$-networks $I_1, \ldots, I_k$ and a $\cJ$-network $J$. Substitute the second ``invisible'' edge of $I_1$ by $I_2$, then the second ``invisible'' edge of $I_2$ by $I_3$, and so on. Finally, substitute the second ``invisible'' edge of $I_k$ by $J$.

	We let $\cO(\mK_n^t)$ denote the largest $\cO$-component in the decomposition of $\cR(\mK_n^t)$. 
\begin{lemma}
	\label{le:laststep2}
	\begin{enumerate}
		\item It holds uniformly for all $\ell \in \ndZ$ 
		\begin{align}
		\label{eq:lltOtnb}
		\Pr{\ed({\cO}({\mK}_n^t)) = \ell} = \frac{1}{g_{{\cK}}(t) n^{2/3}}\left(h\left(\frac{ (1-\Ex{\xi^{{\cK}}})n - \ell}{g_{{\cK}}(t) n^{2/3}}   \right) + o(1)\right).
		\end{align}
		\item   If  $c_n^\cR$ denotes a uniformly selected corner of ${\cR}(\mK_n^t)$, then
		\begin{align}
		\label{eq:fi122}
		\Pr{ ({\cR}(\mK_n^t), c_n^\cR) \mid {\cR}(\mK_n^t)} \convp \mfL(\hat{\bar{\mO}}^t).
		\end{align}
		\item   There is a random infinite planar map $\hat{\mK}^t$ such that
		\begin{align}
		\label{eq:locK}
		\Pr{ (\mK_n^t, c_n^\cK) \mid \mK_n^t} \convp \mfL(\hat{\mK}^t)
		\end{align}
		with $c_n^\cK$ denoting a uniformly selected corner of $\mK_n^t$. There is also a random infinite planar graph $\hat{\mK}^{\mathrm{u},t}$ such that
		\begin{align}
		\label{eq:locKuniv}
		\Pr{ (\mK_n^t, v_n^\cK) \mid \mK_n^t} \convp \mfL(\hat{\mK}^{\mathrm{u},t})
		\end{align}
		with $v_n^\cK$ denoting a uniformly selected vertex of $\mK_n^t$.
	\end{enumerate}
\end{lemma}
\begin{proof}
	We start with the first claim. Equation~\eqref{eq:halfit} and the singular expansion~\eqref{eq:Fsing} entail 
		\begin{align}
		[y^n] \cO(t,y) \sim c_{{\cO}}(t) \rho_{\cR}(t)^{-n} n^{-5/2}
		\end{align}
	for some constant $c_{\cO}(t)>0$. Recall that  $\rho_{\cR}(t) < 1/t$ by Equation~\eqref{eq:eqforrhoF}, so
	\[
		[y^n](\cO(t,y) + \cL(t,y)) = [y^n]\cO(t,y)(1 + o(1))
	\]
	with the $o(1)$ term tending exponentially fast to zero. By Proposition~\ref{pro:asymmm} it follows that
	\begin{align}
		\label{eq:Jasymp}
			[y^n]\cJ(t,y) \sim c_\cJ(t) [y^n]\cO(t,y),
	\end{align}
	for $c_\cJ(t) := \exp\left(\cO(t,\rho_\cR(t)) + t\rho_\cR(t)^2/(1- \rho_\cR(t))\right)$. It also follows that large $\cJ$-objects have a giant $\cO$-component with a stochastically bounded remainder that admits a limit distribution. Likewise, Propositions~\ref{pro:asymmm} and~\ref{pro:bothasymmm} entail that
	\begin{align}
		\label{eq:Iasymp}
		[y^n]\cI^*(t,y) \sim c_{\cI^*}(t) [y^n]\cO(t,y)
	\end{align}
	for some constant $c_{\cI^*}(t)>0$. This yields
	\begin{align}
		\label{eq:SIasymp}
		[y^n]\Seq(\cI^*(t,y)) \sim (1- \cI^*(t, \rho_\cR(t)))^{-2} c_{\cI^*}(t) [y^n]\cO(t,y).
	\end{align}
	
	It also follows that large $\cI^*$-objects have a giant $\frac{1}{y}\cO(t,y)$-component (corresponding canonically to an $\cO$-structure), with a stochastically bounded remainder that admits a limit distribution. By~Proposition~\ref{pro:bothasymmm} and Equation~\eqref{eq:lltRtnb} it  follows that $\cR(\mK_n^t)$ has a giant $\cO$-component and  a stochastically bounded remainder admitting a limit distribution. In particular,
	\begin{align}
		\label{eq:smoland}
		\ed(\cR(\mK_n^t)) = \ed(\cO(\mK_n^t)) + O_p(1).
	\end{align}

	We let $\cJ(\mK_n^t)$ denote the $\cJ$-component of $\cR(\mK_n^t)$, and set $X_n = \ed(\cJ(\mK_n^t))$. Likewise, we let $\Seq(\cI^*)(\mK_n^t)$ be the $\Seq(\cI^*)$-component and let $Y_n$ denote its size. For any integer $m \ge 0$ it holds that
	\begin{align}
	\label{eq:bbbb}
		((X_n, Y_n) \mid \ed(\cR(\mK_n^t)) = m) \eqdist ((X,Y) \mid X+Y = m)
	\end{align}
	with $X$ and $Y$ denoting the sizes of Boltzmann distributed $\cJ(t, y)$ and $\Seq(\cI^*(t, y))$ objects with parameter $\rho_\cR(t)$. It follows from~\cite[Eq. (2.10)]{2019arXiv190104603S} that there is a constant $\epsilon>0$ such that 
	\begin{align}
		\label{eq:cccc}
		\Pr{ \ed(\cR(\mK_n^t)) \le \epsilon n / \log n}  = \Pr{\Delta(\mT_n^\cK) \le \epsilon n / \log n} = o(n^{-2/3}).
	\end{align}
	With foresight we set $t_n := \lfloor n^{\delta} \rfloor $ for some constant $\delta$ satisfying $4/9 < \delta <1$. Using Equations \eqref{eq:bbbb} and \eqref{eq:cccc} as well as the asymptotics~\eqref{eq:Jasymp} and \eqref{eq:SIasymp} it follows that
	\begin{align*}
		n^{2/3} \Pr{ t_n \le X_n \le \ed(\cR(\mK_n^t)) - t_n} &= o(1) + n^{2/3} \,\, \sum_{\mathclap{r = \epsilon n / \log n}}^{n} \,\, \Pr{\ed(\cR(\mK_n^t)) = r} \sum_{s=t_n}^{r-t_n} \frac{\Pr{X=s} \Pr{Y=r-s}}{\Pr{X+Y=r} } \\
		&= o(1) + O(n^{2/3}) \,\, \sum_{\mathclap{r = \epsilon n / \log n}}^{n} \,\, \Pr{\ed(\cR(\mK_n^t)) = r} \sum_{s=t_n}^{r-t_n} \left( s \frac{(r-s)}{r} \right)^{-5/2} \\
		&= o(1) + O(n^{2/3})   t_n^{-3/2} \\
		&= o(1).
	\end{align*}
	This entails that $\Pr{\ed({\cO}({\mK}_n^t)) = \ell}$ may be written as
	\begin{multline*}
		 o(n^{-2/3}) + \sum_{r=\epsilon n  /\log n}^n \frac{\Pr{\ed(\cR(\mK_n^t)) = r}}{\Pr{X+Y=r}} \sum_{s=0}^{t_n} \left(\Pr{X=r-s} \Pr{Y=s} A + \Pr{X=s} \Pr{Y=r-s}B\right).
	\end{multline*}
	Here $A$ denotes the probability that the size of the largest $\cO$-component found in the decomposition of a pair of a random $s$-sized $\Seq(\cI^*(t,y))$-structure and a random $(r-s)$-sized $\cJ(t,y)$-structure equals precisely $\ell$. The probability $B$ is defined analogously for a random $(r-s)$-sized $\Seq(\cI^*(t,y))$-structure and a random $s$-sized $\cJ(t,y)$-structure.
	
	We would like to replace $A$ and $B$ by the corresponding probabilities that involve only the $r-s$ sized components. To this end, let $\delta'$ be a fixed constant satisfying $\delta < \delta' < 1$.
	It follows from Proposition~\ref{prop:gibbsbound} that the probability for the size of the largest $\cO$-component in an $r - s \ge \Theta(n/ \log n)$ sized random $\cJ(t,y)$-structure to be smaller than $n^{\delta'}$ decays faster than any power of $1/n$. The same goes for the largest $\cI^*(t,y)$-structure in a random $r-s$ sized $\Seq(\cI^*(t,y))$-structure. As $\cL(t,y)$ has radius of convergence strictly larger than $\rho_\cR(t)$, it follows that the total variational distance between a random $m$-sized $y \cI^*$-structure and a random $m$-sized 
	\begin{align}
	\label{eq:secondgibbs}
	\cL + y \cI^* = \Set_{\ge 1}(\cO + \cL)
	\end{align}
	is exponentially small in $m$ as $m \to \infty$. This allows us to apply  Proposition~\ref{prop:gibbsbound} again, yielding that the probability for the largest $\cO$-component within the largest $\cI^*(t,y)$ in a random $(r-s)$-sized $\Seq(\cI^*(t,y))$-structure to be less than $n^{\delta'}$ tends to zero faster than any power of~$1/n$.

	Summing up, we may assume that
	\begin{align}
	\label{eq:asonell}
		\ell \ge n^{\delta'}
	\end{align} and replace the constants $A$ and $B$ in the previous expression for $\Pr{\ed({\cO}({\mK}_n^t)) = \ell}$ by constants $\tilde{A}$ and $\tilde{B}$, with $\tilde{A}$ the probability for the size of the largest $\cO$-component in a random $(r-s)$-sized $\cJ(t,y)$-structure to equal $\ell$, and $\tilde{B}$ analogously the probability for the size of the largest $\cO$-component in a random $(r-s)$-sized $\Seq(\cI^*(t,y))$-structure to equal $\ell$.

	There is a constant $0<p<1$ such that uniformly for all $r \ge \epsilon n / \log n$ and all $0 \le s \le t_n$
	\[
		\frac{\Pr{X=r-s}}{\Pr{X+Y=r}} \sim p  \qquad \text{and} \qquad \frac{\Pr{Y=r-s}}{\Pr{X+Y=r}} \sim 1-p.
	\]
	Hence
	\begin{align}
	\label{eq:foobar}
	n^{2/3} \Pr{\ed({\cO}({\mK}_n^t)) = \ell} &= o(1) + (p + o(1)) n^{2/3} \sum_{r=\epsilon n  /\log n}^n \Pr{\ed(\cR(\mK_n^t)) = r} \sum_{s=0}^{t_n} \Pr{Y=s} \tilde{A}  \\
	&+ (1-p + o(1)) n^{2/3} \sum_{r=\epsilon n  /\log n}^n \Pr{\ed(\cR(\mK_n^t)) = r} \sum_{s=0}^{t_n} \Pr{X=s} \tilde{B}. \nonumber
	\end{align}
	
	Let $(s_n)_n$ denote an arbitrary sequence of positive integers satisfying $s_n \le t_n$ and $s_n \to \infty$. We are going to argue that in Equation~\eqref{eq:foobar} only summands with $0 \le s \le s_n$ contribute, regardless how slowly $s_n$ tends to infinity. We start with the sum involving $\tilde{A}$. Let $Z$ denote the size of a random Boltzmann distributed $\cO(t,y)$-object. By Proposition~\ref{prop:gibbsbound} and the fact that $\cL(t,y)$ has radius of convergence strictly larger than $\rho_\cR(t)$ it follows that uniformly for $\epsilon n /\log n \le r \le n$ and $s_n \le s \le t_n$
	\begin{align}
		\tilde{A} \le  E_n +  C \frac{\Pr{Z=\ell} \Pr{Z=r-s-\ell}}{\Pr{Z=r-s}} \exp\left(- \frac{r-s-\ell}{\ell}\right)\one_{\ell \le r-s}
	\end{align}
	for some constant $C>0$ and an error term $E_n$ that depends only on $n$ and tends to zero faster than any power of $1/n$. (Here's a detailed justification: $E_n$ is bounded by the probability that the largest $\cO + \cL$-component in the random $(r-s)$-sized $\cJ(t,y)=\Set(\cO + \cL)$ structure is an $\cL$-structure.   Proposition~\ref{prop:gibbsbound} ensures that the size of this structure is at least $n^{\delta'}$ with a probability that tends to zero faster than any power of $1/n$. As $\cL(t,y)$ has radius of convergence strictly larger than $\rho_\cR(t)$, it follows that the same holds for the decay of $E_n$.) Continuing the argument, we may consider the two cases $\ell  \ge (r-s)/2$ and $\ell < (r-s)/2$ separately to obtain
	\begin{align}
		\frac{\Pr{Z=\ell} \Pr{Z=r-s-\ell}}{\Pr{Z=r-s}}  = O(1)(\Pr{Z= \ell}\one_{\ell < (r-s)/2} + \Pr{Z = r - s - \ell} \one_{\ell  \ge (r-s)/2}).
	\end{align}
	Using Equation~\eqref{eq:lltRtnb} and Inequality~\eqref{eq:asonell}, it follows that 
	\begin{align*}
		&n^{2/3} \sum_{r=\epsilon n  /\log n}^n \Pr{\ed(\cR(\mK_n^t)) = r} \sum_{s=s_n}^{t_n} \Pr{Y=s} \tilde{A} \\
		&= o(1) + O(1)\sum_{s=s_n}^{t_n} \Pr{Y=s}  \sum_{r = \max(\epsilon n /\log n, \ell+s)}^n \left((r -s - \ell)^{-5/2} + \Pr{Z=\ell} \exp\left(- \frac{r-s-\ell}{\ell}\right)  \right) \\
		&= o(1) + O(1)\Pr{s_n \le Y \le t_n}.
	\end{align*}
	This bound clearly tends to zero. Hence we have tight control over the size of the components, the next step is to control the deviation of $\ell$ from $r-s$. For any sequence of positive integers $(u_n)_n$ that tends to infinity it follows by the same exact bounds that 
	\begin{align*}
		&n^{2/3} \sum_{r=\epsilon n  /\log n}^n \Pr{\ed(\cR(\mK_n^t)) = r} \sum_{s=0}^{s_n} \Pr{Y=s} \tilde{A} \one_{r-s-\ell \ge u_n} = o(1).
	\end{align*}
	The Gibbs partition $\cJ = \Set(\cO + \cL)$ is convergent. Hence for any constant integer $u \ge 0$ it holds that $\tilde{A}$ (a quantity that depends on $r-s$ and $\ell$) converges to a limiting probability $a_u$ (with $\sum_{u \ge 0 } a_u = 1$) uniformly for all $r,s,\ell$ with $r-s = \ell + u$. Moreover, the local limit theorem in Equation~\eqref{eq:lltRtnb} entails that 
	\[
		\Pr{\ed({\cR}({\mK}_n^t)) = \ell + x } \sim \Pr{\ed({\cR}({\mK}_n^t)) = \ell + x }
	\]
	uniformly for all integers $x$ with $|x| = o(n^{2/3})$ and all integers $\ell$ satisfying Inequality~\eqref{eq:asonell}.  Hence we may choose the sequences $(t_n)_n$ and $(u_n)_n$ to tend to infinity sufficiently slowly so that
	\begin{align*}
	n^{2/3} \sum_{r=\epsilon n  /\log n}^n & \sum_{s=0}^{s_n} \Pr{Y=s} \tilde{A} \one_{r-s-\ell \le u_n}  \\
	&= 	n^{2/3} \sum_{s=0}^{s_n} \sum_{u=0}^{u_n} \Pr{\ed(\cR(\mK_n^t)) = \ell + u + s} \Pr{Y=s} ( a_u + o(1)) \\
	&= n^{2/3} \Pr{\ed(\cR(\mK_n^t)) = \ell}.
	\end{align*}
	Hence Equation~\eqref{eq:foobar} simplifies to
	\begin{align}
		\label{eq:foobar1}
		n^{2/3} \Pr{\ed({\cO}({\mK}_n^t)) = \ell} &= o(1) + p n^{2/3}\Pr{\ed(\cR(\mK_n^t)) = \ell}  \\
		&+ (1-p + o(1)) n^{2/3} \sum_{r=\epsilon n  /\log n}^n \Pr{\ed(\cR(\mK_n^t)) = r} \sum_{s=0}^{t_n} \Pr{X=s} \tilde{B}, \nonumber
	\end{align}
	with the $o(1)$ terms being uniform in $n$ and all $\ell$ satisfying Inequality~\eqref{eq:asonell}. We have also established above that $n^{2/3} \Pr{\ed({\cO}({\mK}_n^t)) = \ell}$ tends to zero uniformly for all $\ell$ that do not satisfy Inequality~\eqref{eq:asonell}, as does the right hand side of Equation~\eqref{eq:lltOtnb}. Hence the restriction on $\ell$ is not a real restriction at all.
	
	The double sum involving $\tilde{B}$ may be treated using analogous arguments: first argue as before that we may discard all summands for which $s_n \le s \le t_n$. Then expand $\tilde{B}$ and discard the summands for which the largest $\cI^*$-component in the $(r-s)$-sized component is not close to $r-s$. Then use Equation~\eqref{eq:secondgibbs} and again the same arguments to discard all summands for which the largest $\cO$-component within that largest $\cI^*$-component is not close to the size of the $\cI^*$-component. It is clear how to carry out each of these tedious but not difficult steps, hence we leave the details to the inclined reader. 
	
	Having taken care of the double sum involving $\tilde{B}$, Equation~\eqref{eq:foobar1} reduces to 
	\begin{align}
	 \Pr{\ed({\cO}({\mK}_n^t)) = \ell} &= o(n^{-2/3}) + \Pr{\ed(\cR(\mK_n^t)) = \ell}.
	\end{align}
	By Equation~\eqref{eq:lltRtnb} this readily implies Equation~\eqref{eq:lltOtnb}. 
	
	Equation~\eqref{eq:lltOtnb} and Lemma~\ref{le:laststep} state local limit theorems for the sizes of the cores $\bar{\cO}(\mK_n^t)$ and $\cO(\mK_n^t)$.  Equation~\eqref{eq:halfit} ensures that conditioned on having a common fixed size, the cores follow the same conditional distribution. This allows us to apply   Lemma~\ref{le:megatransfer} to transfer the quenched local limit of $\bar{\cO}(\mK_n^t)$ (stated in Lemma~\ref{le:laststep}) to a quenched local limit theorem for $\cO(\mK_n^t)$. That is, 
	\begin{align}
		\Pr{ ({\cO}(\mK_n^t), v_n) \mid {\cO}(\mK_n^t)} \convp \mfL(\hat{\bar{\mO}}^t),
	\end{align}
	with $v_n$ denoting a uniformly selected corner of ${\cO}(\mK_n^t)$. Equation~\eqref{eq:smoland} now allows us to argue analogously as in Corollary~\ref{co:doitf} to deduce the local limit~\eqref{eq:fi122}.

	It remains to show the statements directly concerning  $\mK_n^t$. We may copy the proof of Lemma~\ref{le:convRbar} almost word by word,  replacing all occurrences of $\bar{\cK}$ and $\bar{\cR}$ by $\cK$ and $\cR$,  to deduce analogons to Equation~\eqref{eq:Kattached1} and Inequality~\eqref{eq:Kattached2}. In particular, $\mK_n^t$ is obtained from the $\cR$-core $\cR(\mK_n^t)$ by substituting edges by $\cK$-networks. 
	For any deterministic sequence $t_n \to \infty$ with $t_n = o(n)$ it holds that all but $t_n$ $\cK$-components jointly and asymptotically behave like independent copies of a network $\mK^t$ that follows the Boltzmann distribution for the class $\cK(t,y)$. The total number of corners in the remaining $t_n$ components is with high probability smaller than $t_n$ times a constant that does not depend on $n$. Likewise, an analogon of Equation~\eqref{eq:reallyawesomeKattached} holds, meaning that if we select two corners of $\mK_n^t$ uniformly and independently at random, then their components asymptotically behave like independent copies of a size-biased version $\mK^\bullet$ of $\mK^t$, and the edges corresponding to these components are asymptotically uniform edges of $\cR(\mK_n^t)$. The subsequent arguments in the proof of Lemma~\ref{le:convRbar} that pass convergence from a large random $\bar{\cK}$-structure down to its $\bar{\cR}$-core, also imply that conversely convergence of the $\bar{\cR}$-core implies convergence of the $\bar{\cK}$-structure. As we may copy the proof of Lemma~\ref{le:convRbar} word by word (replacing all occurrences of $\bar{\cK}$ and $\bar{\cR}$ by $\cK$ and $\cR$), this means that the local convergence~\eqref{eq:fi122} implies the local limit~\eqref{eq:locK}.
	
	 It remains to prove the  local convergence~\eqref{eq:locKuniv} with respect to the uniform distribution. We will use  a transfer argument  by \cite{2018arXiv180110007D}, that is based on an extension of a formula by
	 \cite[Eq. (2.3.1)]{MR1666953}. Given an integer $r \ge 1$ and a finite rooted graph $G$, the limit~\eqref{eq:locK} implies that the number $Y_n$ of half-edges in $\mK_n^t$ having $r$-neighbourhood isomorphic to $G$ satisfies
	 \[
	 	Y_n/(2n) \convp \Pr{U_r(\hat{\mK}^t) = G}.
	 \]
	 Letting $d(G)$ denote the root-degree of $G$, the number $X_n$ of vertices in $\mK_n^t$ having $r$-neighbourhood isomorphic to $G$ is given by
	 \[
	 	X_n = Y_n / d(G).
	 \]
	 By the limit~\eqref{eq:veconcentration} below it follows that
	 \[
	 	X_n / \ve(\mK_n^t) = Y_n / (d(G) \ve(\mK_n^t)) \convdis 2 \Pr{U_r(\hat{\mK}^t) = G} / (d(G)  q_1) =: p_{r, G}.
	 \]
	 We may deduce $\sum_G p_{r,G} = 1$ by identical arguments as in the proof of \cite[Eq. (6)]{2018arXiv180110007D}, hence verifying the existence of a random infinite planar graph $\hat{\mK}^{\mathrm{u},t}$ with
	 \begin{align}
	 \Pr{ (\mK_n^t, v_n^\cK) \mid \mK_n^t} \convp \mfL(\hat{\mK}^{\mathrm{u},t}).
	 \end{align}
\end{proof}

\subsection{$2$-connected planar graphs}

Let $\mN_n^t$ denote a random $\cN$-network with $n$ regular edges and weight $t$ at non-pole vertices.
The convergent Gibbs partition \[\cN(t,y) = \cK(t,y) \Seq(t\cK(t,y))\] expresses that any $\cN$-network consists of a series composition of a positive number of $\cK$-networks. This is completely analogous to the fact that $\cD$-networks are series compositions of a positive number of $\cN$-networks. Hence Equations~\eqref{eq:Dtoseq}--\eqref{eq:fin} still hold if we replace $\mD_n^t$ by $\mN_n^t$ and any occurrence of $\bar{\mK}$ and $\bar{\cK}$ by $\mK$ and $\cK$.  That is, identifying $\cN$-networks with sequences of $\cK$-networks, it follows that
\begin{align}
\label{eq:graDtoseq}
\mN_n^t \atv \left({\mK}(1), \ldots, {\mK}(F), {\mK}_{n - E}^t, {\mK}'(1), \ldots, {\mK}'(F')\right), 
\end{align}
with a random integer 
\[
E := \sum_{i=1}^{F} \ed({\mK}(i)) +  \sum_{i=1}^{F'} \ed({\mK}'(i)).
\]
Here we let $F$ and $F'$ denote independent identically distributed geometric variables with distribution
\begin{align}
\Pr{F = k} = {\cK}(t, \rho_{{\cK}})^k ( 1 - {\cK}(t,\rho_{{\cK}})), \qquad k \ge 0.
\end{align}
The networks ${\mK}(i)$ and  ${\mK}'(i)$, $i \ge 1$, denote independent copies of a Boltzmann distributed ${\cK}(t, y)$-network ${\mK}$ with distribution given by
\begin{align}
\label{eq:grafin}
\Pr{{\mK} = {K}} = t^{\ve({K})} \rho_{{\cK}}^{\ed({K})} / {\cK}(t,\rho_{{\cK}}).
\end{align}

The relation $\cN = (1+y)\frac{2}{x^2}\frac{\partial \cB}{\partial y}-1$ entails that the result $\mB_n^t$ of adding an edge between the poles of $\mN^t_n$ (if it isn't already present) is the random $2$-connected planar graph with $n$ edges and weight $t$ at vertices. The enumerative results of \cite{MR1946145} entail that there is a number $0<q_1<1$ such that
\begin{align}
	\ve(\mB_n^t)n^{-1} \convp q_1. 
\end{align}
By Equation~\eqref{eq:graDtoseq}, it readily follows that
\begin{align}
	\label{eq:veconcentration}
	\ve(\mK_n^t) n^{-1} \convp q_1.
\end{align}
This concentration phenomena is used in the proof of Lemma~\ref{le:laststep2} in order to pass from the stationary distribution to the uniform distribution.  
Hence Lemma~\ref{le:laststep2} is now fully verified and having it at hand, it follows from Equation~\eqref{eq:graDtoseq} (by identical arguments as in the proof of Corollary~\ref{co:doitf}) that:
\begin{theorem}
	\label{te:2conlim} 
	Letting $v_n^\cB$ denote a vertex selected according to the uniform distribution $\mu_n^{\cB,t}$ on the vertex set of $\mB_n^t$, it holds that
	\begin{align}
	\label{eq:locB}
	\Pr{ (\mB_n^t, v_n^\cB) \mid \mB_n^t} \convp \mfL(\hat{\mK}^{\mathrm{u},t})
	\end{align}
\end{theorem}
Likewise, Equation~\eqref{eq:locK} implies such a limit when $v_n^{\mathrm{e}}$ is chosen according to the stationary distribution instead. We let $\hat{\mB}$ denote the limit in the case $t=1$.

Letting $\cO(\mB_n^t)$ denote the $\cO$-core of the largest $\cK$-component of the $\cN$-network $\mN_n^t$ (out of which we constructed $\mB_n^t$), it follows by identical arguments as in the proof of Corollary~\ref{co:tillicollapse} that Equation~\eqref{eq:lltOtnb} implies 
\begin{align}
\label{eq:lltBO}
		\Pr{\ed({\cO}({\mB}_n^t)) = \ell} = \frac{1}{g_{{\cK}}(t) n^{2/3}}\left(h\left(\frac{ (1-\Ex{\xi^{{\cK}}})n - \ell}{g_{{\cK}}(t) n^{2/3}}   \right) + o(1)\right).
\end{align}
uniformly for all $\ell \in \ndZ$

\subsection{Bundles of $2$-connected planar graphs}

 Equation~\eqref{eq:blockdecomp} entails that planar graphs (with vertices as atoms) are $\cW$-enriched trees for the class $\cW$ given by
\begin{align}
\label{eq:doub}
\cW(x) = \Set\left( \frac{\partial \cB}{\partial x}(x,1)\right). 
\end{align}
That is, a $\cW$-object is an unordered collection (or bundle) of derived (that is, rooted at a vertex without a label) $2$-connected planar graphs that are glued together at their distinguished vertices. The resulting vertex becomes the root of the $\cW$-object. By  the discussion in  Section~\ref{sec:enriched}, this entails that the random planar graph $\mP_n$ may be generated as follows (see also \cite[Prop. 3.6]{PaStWe2016}):

\begin{enumerate}
	\item Generate a simply generated tree $\mT_{n}^\cP$ with weight sequence $(\omega_k^\cP)_{k \ge 0}$ given by $\omega_k^\cP= [x^k]\cW(x)$.
	\item For each vertex $v \in \mT_n^{\cP}$ let $\beta_n^\cP(v)$ denote a uniformly selected $\cW$-structure with $d_{\mT_n^{\cP}}^+(v)$ labelled non-root vertices.
	\item Assemble $\mP_n$ from the $\cW$-enriched tree $(\mT_{n}^\cP,\beta_n^\cP)$ by applying the correspondence between rooted planar graphs and $\cW$-enriched trees, and forgetting about the root vertex. 
\end{enumerate}

The last step means that we start with the $\cW$-object $\beta^\cP(o)$ of the root $o$ of $\mT_n^\cP$ and identify its non-marked vertices in a canonical way with the offspring vertices of $o$. The graph is then constructed recursively by identifying each non-marked vertex of   $\beta^\cP(o)$ with the rooted graph corresponding to the enriched fringe subtree of $(\mT_{n}^\cP,\beta_n^\cP)$ at the corresponding offspring of $o$.

Inequality~\eqref{eq:Cnu}, the asymptotic expression~\eqref{eq:phiC},  and Lemma~\ref{le:simplygen} entail that $\mT_n^\cP$ follows the distribution of a Galton--Watson tree $\mT^\cP$ conditioned on having $n$ vertices, with offspring law $\xi^\cP$ satisfying
\begin{align}
	\label{eq:tailplanar}
	\Ex{\xi^\cP} = \rho_\cB \frac{\partial^2 \cB}{\partial x^2}(\rho_\cB,1) <1 \qquad \text{and} \qquad \Pr{\xi^\cP = n} \sim c_\cP \rho_{\cB}^{-n} n^{-5/2},
\end{align}
for some constant $c_{\cP}>0$. We let $\cW(\mP_n)$ denote the $\cW$-structure corresponding to the lexicographically first vertex of $\mT_n^\cP$ with maximal outdegree.  By Lemma~\ref{le:maxllt} it follows that the (with high probability unique) largest ${\cW}$-component ${\cW}(\mP_n)$ satisfies
\begin{align}
\label{eq:lltWt}
\Pr{\ve(\cW({\mP_n})) = \ell} = \frac{1}{g_\cP n^{2/3}}  \left(h\left(\frac{ (1-\Ex{\xi^{\cP}})n - \ell}{g_{\cP} n^{2/3}}   \right) + o(1)\right)
\end{align}
uniformly for all $\ell \in \ndZ$ with $g_\cP >0$ a constant. The second largest $\cW$-component has order $O_p(n^{2/3})$, this follows for example from \cite[Thm. 19.34]{MR2908619}. 
The Gibbs partition~\eqref{eq:doub} is convergent, that is $\cW(\mP_n)$ exhibits a giant $\cB$-component denoted by $\cB(\cW(\mP_n))$, and the small fragments admit a limit distribution. Hence the $\cB$-core $\cB(\mP_n) := \cB(\cW(\mP_n))$ corresponds with high probability to the largest $2$-connected block of $\mP_n$.

\begin{remark}
	By identical arguments as in the proof of Corollary~\ref{co:tillicollapse}, Equation~\eqref{eq:lltWt} implies that uniformly for all $\ell \in \ndZ$
	\begin{align}
	\label{eq:lltBt}
	\Pr{\ve(\cB({\mP_n})) = \ell} = \frac{1}{g_\cP n^{2/3}}  \left(h\left(\frac{ (1-\Ex{\xi^{\cP}})n - \ell}{g_{\cP} n^{2/3}}   \right) + o(1)\right).
	\end{align}
A local limit law for the number of vertices ${L}_n$ of the largest block in $\mP_n$ was proven by \cite[Thm. 5.4]{MR3068033}. Note that Equation~\eqref{eq:lltBt} is a slightly different  statement. Clearly $L_n \atv \ve(\cB(\mP_n))$, but in order to deduce a local limit theorem for ${L}_n$ we would additionally have to verify that the probability for the event, that simultaneously  $\ve(\cB({\mP_n})) = \ell$ and  ${L}_n~>~\ve(\cB(\mP_n))$, lies in $o(n^{-2/3})$ uniformly for all $\ell$. The proof is similar to arguments used in the proof of \cite[Thm. 1.1]{2019arXiv190104603S}, specifically  the step that shows that the bound in \cite[Eq.~(3.23)]{2019arXiv190104603S} tends to zero. We leave the details to the reader, since this subtle difference between $\cB(\mP_n)$ and the largest $2$-connected block is not relevant for the arguments in the present work.
\end{remark}

Note that conditioning the core $\cB(\mP_n)$ on having a certain number of edges does not yield the uniform distribution on the $2$-connected planar graphs with that number of edges. This effect does \emph{not} go aways as $n$ becomes large. In fact, letting $E_n$ denote the number of edges in the largest $2$-connected block of $\mP_n$, it was shown by \cite[Lem. 6.6]{MR3068033} that
\begin{align}
\label{eq:tilt}
\cB(\mP_n) \atv \mB_{E_n}^{\rho_\cB}
\end{align}
as $n$ tends to infinity.  That is, we have to introduce weight $t= \rho_\cB$ at vertices. Here we assume $E_n$ to be independent from $(\mB_k^{\rho_\cB})_{k \ge 0}$. \cite[Thm. 6.5]{MR3068033} showed that the number $E_n$ has order $\alpha_0 n$ with an analytically given constant
\begin{align}
	\label{eq:constalpha0}
	\alpha_0 \approx 2.17
\end{align}
and a fluctuation of order~$n^{2/3}$ that admits a local limit theorem of Airy type. Letting $v_n$ denote a uniformly selected vertex of $\cB(\mP_n)$, it follows from~\eqref{eq:tilt} and Theorem~\ref{te:2conlim} that
\begin{align}
\label{eq:coreloca1}
\Pr{ (\cB(\mP_n), v_n) \mid \cB(\mP_n)} \convp \mfL(\hat{\mB}).
\end{align}
 As the Gibbs partition~\eqref{eq:doub} is convergent,  it follows from \eqref{eq:coreloca1} 
\begin{align}
\label{eq:Wlocal}
\Pr{ (\cW(\mP_n), v_n) \mid \cB(\mP_n)} \convp \mfL(\hat{\mB}).
\end{align}

\subsection{Connected planar graphs}

A result of \cite[Thm. 6.39]{2016arXiv161202580S} states that for block-weighted random graphs in a certain condensation regime (encompassing $\mP_n$)  annealed local convergence of the random connected graph is equivalent to annealed local convergence of its $2$-connected core.  By Equation~\eqref{eq:coreloca1} Theorem~\ref{te:2conlim} it holds that the $2$-connected core $\cB(\mP_n)$ of the uniform connected planar graph $\mP_n$ with $n$ labelled vertices admits a distributional limit $\hat{\mB}$ in the local topology. Hence $\mP_n$ admits an annealed local limit~$\hat{\mP}$. As stated in \cite[Thm. 6.39]{2016arXiv161202580S}, it also follows that the UIPG $\hat{\mP}$ may be constructed from the uniform infinite $2$-connected planar graph $\hat{\mB}$ by inserting an independent copy of a Boltzmann distributed   rooted connected planar graph  at each non-root vertex of the uniform infinite random $2$-connected graph~$\hat{\mB}$, and a Boltzmann distributed doubly rooted connected planar graph at the root of $\hat{\mB}$. 

We are now going to prove quenched convergence of $\mP_n$.

\begin{proof}[Proof of Theorem~\ref{te:mainfinal}]
	The graph $\mP_n$ consists of its $\cW$-core $\cW(\mP_n)$ together with planar graphs $(\cP_i(\mP_n))_{1 \le i \le \ve(\cW(\mP_n))}$ attached to each of its vertices. We assume that the case $i=\ve(\cW(\mP_n)$ corresponds to the component attached to the root of $\cW(\mP_n)$. For all $i \ge 1$ we let $\mP(i)$ denote an independent copy of a Boltzmann distributed vertex-rooted connected planar graph~$\mP$. Note that $\mP$ corresponds to the canonical decoration of a $\xi^\cP$-Galton--Watson tree. Hence by Lemma~\ref{le:fringe}, it follows that there is a constant $C>0$ such that for any sequence of integers $(t_n)_n$ with $t_n\to \infty$ and $t_n=o(n)$ it holds that
	\begin{align}
		\label{eq:at1}
	(\cP_i(\mP_n))_{1 \le i \le \ve(\cW(\mP_n)) - t_n} \atv (\mP(i))_{1 \le i \le \ve(\cW(\mP_n)) - t_n}.
	\end{align}
	and with high probability
	\begin{align}
		\label{eq:at2}
		\sum_{i=\ve(\cW(\mP_n)) - t_n}^{\ve(\cW(\mP_n))} \ve(\cP_i(\mP_n)) \le C t_n.
	\end{align}

	We select two vertices $v_1$ and $v_2$ of $\mP_n$ uniformly and independently at random. We refer to $v_1$ as the red vertex, and $v_2$ as the blue vertex. The vertices of $\mP_n$ correspond bijectively to the vertices of $(\cP_i(\mP_n))_{1 \le i \le \ve(\cW(\mP_n))}$. (The edges of~$\mP_n$ correspond bijectively to the edges of $(\cP_i(\mP_n))_{1 \le i \le \ve(\cW(\mP_n))}$ plus the edges of $\cW(\mP_n)$.) For all $1 \le i \le \ve(\cW(\mP_n))$ we let $\bar{\cP}_i(\mP_n)$ denote the vertex-rooted connected planar graph $\cP_i(\mP_n)$ with the additional information if and where it contains a marked red or blue vertex. We let $\mP_1^\bullet$ and $\mP_2^\bullet$ denote independent copies of a Boltzmann distributed doubly vertex-rooted planar graph. We colour the second root of $\mP_1^\bullet$ red and the second root of $\mP_2^\bullet$ blue. We let $j_1$ and $j_2$ denote a pair of uniformly selected distinct integers between $1$ and $\ve(\cW(\mP_n))-t_n$. For each $1 \le i \le \ve(\cW(\mP_n))$ we set $\bar{\mP}(i) = \mP(i)$ if $i \ne j_1$ and $i \ne j_2$. If $i = j_k$ (for $k=1$ or $k=2$) we set $\bar{\mP}(i) = \mP^\bullet_k$. It follows by Corollary~\ref{co:sizebias} that
	\begin{align}
		\label{eq:pwithdot}
		(\bar{\cP}_i(\mP_n))_{1 \le i \le \ve(\cW(\mP_n)) - t_n} \atv (\bar{\mP}(i))_{1 \le i \le \ve(\cW(\mP_n)) - t_n}.
	\end{align}

	We let $r \ge 0$ denote fixed arbitrary integers.  By Proposition~\ref{pro:sparse} and the local convergence~\ref{eq:Wlocal} it follows that the neighbourhoods $U_{r}(\mP_n, v_1)$ and $U_{r}(\mP_n, v_2)$ are with high probability disjoint. Applying Proposition~\ref{pro:sparse} repeatedly also entails that we may choose the sequence $(t_n)_n$ to converge sufficiently slowly to infinity such that with high probability neither of these neighbourhoods contains any of the  last $t_n$ vertices of $\cW(\mP_n)$. It follows that with high probability the union $U_{r}(\mP_n, v_1) \cup U_{r}(\mP_n, v_2)$ does not intersect with $\cP_i(\mP_n)$ for any $\ve(\cW(\mP_n))-t_n+1 \le i \le \ve(\cW(\mP_N))$.
	
 	For $k\in\{1,2\}$ we let $v_k'$  denote the vertex of $\cW(\mP_n)$ corresponding to the component containing $v_k$. If $d_{\mP_n}(v_k, v_k') \ge r$ then $U_{r}(\mP_n, v_k)$ is fully contained in the component $\cP$-component containing $v_k$. If the distance equals some $h< r$, then $U_{r}(\mP_n, v_k)$ is glued together from the $r$-neighbourhood of $v_k$ in that component (with additional information on the location of $v_k'$ within that neighbourhood), the neighbourhood $U_{r -h}(\cW(\mP_n), v_k')$, and  neighbourhoods in the $\cP$-components corresponding to vertices from $U_{r -h-1}(\cW(\mP_n), v_k') \setminus \{v_k'\}$.

	Equation~\eqref{eq:pwithdot}, the local convergence~\eqref{eq:Wlocal} and the observations made in the penultimate paragraph  entail that asymptotically and jointly the components corresponding to $v_1'$ and $v_2'$ behave like independent copies of $\mP^\bullet$, and the components corresponding to vertices from $U_{r}(\cW(\mP_n), v_1') \setminus \{v_1'\}$ and $U_{r}(\cW(\mP_n), v_2') \setminus \{v_2'\}$ behave like independent copies of $\mP$, and the neighbourhoods $U_{r}(\cW(\mP_n), v_1')$ and $U_{r}(\cW(\mP_n), v_1')$ behave like independent copies of the neighbourhood $U_r(\hat{\mB})$. It follows that the pair of neighbourhoods $(U_{r_1}(\mP_n, v_1), U_{r_2}(\mP_n, v_2))$ converges in distribution to a pair of independent copies of a certain random rooted graph with radius $r$. As this is true for arbitrary $r$, it follows by Proposition~\ref{prop:char} that there exists an infinite random rooted  planar graph $\hat{\mP}$ with
	\begin{align}
	\Pr{ (\mP_n, v_n) \mid \mP_n} \convp \mfL(\hat{\mP}).
	\end{align}
\end{proof}

The number of edges $\ed(\mP_n)$ is known to satisfy a normal central limit theorem, see \cite[Thm. 4.1]{MR3068033}. Arguing analogously as in the proof of \cite[Thm. 2.1]{2018arXiv180110007D}, it follows that the convergence of Theorem~\ref{te:mainfinal} also entails a local limit (following a \emph{different} distribution) for $\mP_n$ marked according to the stationary distribution.   We close the Section with the following remark on the small blocks in $\mP_n$.

\begin{remark}
	The Gibbs partition~\eqref{eq:doub} is convergent. This entails that the  $\cW$-core $\cW(\mP_n)$ consists of a giant $2$-connected component and a remainder that asymptotically behaves like a Boltzmann distributed $\cW$-object, that is a $\mathrm{Poisson}(\cB(\rho_\cB,1))$ number of independent copies of Boltzmann distributed $2$-connected $\cB(x,1)$-graph $\mB$. It follows that the collection $\mathrm{frag}(\mP_n)$ of all blocks with non-maximal size satisfies
	\begin{align}
		\mathrm{frag}(\mP_n) \atv (\mB(i))_{1 \le i \le N_n}
	\end{align}
	with $(\mB(i))_{i \ge 1}$ denoting independent copies of $\mB$, and $N$ an independent random integer with a Poisson distribution 
	\begin{align}
		N_n \eqdist \mathrm{Poisson}(n\cB(\rho_\cB,1)).
	\end{align}
\end{remark}

\bibliographystyle{alea3}
\bibliography{planar}

\end{document}